\documentclass{article}

\RequirePackage[OT1]{fontenc}
\RequirePackage[
amsthm,amsmath
]{imsart}

\usepackage{graphicx}
\usepackage{latexsym,amsmath}
\usepackage{amsmath,amsthm,amscd}
\usepackage{amsfonts}
\usepackage[psamsfonts]{amssymb}
\usepackage{enumerate}

\usepackage{url}

\startlocaldefs
\numberwithin{equation}{section}

\theoremstyle{plain} 
\newtheorem{theorem}{Theorem}[section]
\newtheorem{corollary}[theorem]{Corollary}

\newtheorem{lemma}[theorem]{Lemma}
\newtheorem{proposition}[theorem]{Proposition}

\theoremstyle{definition} 
\newtheorem{definition}[theorem]{Definition}
\theoremstyle{definition} 

\newtheorem*{ex*}{Example}
\theoremstyle{remark} 

\theoremstyle{remark} 
\newtheorem{remark}[theorem]{Remark}
\newtheorem*{remark*}{Remark}
\numberwithin{equation}{section}
 
\newcommand{\beqa}{\begin{eqnarray}}
\newcommand{\eeqa}{\end{eqnarray}}
 
\newcommand{\bseq}{\begin{subequations}}
 
\newcommand{\eseq}{\end{subequations}}

\newcommand{\N}{\{1,2,\dots\}}
\newcommand{\dd}{\partial}

\renewcommand{\dd}{{\,\operatorname{d}}}
\newcommand{\res}[1]{\underset{#1}{\operatorname{Res}}}
\newcommand{\supp}{\operatorname{supp}}

\newcommand{\BH}{\operatorname{BH}}
\newcommand{\PU}{\operatorname{PU}}
\newcommand{\Ea}{\operatorname{Ea}}
\newcommand{\Be}{\operatorname{Be}}
\newcommand{\Pin}{\operatorname{Pin}}
\newcommand{\Ne}{\operatorname{EN}}
\newcommand{\Ca}{\operatorname{Ca}}
\newcommand{\EN}{\operatorname{EN}}

\newcommand{\al}{\alpha}

\newcommand{\Ga}{\Gamma}
\newcommand{\si}{\sigma}
\newcommand{\ka}{\kappa}
\newcommand{\la}{\lambda}
\newcommand{\de}{\delta}
\newcommand{\be}{\beta}
\newcommand{\De}{\Delta}

\renewcommand{\th}{\theta}
\renewcommand{\Psi}{\overline{\Phi}}

\newcommand{\ffrown}{\text{\raisebox{3pt}[0pt][0pt]{$\frown$}}}

\renewcommand{\O}{\underset{\ffrown}{<}}
\newcommand{\OG}{\underset{\ffrown}{>}}
\newcommand{\OO}{\mathrel{\text{\raisebox{2pt}{$\O$}}}}
\newcommand{\OOG}{\mathrel{\text{\raisebox{2pt}{$\OG$}}}}

\newcommand{\ii}[1]{\,\mathbf{I}\{#1\}}

\newcommand{\PP}{\operatorname{\mathsf{P}}} 
\newcommand{\E}{\operatorname{\mathsf{E}}}

\newcommand{\lc}{\mathsf{L\!C}}

\newcommand{\Z}{\mathbb{Z}}
\newcommand{\R}{\mathbb{R}}
\newcommand{\C}{\mathcal{C}}
\newcommand{\F}{\mathcal{F}}
\newcommand{\EE}{\mathcal{E}}

\newcommand{\G}{\mathcal{F}}
\renewcommand{\H}[1]{\mathcal{H}_+^{#1}}
\renewcommand{\F}[1]{\mathcal{F}_+^{#1}}

\newcommand{\vp}{\varepsilon}

\newcommand{\ta}{{\tilde{a}}}
\newcommand{\tb}{{\tilde{b}}}
\newcommand{\tp}{{\tilde{p}}}
\newcommand{\tq}{{\tilde{q}}}
\newcommand{\tPi}{{\tilde{\Pi}}}

\newcommand{\tG}{{\tilde{G}}}
\newcommand{\tvp}{{\tilde{\vp}}}
\newcommand{\tbe}{{\tilde{\beta}}}
\newcommand{\tsi}{{\tilde{\si}}}
\newcommand{\tS}{{\tilde{S}}}

\renewcommand{\le}{\leqslant}
\renewcommand{\ge}{\geqslant}

\newcommand{\No}{\operatorname{N}}
\newcommand{\Po}{\operatorname{Pois}}

\newcommand{\fl}[1]{\lfloor#1\rfloor}
\newcommand{\ce}[1]{\lceil#1\rceil}

\renewcommand{\Re}{\operatorname{\mathsf{Re}}}
\renewcommand{\Re}{\operatorname{\mathfrak{Re}}}

\endlocaldefs

\begin{document}

\begin{frontmatter}

\title{On the Bennett-Hoeffding inequality}
\runtitle{On the Bennett-Hoeffding bound}

\begin{aug}
\author{\fnms{Iosif} \snm{Pinelis}\ead[label=e1]{ipinelis@math.mtu.edu}}
\runauthor{Iosif Pinelis}


\address{Department of Mathematical Sciences\\
Michigan Technological University\\
Houghton, Michigan 49931, USA\\
E-mail: \printead[ipinelis@math.mtu.edu]{e1}}
\end{aug}

\begin{abstract}
The well-known Bennett-Hoeffding bound for sums of independent random variables is refined, by taking into account truncated third moments, and at that also improved by using, instead of the class of all increasing exponential functions,  the much larger class of all generalized moment functions $f$ such that $f$ and $f''$ are increasing and convex. It is shown that the resulting bounds have certain optimality properties. Comparisons with related known bounds are given. The results can be extended in a standard manner to (the maximal functions of) (super)martingales.  
\end{abstract}



\end{frontmatter}


{\small\tableofcontents} 



\theoremstyle{plain} 
\numberwithin{equation}{section}


\eject

\section{Introduction}\label{intro}


Let $X_1,\dots,X_n$ be independent real-valued zero-mean random variables (r.v.'s) such that $X_i\le y$ almost surely (a.s.) for some $y>0$ and all $i$. Let $S:=X_1+\dots+X_n$ and assume that $\si:=\sqrt{\sum_i\E X_i^2}\in(0,\infty)$. The Bennett-Hoeffding \cite{bennett,hoeff} inequality states that 
\begin{equation}\label{eq:hoeff}
	\PP(S\ge x)\le\BH(x):=\BH_{\si^2,y}(x):=\exp\Big\{-\frac{\si^2}{y^2}\,\psi\Big(\frac{xy}{\si^2}\Big)\Big\}
\end{equation}
for all $x\ge0$, 
where 
\begin{equation}\label{eq:psi}
	\psi(u):=(1+u)\ln(1+u)-u; 
\end{equation}
see e.g.\ \cite{bennett} concerning the importance of such bounds. 
Inequality \eqref{eq:hoeff} has been generalized to include cases when the $X_i's$ are not independent and/or are not real-valued; see e.g.\ 
\cite{bouch-lug-mas,bousquet02,bousquet03,cohen,de la pena,dzhap,freedman,janson,klein-priv,klein-rio,massart,pin-sakh,pin92,ann_prob,rous,shorack-wellner,talag}.

Inequality \eqref{eq:hoeff} is based on an upper bound on the exponential moments of $S$:
\begin{equation}\label{eq:BHexp}
	\E e^{\la S}\le\BH_{\exp}(\la):=\exp\Big\{\frac{e^{\la y}-1-\la y}{y^2}\,\si^2\Big\}\quad\text{for all $\la>0$;}
\end{equation}
that is, 
\begin{equation}\label{eq:BH}
	\BH(x)=\inf_{\la>0}e^{-\la x}\BH_{\exp}(\la). 
\end{equation}

Attempts at refining the Bennett-Hoeffding inequality by taking moments higher than the second ones into consideration were made in \cite{fuk-nagaev,fuk,nag,vol}; however, in contrast with the Bennett-Hoeffding bounds, the bounds given in \cite{fuk-nagaev,fuk,nag,vol} were not the best possible in their own terms. Such best possible, exact bounds refining the Bennett-Hoeffding ones were obtained by Pinelis and Utev \cite[Theorems~2 and 6]{pin-utev2}. In particular, \cite[Theorem~2]{pin-utev2} implies that 
\begin{equation}\label{eq:PUexp}
	\E e^{\la S}\le\PU_{\exp}(\la):=\exp\Big\{\frac{\la^2}2\,(1-\vp)\si^2+\frac{e^{\la y}-1-\la y}{y^2}\,\vp\si^2\Big\}\quad
	\forall\la>0,
\end{equation}
where 
\begin{equation*}
	\vp:=\frac{\be^+_3}{\si^2y}, \quad 
	\be^+_3:=\sum_i\E(X_i)_+^3,
\end{equation*}
$x_+:=0\vee x=\max(0,x)$ and $x_+^\al:=(x_+)^\al$, 
whence for all $x\ge0$
\begin{align}
	\PP(S\ge x)\le\PU(x)&:=\inf_{\la>0}e^{-\la x}\PU_{\exp}(\la). \label{eq:PU}
\end{align} 


Note that $\vp\in(0,1)$. Hence and because $\frac{\la^2}2<\frac{e^{\la y}-1-\la y}{y^2}$ for all $\la>0$ and $y>0$, the Pinelis-Utev upper bounds $\PU_{\exp}(\la)$ and $\PU(x)$ are always less than the Bennett-Hoeffding upper bounds $\BH_{\exp}(\la)$ and $\PU(x)$, respectively. 
Moreover, the $\PU$ bounds may be significantly less than the $\BH$ ones; this happens when $\vp$ is much less than $1$, which in particular is the case when $X_1,\dots,X_n$ form the initial segment of an infinite sequence of i.i.d.\ r.v.'s $X_1,X_2,\dots$ with finite $\E X_1^2$ and $\E(X_1)_+^3$, $n$ is large, and $y$ is of the order of $\sqrt n$ (such a situation occurs in proofs of non-uniform Berry-Esseen type bounds). 

Note also that the mentioned Theorem~2 in \cite{pin-utev2} is 
formally more general than inequality \eqref{eq:PUexp}, in that \cite[Theorem~2]{pin-utev2} is given in terms of $\be^+_p:=\sum_i\E (X_i)_+^p$ for any $p\in[2,3]$, rather than $\be^+_3$. However, the exact upper bound in \cite[Theorem~2]{pin-utev2} on $\E e^{\la S}$ with $p\in[2,3]$ is no less than that with $p=3$, since $\be^+_3\le\be^+_p\,y^{3-p}$ for all $p\in[2,3]$. Thus, nothing will be lost by taking $p$ to be just $3$. 

As pointed out in \cite{hoeff,pin-utev2}, the exponential bounds $\BH_{\exp}(\la)$ and $\PU_{\exp}(\la)$ are each exact in its own terms. That is, $\BH_{\exp}(\la)$ is the exact upper bound on $\E e^{\la S}$ with $\la$, $y$, and $\si$ fixed; and $\PU_{\exp}(\la)$ is the exact upper bound on $\E e^{\la S}$ with $\la$, $y$, $\si$, and $\vp$ fixed. 

If $\vp$ is small indeed, then the bounds $\PU_{\exp}(\la)$ and $\PU(x)$ are close to the corresponding exponential bounds for the normal distribution, $e^{\la^2\si^2/2}$ and $e^{-x^2/(2\si^2)}$. 
However, even for a standard normal r.v.\ $Z$, the best exponential upper bound, $e^{-x^2/2}$, on the tail probability $\PP(Z\ge x)$ is ``missing'' a factor of the order of $1/x$ for large $x>0$, since $\PP(Z\ge x)\sim\frac1{x\sqrt{2\pi}}e^{-x^2/2}$ as $x\to\infty$. 
This deficiency of exponential bounds is caused by the fact that the class of all increasing exponential functions is too small. 

Apparently the first step towards removing this deficiency was made by Eaton~\cite{eaton1,eaton2}, who proved that for all functions $f$ in a rich class $\G_{\textsf{Ea}}$ containing all functions of the form $\R\ni x\mapsto(|x|-t)_+^3$ for $t>0$ one has
\begin{equation}\label{eq:EaF}
	\E f(S)\le\E f(Z)
\end{equation}
if $X_i=a_i\eta_i$ for all $i$, where the $\eta_i$'s are independent (not necessarily identically distributed) zero-mean r.v.'s such
that $|\eta_i|\le1$ a.s.\ for all $i$, and $a_1^2+\dots+a_n^2=1$. 
It is easy to see that inequality \eqref{eq:EaF} for all $f$ in the Eaton class $\G_{\textsf{Ea}}$ implies the same inequality for all symmetrized exponential functions of the form $\R\ni x\mapsto\cosh\la x$, with any $\la\in\R$. 
In view of the central limit theorem, it is clear that the upper bound $\E f(Z)$ in \eqref{eq:EaF} on $\E f(S)$ is exact for each $f\in\G_{\textsf{Ea}}$. 
Moreover, then the inequality
\begin{equation}\label{eq:ea}
	\PP(|S|\ge x)\le\Ea(x):=\inf_{t\in(0,x)}\E(|Z|-t)_+^3/(x-t)^3
\end{equation}
for $x>0$ provides the best possible upper bound $\Ea(x)$ on $\PP(|S|\ge x)$ based on comparison inequality \eqref{eq:EaF}. 
Eaton showed that the bound $\Ea(x)$ is majorized by a function which is asymptotic to $c_3\PP(|Z|\ge x)$ as $x\to\infty$, where 
$c_3:=\frac{2e^3}9\approx4.46$. Thus, the ``missing'' factor of the order of $1/x$ was restored, for the bounded $X_i$'s. Tables for the bound $\Ea$ and related bounds were given in \cite{dufour-hallin}. 
Eaton ~\cite{eaton2} also conjectured that $\PP(|S|\ge x)\le2c_3\frac1x\,e^{-x^2/2}/\sqrt{2\pi}$ for $x>\sqrt2$. 
The stronger form of this conjecture, 
\begin{equation}\label{eq:pin94}
	\PP(S\ge x)\le c\PP(Z\ge x)
\end{equation}
for all $x\in\R$ with $c=c_3$
was proved by Pinelis~\cite{pin94}, along with multidimensional extensions and applications to the Hotelling-type tests. 
\big(More exactly, in \cite{pin94} a two-tail version of inequality \eqref{eq:pin94} was given. The right-tail inequality \eqref{eq:pin94} can be proved quite similarly; alternatively, it follows from general results of \cite{pin98}.\big) 
Various generalizations and improvements of inequality \eqref{eq:pin94} as well as related results were given by Pinelis \cite{pin98,pin99,pin-eaton,binom,normal,asymm}
and Bentkus \cite{bent-liet02,bent-jtp,bent-ap,bent-skew} (with co-authors). 
For Rademacher $\eta_i$'s, a version of \eqref{eq:pin94} with a better constant factor $c$, which is about 1\% off the best possible one, was given in \cite{pin-houdre}; related inequalities were obtained in \cite{BGH,pin-MIA}.

Pinelis \cite{pin98} provided a general device allowing one to extract the optimal tail comparison inequality from a generalized moment comparison. To state that result, consider the Eaton-type classes of functions $f\colon\R\to\R$:
\begin{equation}\label{eq:H}
\H\al:=\{f\colon
f(u)=\textstyle{\int\nolimits_{-\infty}^\infty} (u-t)_+^\al\,\mu(dt)\quad \forall u\in\R\},\quad\al\ge0, 
\end{equation}
where $\mu\ge0$ is a Borel measure, and 
$0^0:=0$; of course, when used with functions or classes of functions (as, for example, in the symbol $\H\al$), the subscript ${}_+$ will have a meaning different from that in the definition $x_+:=0\vee x$. 

It is easy to see \cite[Proposition 1(ii)]{pin99} that
\begin{equation}\label{eq:F-al-beta}
0\le\beta<\al\quad\text{implies}\quad\H\al\subseteq\H\beta.
\end{equation}

\begin{proposition}\label{prop:F-al}
\emph{\cite{binom}}\ 
For natural $\al$, one has $f\in\H\al$ if and only if $f$ has finite derivatives $f^{(0)}:=f,f^{(1)}:=f',\dots,f^{(\al-1)}$ on $\R$ such that $f^{(\al-1)}$ is convex on $\R$ and $f^{(j)}(-\infty+)=0$ for $j=0,1,\dots,\al-1$. 
\end{proposition}

It follows from \eqref{eq:F-al-beta} and Proposition \ref{prop:F-al} that, for every $t\in\R$, every $\al\ge0$, every $\beta\ge\al$, and every $\la>0$, the functions $u\mapsto(u-t)_+^\beta$ and $u\mapsto e^{\la(u-t)}$
belong to $\H\al$.  

The next theorem follows immediately from results of \cite{pin98,pin99}; in particular, see \cite[Theorem~3.11]{pin98} (and its proof) and \cite[Theorem~4]{pin99}.
\begin{theorem}
\label{th:comparison} 
Suppose that $0\le\beta\le\al$, $\xi$ and $\eta$ are real-valued r.v.'s, and the tail function $u\mapsto\PP(\eta\ge u)$ is log-concave on $\R$. Then the comparison inequality 
\begin{equation}\label{eq:comp-al}
\E f(\xi)\le\E f(\eta)\quad\text{for all }f\in\H\al
\end{equation}
implies
\begin{equation}\label{eq:comp-beta}
\E f(\xi)\le c_{\al,\beta}\,\E f(\eta)\quad\text{for all }f\in\H\beta
\end{equation}
and, in particular, for all real $x$,
\begin{align}
\PP(\xi\ge x) \le 
P_\al(\eta;x)&:=\inf_{t\in(-\infty,x)}\,\frac{\E(\eta-t)_+^\al}{(x-t)^\al} 
\label{eq:comp-prob2} \\
&\le c_{\al,0}\,\PP(\eta\ge x),
\label{eq:comp-prob3} 
\end{align}
where 
\begin{equation*}
c_{\al,\beta}:=\frac{\Gamma(\al+1)(e/\al)^\al}{\Gamma(\beta+1)(e/\beta)^\beta}
\end{equation*}
for $\beta>0$; $c_{\al,0}:=\Gamma(\al+1)(e/\al)^\al$.
Moreover, the constant $c_{\al,\beta}$ is the best possible in \eqref{eq:comp-beta} and \eqref{eq:comp-prob3}. 
\end{theorem}

A similar result for the case when $\al=1$ and $\beta=0$ 
is contained in the book by Shorack and Wellner (1986) \cite{shorack-wellner}, pages 797--799. 


\begin{definition}\label{def:lc}
For any r.v.\ $\eta$, let the function $\R\ni x\mapsto\PP^\lc(\eta\ge x)$ be defined as the least log-concave majorant over $\R$ of the tail function $\R\ni x\mapsto\PP(\eta\ge x)$ of the r.v.\ $\eta$. 
\end{definition}

\begin{remark}\label{rem:PLC} One has $\PP^\lc(a+b\eta\ge x)=\PP^\lc(\eta\ge\frac{x-a}b)$ for all $x\in\R$ and all real constants $a$ and $b$ such that $b>0$. 
\end{remark}

\begin{remark}\label{comparison-remark}
As follows from \cite[Remark 3.13]{pin98}, a useful point is that the requirement of the log-concavity of the tail function $\R\ni u\mapsto\PP(\eta\ge u)$ in Theorem \ref{th:comparison}
can be removed --- at least as far as \eqref{eq:comp-prob3} is concerned --- by replacing $\PP(\eta\ge x)$ in \eqref{eq:comp-prob3} with $\PP^\lc(\eta\ge x)$. 
However, then the optimality of $c_{\al,\beta}$ is not guaranteed. 
\end{remark}

Detailed studies of various cases and aspects of the 
bound $P_\al(\eta;x)$ defined in \eqref{eq:comp-prob2} were presented in \cite{dufour-hallin,pin98,bent-64pp}.

Note that 
$c_{3,0}=c_3=2e^3/9$, 
which is the constant factor mentioned above, after inequality \eqref{eq:ea}. 

Going back to the Bennett-Hoefding and Pinelis-Utev bounds defined in \eqref{eq:BHexp} and \eqref{eq:PUexp}, observe that they have a transparent probabilistic interpretation:  
\begin{align}
	\BH_{\exp}(\la)&=\E\exp\big\{\la y\tPi_{\si^2/y^2}\big\} \quad\text{and} \label{eq:BHexp-expr} \\
	\PU_{\exp}(\la)&=\E\exp\big\{\la\big(\Ga_{(1-\vp)\si^2}+y\tPi_{\vp\si^2/y^2}\big)\big\} \label{eq:PUexp-expr}
\end{align}
for all $\la$, where the following definition is employed. 

\begin{definition}\label{def:Ga,tPi}
For any $a>0$ and $\th>0$, let  
$\Ga_{a^2}$ and $\Pi_{\th}$ stand for any independent r.v.'s such that 
\begin{equation*}
\text{$\Ga_{a^2}\sim\No(0,a^2)$ and $\Pi_{\th}\sim\Po(\th)$};	
\end{equation*}
that is, $\Ga_{a^2}$ has the normal distribution with parameters $0$ and $a^2$, and $\Pi_{\th}$
has the Poisson distribution with parameter $\th$; at that, let $\Ga_0$ and $\Pi_0$ be defined as the constant zero r.v. 
Let also 
\begin{equation*}
	\tPi_{\th}:=\Pi_{\th}-\E\Pi_{\th}=\Pi_{\th}-\th.
\end{equation*} 
\end{definition}

Thus, \eqref{eq:BHexp} and \eqref{eq:PUexp} can be viewed as the generalized moment comparison inequalities 
\begin{align}
	\E f(S)&\le\E f(y\tPi_{\si^2/y^2}) \quad\text{and}\label{eq:Bef}\\
	\E f(S)&\le\E f\big(\Ga_{(1-\vp)\si^2}+y\tPi_{\vp\si^2/y^2}\big),\label{eq:PUf}
\end{align}
over the class of all increasing exponential functions $\R\ni x\mapsto f(x)=e^{\la x}$, $\la>0$. 
Note that, of the total variance $\si^2$ of the r.v.\ $\Ga_{(1-\vp)\si^2}+y\tPi_{\vp\si^2/y^2}$ in \eqref{eq:PUf}, the part of the variance equal $(1-\vp)\si^2$ is apportioned to the light-tail centered-Gaussian component $\Ga_{(1-\vp)\si^2}$, while the rest of the variance, $\vp\si^2$, is apportioned to the heavy-tail centered-Poisson component $y\tPi_{\vp\si^2/y^2}$.  

Bentkus~\cite{bent-liet02,bent-ap} extended inequality \eqref{eq:Bef} to all $f$ of the form $f(x)\equiv(x-t)_+^2$; hence, recalling \eqref{eq:H}, one has \eqref{eq:Bef} for all $f\in\H2$. Moreover, it follows by \eqref{eq:comp-prob2}, \eqref{eq:comp-prob3}, and Remark~\ref{comparison-remark} that for all $x\ge0$ 
\begin{equation}\label{eq:Be}
	\PP(S\ge x)\le\Be(x):=P_2(y\tPi_{\si^2/y^2};x)\le c_{2,0}\,\PP^\lc(y\tPi_{\si^2/y^2}\ge x);
\end{equation}
note also that $c_{2,0}=e^2/2$. 
Similar results for stochastic integrals were obtained in \cite{klein-priv}. 
Since the class $\H2$ contains all increasing exponential functions, the Bentkus bound $\Be(x)$ is an improvement of the Bennett-Hoeffding bound $\BH(x)$ given by \eqref{eq:hoeff}. 


In this paper, we shall similarly improve the Pinelis-Utev exponential bounds given by \eqref{eq:PUexp} and \eqref{eq:PU}, which, as was mentioned, in turn refine and improve the corresponding Bennett-Hoeffding bounds. This will require proofs of a significantly higher level of difficulty, with some substantially new ideas.  

\section{Statements of the main results}\label{results}
We shall show that the generalized moment comparison inequality \eqref{eq:PUf} takes place for all $f$ in $\H3$ and, in fact, for all $f$ in the slightly larger class 
\begin{align} \notag
\F3&:=\{f\in\C^2\colon \text{$f$ and $f''$ are nondecreasing and convex}\} \notag\\
&=\{f\in\C^2\colon \text{$f$, $f'$, $f''$, $f'''$ are nondecreasing}\}, \label{eq:F3}
\end{align} 
where $\C^2$ denotes the class of all twice continuously differentiable functions 
$f\colon\R\to\R$ and $f'''$ denotes the right derivative of the convex function $f''$. 
For example, functions $x\mapsto a+b\,x+c\,(x-t)_+^\al$ and 
$x\mapsto a+b\,x+c\,e^{\la x}$ belong to $\F3$ for all $a\in\R$, $b\ge0$, $c\ge0$, $t\in\R$, $\al\ge3$, and $\la\ge0$. 
It is easy to see that $\H3\subseteq\F3$.

\begin{remark*}
If a function $f\colon\R\to\R$ is convex and a r.v.\ $X$ has a finite expectation, then, by Jensen's inequality, $\E f(X)$ always exists in $(-\infty,\infty]$. This remark will be used in this paper (sometimes tacitly) for functions $f$ in the class $\F3$, as well as for other convex functions. 
\end{remark*} 

Let $X_1,\dots,X_n$ be independent r.v.'s, with the sum $S:=X_1+\dots+X_n$. 
Also, recall now Definiton~\ref{def:Ga,tPi}. 

\begin{theorem}\label{th:main}
Let $\si$, $y$, and $\be$ be any (strictly) positive real numbers  such that 
\begin{equation}\label{eq:vp}
	\vp:=\frac{\be}{\si^2y}\in(0,1).     
\end{equation}
Suppose that  
\begin{equation}\label{eq:ineqs}
	\sum_i\E X_i^2\le\si^2,\quad\sum_i\E(X_i)_+^3\le\be,\quad\E X_i\le0,\quad\text{and $X_i\le y$ a.s., 
	}
\end{equation}
for all $i$.  
Then  
\begin{equation}\label{eq:PUfF3}
	\E f(S)\le\E f\big(\Ga_{(1-\vp)\si^2}+y\tPi_{\vp\si^2/y^2}\big)
\end{equation}
for all $f\in\F3$. 

\end{theorem}

The proof of Theorem~\ref{th:main} will be given in Section~\ref{proofs}, where all the necessary proofs are deferred to. 

Note that the condition $\vp\in(0,1)$ in \eqref{eq:vp} does not at all diminish generality, since it is easy to see that 
$\sum_i\E(X_i)_+^3<\si^2y$ for any positive $\si$ and $y$ and any r.v.'s $X_1,\dots,X_n$ such that $\sum_i\E X_i^2\le\si^2$, $\E X_i\le0$,  
and $X_i\le y$ a.s., for all $i$; so, one can always choose $\be$ to be in the interval $\big(\sum_i\E(X_i)_+^3,\si^2y\big)$, and then one will have $\vp\in(0,1)$.  

\begin{proposition}\label{prop:F3larger}
Let the class ${\F3}_{,1}$ of functions be defined by removing $f$ from the list ``$f,f',f'',f'''$'' in \eqref{eq:F3}; similarly define the class 
${\F3}_{,12}$ by removing both $f$ and $f'$ from the same list; thus, each of these two new classes is larger than the class $\F3$. 
\begin{enumerate}[(i)]
	\item If the condition ``$\E X_i\le0$ $\forall i$'' in Theorem~\ref{th:main} is replaced by  
``$\E X_i=0$ $\forall i$'', then inequality \eqref{eq:PUfF3} will hold for all $f$ in the larger class ${\F3}_{,1}$. 
	\item If the conditions ``$\E X_i\le0$ $\forall i$'' and $\sum_i\E X_i^2\le\si^2$ 
	in Theorem~\ref{th:main} are both replaced by the equalities  
``$\E X_i=0$ $\forall i$'' and $\sum_i\E X_i^2=\si^2$, then \eqref{eq:PUfF3} will hold for all $f$ in the larger class ${\F3}_{,12}$. 
\end{enumerate}  
\end{proposition}

\begin{proposition}\label{prop:exact}
For each triple $(\si,y,\be)$ of positive numbers satisfying condition \eqref{eq:vp} and each $f\in\F3$, the upper bound $\E f\big(\Ga_{(1-\vp)\si^2}+y\tPi_{\vp\si^2/y^2}\big)$ on $\E f(S)$, given by \eqref{eq:PUfF3}, is exact; moreover, this bound remains exact if the first three inequalities in the condition \eqref{eq:ineqs} are replaced by the corresponding equalities. 
\end{proposition}

Comparison inequality \eqref{eq:PUfF3} is optimal in yet another sense: namely, there the class $\F3$ of generalized moment functions $f$ cannot be substantially enlarged if \eqref{eq:PUfF3} is to remain true. To state this optimality property more precisely, let us first note a simple corollary of Theorem~\ref{th:main}, which follows immediately  
because $\H3\subseteq\F3$: 

\begin{corollary}\label{cor:H3}
In Theorem~\ref{th:main}, one can replace $\F3$ by $\H3$. 
\end{corollary}

In fact, in Section~\ref{proofs} essentially we shall first prove Corollary~\ref{cor:H3} and then extend the comparison inequality from $\H3$ to $\F3$. In this sense, one can say that Theorem~\ref{th:main} and Corollary~\ref{cor:H3} are equivalent to each other. Now one is ready to state the other optimality property: 

\begin{proposition}\label{prop:p=3}
 For any given $p\in(0,3)$, one cannot replace $\H3$ in Corollary~\ref{cor:H3} by the larger class $\H p$; in fact, this cannot be done even for $n=1$. 
\end{proposition}

%

By Theorem~\ref{th:comparison} and Remark~\ref{comparison-remark}, one immediately obtains the following corollary of Theorem~\ref{th:main}. 

\begin{corollary}\label{cor:main}
Under the conditions of Theorem~\ref{th:main}, for all $x\in\R$
\begin{align}
	\PP(S\ge x)\le\Pin(x)&:=P_3(\Ga_{(1-\vp)\si^2}+y\tPi_{\vp\si^2/y^2};x) \label{eq:main}\\
&\le c_{3,0}\,\PP^\lc(\Ga_{(1-\vp)\si^2}+y\tPi_{\vp\si^2/y^2}\ge x). \label{eq:main LC}
\end{align}
\end{corollary}


Bennett~\cite{bennett} states that ``for most practical problems, precisely'' ``information on the distribution function of a sum when the number of component random variables is small and/or the variables have different distributions'' ``is required''. Accordingly, let us consider now the case when --- instead of the upper bounds in \eqref{eq:ineqs} on the \emph{sums} of moments and the uniform a.s.\ upper bound $y$ on the $X_i$'s --- 
such upper bounds are available for the individual distributions of the summands $X_i$, with possibly different upper bounds for different $i$. More specifically, some of the summands $X_i$ may be significantly smaller (in a certain sense) than the rest of them. Then, grouping them together and using certain results of \cite{asymm}, one can obtain the following improvement of Theorem~\ref{th:main} and Corollary~\ref{cor:main}. 

\begin{corollary}\label{cor:improve}
Suppose that  
\begin{equation}\label{eq:ineqs individ}
	\text{$X_i\le y_i\le y$ a.s.},\quad\E X_i^2\le\si_i^2,\quad\E(X_i)_+^3\le\be_i,\quad\E X_i\le0,
\end{equation}
for all $i$, where $y$, $y_i$, $\si_i$, and $\be_i$ are some positive real numbers.  
Also, suppose that (cf.\ \eqref{eq:vp}) 
\begin{equation}\label{eq:tvp}
	\tvp:=\frac{\tbe}{\si^2y}\in(0,1),     
\end{equation}
where 
\begin{equation}\label{eq:tbe}
\tbe:=\sum_i\be_i\ii{y_i>\si_i}\quad\text{and}\quad\si:=\sqrt{\sum_i\si_i^2}. 
\end{equation}
Then inequalities \eqref{eq:PUfF3}, \eqref{eq:main}, and \eqref{eq:main LC} hold with $\tvp$ in place of $\vp$: 
\begin{align}
	\E f(S)&\le\E f\big(\Ga_{(1-\tvp)\si^2}+y\tPi_{\tvp\si^2/y^2}\big)\quad\text{for all }f\in\F3;  
	\label{eq:tPUfF3} \\
	\PP(S\ge x)&\le P_3(\Ga_{(1-\tvp)\si^2}+y\tPi_{\tvp\si^2/y^2};x) \label{eq:tmain}\\
&\le c_{3,0}\,\PP^\lc(\Ga_{(1-\tvp)\si^2}+y\tPi_{\tvp\si^2/y^2}\ge x)\quad\text{for all }x\in\R. \label{eq:tmain LC}
\end{align}	
\end{corollary} 

Note that conditions \eqref{eq:ineqs individ} together with \eqref{eq:tbe} will imply \eqref{eq:ineqs} if $\be:=\sum_i\be_i$. 
As for condition \eqref{eq:tvp}, similarly to condition \eqref{eq:vp}, it does not diminish generality. 
In fact, one will obviously have   
\begin{equation}\label{eq:tvp<vp}
	\tvp\le\vp. 
\end{equation} 
Then, \eqref{eq:tvp<vp} will imply (by Lemma~\ref{lem:le}) that inequalities \eqref{eq:tPUfF3}, \eqref{eq:tmain}, and \eqref{eq:tmain LC}, as established by Corollary~\ref{cor:improve}, will respectively be improvements of \eqref{eq:PUfF3}, \eqref{eq:main}, and \eqref{eq:main LC}.  

For completeness, let us also present results similar to Theorem~\ref{th:main}, Propositions~\ref{prop:F3larger}, \ref{prop:exact}, and \ref{prop:p=3}, and Corollary~\ref{cor:H3}, 
without conditions on the truncated third moments $\E(X_i)_+^3$ and for somewhat larger classes of generalized moment functions. 
Let (cf.\ \eqref{eq:F3})
\begin{align} \notag
\F2&:=\{f\in\C^1\colon \text{$f$ and $f'$ are nondecreasing and convex}\} \notag\\
&=\{f\in\C^1\colon \text{$f$, $f'$, $f''$ are nondecreasing}\}, \label{eq:F2}
\end{align} 
where $\C^1$ denotes the class of all continuously differentiable functions 
$f\colon\R\to\R$ and $f''$ denotes the right derivative of the convex function $f'$. 
For example, functions $x\mapsto a+b\,x+c\,(x-t)_+^\al$ and 
$x\mapsto a+b\,x+c\,e^{\la x}$ belong to $\F2$ for all $a\in\R$, $b\ge0$, $c\ge0$, $t\in\R$, $\al\ge2$, and $\la\ge0$. 
It is easy to see that $\H2\subseteq\F2$, and it is obvious that $\F3\subset\F2$. 

\begin{proposition}\label{prop:F2main}
\emph{(Cf.\ Theorem~\ref{th:main}.)} 
Let $\si$ and $y$ be any (strictly) positive real numbers. 
Suppose that  
\begin{equation}\label{eq:F2ineqs}
	\sum_i\E X_i^2\le\si^2,\quad\E X_i\le0,\quad\text{and $X_i\le y$ a.s., 
	}
\end{equation}
for all $i$.  
Then 
\begin{equation}\label{eq:PUfF2}
	\E f(S)\le\E f\big(y\tPi_{\si^2/y^2}\big)
\end{equation}
for all $f\in\F2$. 
\end{proposition}

\begin{proposition}\label{prop:F2larger}
\emph{(Cf.\ Proposition~\ref{prop:F3larger}.)} 
Let the class ${\F2}_{,1}$ of functions be defined by removing $f$ from the list ``$f,f',f''$'' in \eqref{eq:F2}; similarly define the class 
${\F2}_{,12}$ by removing both $f$ and $f'$ from the same list; thus, each of these two new classes is larger than the class $\F2$. 
\begin{enumerate}[(i)]
	\item If the condition ``$\E X_i\le0$ $\forall i$'' in Proposition~\ref{prop:F2main} is replaced by  
``$\E X_i=0$ $\forall i$'', then inequality \eqref{eq:PUfF2} will hold for all $f$ in the larger class ${\F2}_{,1}$. 
	\item If the conditions ``$\E X_i\le0$ $\forall i$'' and $\sum_i\E X_i^2\le\si^2$ 
	in Proposition~\ref{prop:F2main} are both replaced by the equalities  
``$\E X_i=0$ $\forall i$'' and $\sum_i\E X_i^2=\si^2$, then \eqref{eq:PUfF2} will hold for all $f$ in the larger class ${\F2}_{,12}$. 
\end{enumerate}  
\end{proposition}

As mentioned in the Introduction, similar results for (continuous-time) martingales that are stochastic integrals were obtained by Klein, Ma and Privault \cite{klein-priv}, for the class ${\F2}_{,1}$; that is, for the class of all functions $f$ such that $f$ and $f'$ are convex. Cf.\ Remark~\ref{rem:mart} below.   

\begin{proposition}\label{prop:exact2}
\emph{(Cf.\ Proposition~\ref{prop:exact}.)} 
For each pair $(\si,y)$ of positive numbers and each $f\in\F2$, the upper bound $\E f\big(y\tPi_{\si^2/y^2}\big)$ on $\E f(S)$, given by \eqref{eq:PUfF2}, is exact; moreover, this bound remains exact if the first two inequalities in the condition \eqref{eq:F2ineqs} are replaced by the corresponding equalities. 
\end{proposition}

\begin{corollary}\label{cor:H2}
\emph{(Cf.\ Corollary~\ref{cor:H3}.)} 
In Proposition~\ref{prop:F2main}, one can replace $\F2$ by $\H2$. 
\end{corollary}

As mentioned in the Introduction, Corollary~\ref{cor:H2} is essentially contained in Bentkus~\cite{bent-ap}. 
By Theorem~\ref{th:comparison} and Remark~\ref{comparison-remark}, Corollary~\ref{cor:H2} immediately implies the Bentkus inequality \eqref{eq:Be}. 

\begin{proposition}\label{prop:p=2}
\emph{(Cf.\ Proposition~\ref{prop:p=3}.)} 
For any given $p\in(0,2)$, one cannot replace $\H2$ in Corollary~\ref{cor:H2} by the larger class $\H p$; in fact, this cannot be done even for $n=1$. 
\end{proposition}

\begin{remark}\label{rem:mart}
Quite similarly to how it was done e.g.\ in \cite{normal,asymm}, it is easy to extend the results of Theorem~\ref{th:main}, Propositions~\ref{prop:F3larger}, \ref{prop:F2main}, and \ref{prop:F2larger}, and Corollary~\ref{cor:main} to the more general case when the $X_i$'s are the incremental differences of a (discrete-time) (super)martingale and/or replace $S$ by the maximum of the partial sums; cf.\ e.g.\ \cite[Corollary~5]{asymm}. Let us omit the details. 
\end{remark} 

On majorization of the distributions of sums of independent r.v.'s by compound Poisson distributions see e.g.\ \cite{prokh,pin-utev1,pin-utev2,utev,pin-spher,borisov-compound}. 
Also indirectly related to the present paper is the work \cite{bor-bor,borisov}, where it was shown that the rate of convergence in the functional central limit theorem can be significantly improved if the limit is taken to be the convolution of appropriately chosen Gaussian and Poisson distributions, rather than just a Gaussian distribution. Of course, this quite well corresponds with the fact that the limit distributions for the sums of uniformly small independent r.v.'s are precisely the limits of convolutions of Gaussian and compound Poisson distributions. 
One may also note here the work \cite{a_nag}, where, by taking specific heavy tails into account, asymptotics of large deviation probabilities $\PP(S_n\ge x)$ for the sum $S_n$ of i.i.d.\ r.v.'s was obtained essentially without any restrictions on $x$ other than that just $x/\sqrt n\to\infty$ or, equivalently, $\PP(S_n\ge x)\to0$; functional versions of such results were given in\cite{pin81}.

\section{Computation and comparison of the upper bounds on the tail \\ 
probability $\PP(S\ge x)$ }\label{comp}

\subsection{Computation}\label{comput}
The Bennett-Hoeffding upper bound $\BH(x)$, given by \eqref{eq:hoeff}, is quite easy to compute. It is almost as easy to compute the Pinelis-Utev upper bound $\PU(x)$, defined in \eqref{eq:PU}. 

\begin{proposition}\label{prop:lambert}
For all $\si>0$, $y>0$, $\vp\in(0,1)$, and $x\ge0$
\begin{align}
	\PU(x)&=e^{-\la_x x}\PU_{\exp}(\la_x) \label{eq:PU la_x}\\
	&=\exp\frac{
   (1-\vp)^2 (w_x+1)^2-(\vp+x y/\si^2)^2-(1-\vp^2)}{2 (1-\vp)y^2/\si^2} \label{eq:PUexpr}
\end{align} 
where $\PU_{\exp}$ is defined in \eqref{eq:PUexp}, 
\begin{equation}\label{eq:la_x,w}
	\la_x:=\frac1y\Big(\frac{\vp+xy/\si^2}{1-\vp}-w_x\Big),\quad w_x:=L\Big(\frac\vp{1-\vp}\,\exp\frac{\vp+xy/\si^2}{1-\vp}\Big), 
\end{equation}
and $L$ is (the principal branch of) the Lambert product-log function, so that for all $z\ge0$ the value $w=L(z)$ is the only real root of the equation $w e^w=z$. 

Moreover, $\la_x$ increases in $x$ from $0$ to $\infty$ as $x$ does so. 
\end{proposition}


Thus, indeed $\PU(x)$ is easy to compute, since the Lambert function is about as easy to compute as the logarithmic one; in particular, in Mathematica the Lambert function is the built-in function \verb9ProductLog9; see e.g.\ \cite{knuth} and references there concerning this matter.  

A slight advantage of expression \eqref{eq:PUexpr} over \eqref{eq:PU la_x} is that \eqref{eq:PUexpr} contains just one entry of $w_x$, while \eqref{eq:PU la_x} contains several entries of $\la_x$ (recall \eqref{eq:PUexp} and \eqref{eq:la_x,w}); also, the exponent in \eqref{eq:PUexpr} is algebraic (actually quadratic) in $w_x$.

As for bounds $\Be(x)=P_2(y\tPi_{\si^2/y^2};x)$ and $\Pin(x)=P_3(\Ga_{(1-\vp)\si^2}+y\tPi_{\vp\si^2/y^2};x)$, as defined by \eqref{eq:Be} and \eqref{eq:main}, the computation of $P_\al(\eta;x)$ for general $\al$ and $\eta$ is described by \cite[Theorem~2.5]{pin98};  
for normal $\eta$, similar considerations were given already in \cite[page~363]{pin94}.   
The following proposition is essentially a special case of \cite[Theorem~2.5]{pin98}. 

\begin{proposition}\label{prop:Th2.5,pin98}
Take any real $\al>1$ and let $\eta$ be any real-valued r.v.\ such that $\E\eta_+^\al<\infty$. 
Then there exists $\E\eta\in[-\infty,\infty)$. 
Let 
$$x_*:=\sup\supp(\eta)\quad\text{and}\quad x_{**}:=\sup\big(\supp(\eta)\setminus\{x_*\}\big),$$
where $\supp(\eta)$ denotes, as usual, the topological support of the distribution of the r.v.\ $\eta$; note that $x_{**}=x_*$ unless $x_*$ is an isolated point of $\supp(\eta)$; 
in most applications, $x_*=\infty$ and hence $x_{**}=\infty$). 
For all $t\in(-\infty,x_*)$, let 
\begin{align}
m(t)&:=m_{\al,\eta}(t):=t+\frac{\E(\eta-t)_+^\al}{\E(\eta-t)_+^{\al-1}}=\frac{\E\eta(\eta-t)_+^{\al-1}}{\E(\eta-t)_+^{\al-1}}; \label{eq:m}
\end{align}
let also $m(x_*):=x_*$. 
Then
\begin{enumerate}[(i)]
\item the function $m$ is continuous on $(-\infty,x_*)$, left-continuous at $x_*$, and strictly increasing on $(-\infty,x_{**})$, from $\E\eta$ to $x_*$; also, $m(t)=x_*$ for all $t\in[x_{**},x_*]$. 
\item for every $x\in(\E\eta,x_*)$ there exists a unique $t_x=t_{x;\al,\eta}\in(-\infty,x_*)$ such that 
\begin{equation}\label{eq:m(t)=x}
	m(t_x)=x;
\end{equation}
in fact, $t_x\in(-\infty,x)$; 
\item for every $x\in(\E\eta,x_*)$
\begin{equation}\label{eq:P(x)}
P_\al(\eta;x)
=\frac{\E(\eta-t_x)_+^\al}{(x-t_x)^\al}=\frac{\E^\al(\eta-t_x)_+^{\al-1}}{\E^{\al-1}(\eta-t_x)_+^\al};	
\end{equation}
\item 
\begin{enumerate}
\item if $x\in(-\infty,\E\eta]$ then 
$P_\al(\eta;x)=1$; 
\item if $x\in[x_*,\infty)$ then 
$P_\al(\eta;x)=\PP(\eta=x)=\PP(\eta\ge x)
$; 
\end{enumerate}
it is therefore natural to extend $P_\al(\eta;x)$ to all $x\in[-\infty,\infty]$ by letting $P_\al(\eta;-\infty):=1$ and $P_\al(\eta;\infty):=0$ --- as will henceforth be assumed; 
\item $P_\al(\eta;x)$ strictly and continuously decreases from $1$ to $\PP(\eta=x_*)=\PP(\eta\ge x_*)$ as $x$ increases from $\E\eta$ to $x_*$; more specifically, 
\begin{enumerate}
\item the function $x\mapsto P_\al(\eta;x)$ is strictly decreasing on $(\E\eta,x_*)$;
\item it is also continuous on $(\E\eta,x_*)$, right-continuous at $\E\eta$, and left-continuous at $x_*$; 
hence, it is in fact strictly decreasing on the entire closed interval $[\E\eta,x_*]$;
\end{enumerate}
\item for any $a\in\R$ and $b>0$, one has
\begin{align*}
	t_{a+bx;\al,a+b\eta}&=a+bt_{x;\al,\eta}\quad\text{for all $x\in(\E\eta,x_*)$};\\
	P_\al(a+b\eta;x)&=P_\al(\eta;\tfrac{x-a}b)\quad\text{for all $x\in\R$}.
\end{align*}
\end{enumerate}
\end{proposition}
(Concerning the case $\al\in(0,1]$, see \cite[Remark~2.6]{pin98}.)


The following example illustrates Proposition~\ref{prop:Th2.5,pin98}, and also  
Proposition~\ref{prop:P_infty} (to be presented later, in Subsubsection~\ref{ineqs,idents}). 

\begin{ex*}
Take any real $\al>1$. 
Let $\eta$ be a zero-mean r.v.\ taking on only two values, $-a$ and $b$, where $a$ and $b$ are arbitrary positive real numbers. Then $x_*=b$, $x_{**}=-a$, and, using (say) the first expression for $P_\al(\eta;x)$ in \eqref{eq:P(x)}, one can see that 
\begin{equation*}
	P_\al(\eta;x)=\frac{(b+a)^{\al-1}ba}{\big[\big(b(x+a)^\al\big)^{\frac1{\al-1}}
	+\big(a(b-x)^\al\big)^{\frac1{\al-1}}\big]^{\al-1}}
\end{equation*}
for all $x\in[0,b]$; also, $P_\al(\eta;x)=1$ for all $x\in[-\infty,0]$ and $P_\al(\eta;x)=\frac a{b+a}\ii{x=b}=\PP(\eta\ge x)$ for all $x\in[b,\infty]$. 
\begin{center}
\includegraphics[scale=0.5]{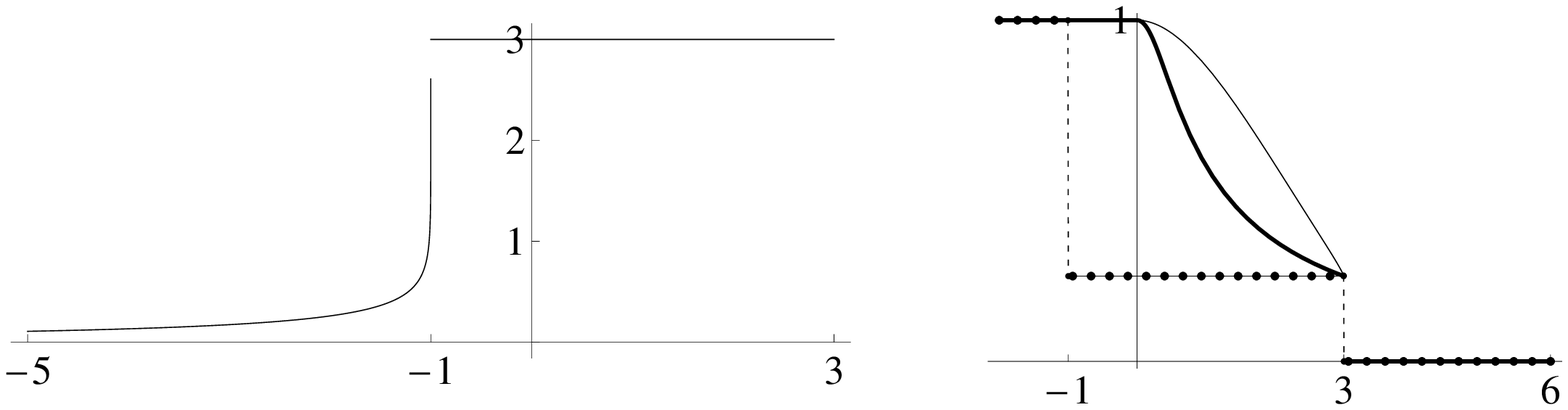}
\end{center}
Here the picture on the left shows the graph $\{(t,m(t))\colon -5<t<x_*\}$ for $a=1$, $b=3$ and $\al=1.2$, while the  picture on the right shows the graphs $\{(x,P_\al(\eta;x))\colon -2<x<x_*+3\}$ (the thick line) and $\{(x,\PP(\eta\ge x))\colon -2<x<x_*+3\}$ (thick-dotted over the thin line), also for $a=1$, $b=3$ and $\al=1.2$. A gap is seen in the graph $\{(t,m(t))\colon -5<t<x_*\}$ in a left neighborhood of $t=-1$, which is caused (despite making, with Mathematica,  15 recursive subdivisions with 1000 initial sample points) by a very steep increase of the function $m$ in such a neighborhood; for instance, $m(-1.000001)$ is only $2.498\ldots$, while $m(-1)=3$; yet, according to Proposition~\ref{prop:Th2.5,pin98}(i), there is no discontinuity there. 
The picture on the right also shows \big(see definition \eqref{eq:P_infty def} and relation \eqref{eq:P_infty} below\big) 
the graph $\{(x,P_\infty(\eta;x))\colon -2<x<x_*+3\}$ (the thinner line) of the best exponential bound 
\begin{align*}
	P_\infty(\eta;x)=\inf_{\la>0}e^{-\la x}\E e^{\la\eta}&=\lim_{\al\to\infty}P_\al(\eta;x) \\
	&=
	\Big(\frac{x+a}a\Big)^{-\frac{x+a}{a+b}} \Big(\frac{b-x}b\Big)^{-\frac{b-x}{a+b}} 
\end{align*}
for all $x\in[0,b)$, also with $P_\infty(\eta;x)=1$ for all $x\in[-\infty,0]$ and $P_\infty(\eta;x)=\frac a{b+a}\ii{x=b}=\PP(\eta\ge x)$ for all $x\in[b,\infty]$. 
While, in this case, one may not be greatly impressed with the overall degree of closeness of the upper bound $P_\al(\eta;x)$ to $\PP(\eta\ge x)$, note that in the ``large-deviation'' zone $x\ge b$ the performance of the bound $P_\al(\eta;x))$ is perfect: $P_\al(\eta;x))=\PP(\eta\ge x)$ for all $x\ge b$, just in accordance with Proposition~\ref{prop:Th2.5,pin98}(iv)(b). 
\end{ex*}

In particular, Proposition~\ref{prop:Th2.5,pin98} shows that the computation of the upper bound $P_\al(\eta;x)$ is based on that of the positive-part moments $\E(\eta-t)_+^\al$ and $\E(\eta-t)_+^{\al-1}$. For $\al\in\{1,2,3\}$ and a number of common families of distributions of $\eta$, including the Poisson one, this computation was detailed in \cite{bent-64pp}. In particular, see formula \cite[(10.5)]{bent-64pp} for $P_2(\eta;x)$ with a centered Poisson r.v.\ $\eta$. That formula is relatively simple, since, for a natural $\al$ and a r.v.\ $\eta$ with (say) a lattice distribution, the generalized moment $\E(\eta-t)_+^\al$ can be computed ``locally''; indeed, if $\cdots<d_k<d_{k+1}<\cdots$ are the atoms of the distribution of $\eta$, then for any $t\in[d_k,d_{k+1})$ one a.s.\ has 
$\eta>t$ iff $\eta>d_k$; thus, for such $t$, $\E(\eta-t)_+^\al$ can be easily expressed in terms of the truncated moments $\E(\eta-d_k)_+^j$ with $j=0,\dots,\al$.  
These comments provide a simple way to compute the bound $\Be(x)=P_2(y\tPi_{\si^2/y^2};x)$.

As for the bound $\Pin(x)=P_3(\Ga_{(1-\vp)\si^2}+y\tPi_{\vp\si^2/y^2};x)$, here there is no such nice localization property as the one mentioned in the previous paragraph, since the distribution of the r.v.\ $\Ga_{(1-\vp)\si^2}+y\tPi_{\vp\si^2/y^2}$ is not discrete. It appears that the computation of the positive-part moments $\E(\eta-t)_+^\al$ for $\eta=\Ga_{(1-\vp)\si^2}+y\tPi_{\vp\si^2/y^2}$ can be done most effectively via formulas expressing such moments in terms of the Fourier or Fourier-Laplace transform of the distribution of $\eta$; see \cite{pos}, where such formulas were developed (with this specific motivation in mind). 
A reason for this approach to work is that the Fourier-Laplace transform of the distribution of the r.v.\ $\Ga_{(1-\vp)\si^2}+y\tPi_{\vp\si^2/y^2}$ has a simple expression (cf.\ \eqref{eq:PUexp-expr} and \eqref{eq:PUexp}). 

Namely, one has
\begin{align}
 \E X_+^p &=
 \frac{\Ga(p+1)}{\pi}\int_0^\infty
\Re\frac{\E e_j\big((s+it)X\big)}{(s+it)^{p+1}}\,\dd t, \label{eq:laplace}
\end{align}
where $p\in(0,\infty)$, $s\in(0,\infty)$, $\Ga$ is the Gamma function, $\Re$ denotes the real part of a complex number, $i$ is the imaginary unit, $j=-1,0,\dots,\ell$, $\ell:=\ce{p-1}$, $e_j(u):=e^u-\sum
	_{m=0}^j\tfrac{u^m}{m!}$, 
and $X$ is any r.v.\ such that  $\E|X|^{j_+}<\infty$ and $\E e^{sX}<\infty$.
Also, 
\begin{equation}\label{eq:char}
 \E X_+^p =
 \frac{\E X^k}2\ii{p\in\N}+
 \frac{\Ga(p+1)}\pi 
\,\int_0^\infty\Re
\frac{\E e_\ell(itX)}{(it)^{p+1}}\,\dd t,
\end{equation} 
where $k:=\fl{p}$ 
and $X$ is any r.v.\ such that  $\E|X|^p<\infty$.
Of course, formulas \eqref{eq:laplace} and \eqref{eq:char} will be applied here to r.v.'s of the form $X=\Ga_{(1-\vp)\si^2}+y\tPi_{\vp\si^2/y^2}-w$, with $w\in\R$.

\subsection{Comparison}\label{compar}
In this subsection, we shall compare the bounds $\BH(x)$, $\PU(x)$, $\Be(x)$, and $\Pin(x)$, by means of identities and inequalities (in Subsubsection~\ref{ineqs,idents}), asymptotic relations for large $x>0$ (in Subsubsection~\ref{large dev}), and graphics and numerics for moderate $x>0$ (in Subsubsection~\ref{moder dev}); we shall also include into these comparisons 
the Cantelli bound $\frac{\si^2}{\si^2+x^2}$ and the best exponential upper bound $\exp\{-\frac{x^2}{2\si^2}\}$ on the tail of the normal distribution $\No(0,\si^2)$.

\subsubsection{Inequalities and identities}\label{ineqs,idents}

Let us begin here with the following simple proposition concerning the bounds $P_\al(\eta;x)$ (as defined in \eqref{eq:comp-prob2}). 
Unless specified otherwise, let $\eta$ in this subsubsection stand for any r.v., and take any $\al\in(0,\infty)$. 

\begin{proposition}\label{prop:P_al=inf}
For any $x\in\R$,  
\begin{align}
	P_\al(\eta;x)&=\inf\{\E f(\eta)\colon f\in\H\al,\ f(u)\ge\ii{u\ge x}\ \forall u\in\R\} \label{eq:P_al=inf1} \\ 
	&=\inf\Big\{\frac{\E f(\eta)}{f(x)}\colon f\in\H\al,\ f(x)>0\Big\}. \label{eq:P_al=inf2} 
\end{align}
\end{proposition} 


Now let us state general relations between the bounds $P_\al(\eta;x)$ for different values of $\al$, as well as their relation with the best exponential upper bound 
%
\begin{align}
	P_\infty(\eta;x)
	&:=\inf_{\la>0}\frac{\E e^{\la\eta}}{e^{\la x}}. \label{eq:P_infty def} 
\end{align}

%

\begin{proposition}\label{prop:H_infty} 
\begin{equation}\label{eq:H_infty}
	\bigcap_{\al>0}\H\al=\H\infty=\H{\exp},  
\end{equation}
where $\H\infty$ is defined as the class of all infinitely differentiable real functions $f$ on $\R$ such that $f^{(j)}\ge0$ on $\R$ and $f^{(j)}(-\infty+)=0$ for all $j=0,1,\dots$, and 
\begin{equation*}
	\H{\exp}:=\{f\colon f(x)=\textstyle{\int_{(0,\infty)}}
	e^{tx}\mu(\dd t)\ \forall x\in\R\},
\end{equation*}
where $\mu$ denotes a nonnegative Borel measure such that the integral $\int_{(0,\infty)}e^{tx}\mu(\dd t)$ is finite $\forall x\in\R$; thus $\H{\exp}$ may be viewed as a closed convex hull of the set of all increasing exponential functions. 
\end{proposition}

Using Proposition~\ref{prop:H_infty}, one can obtain

\begin{proposition}\label{prop:P_infty}\ 
\begin{enumerate}[(i)]
	\item The function $(0,\infty]\ni\al\mapsto P_\al(\eta;x)$ is nondecreasing. 
	\item For all $x\in\R$ 
\begin{equation}\label{eq:P_infty}
	P_\infty(\eta;x)=\lim_{\al\to\infty}P_\al(\eta;x). 
\end{equation}
\end{enumerate} 
\end{proposition}

For completeness, let us also consider 
the Cantelli bound 
\begin{equation}\label{eq:Ca}
\Ca(x):=\Ca_{\si^2}(x):=\frac{\si^2}{\si^2+x^2}   	
\end{equation}
and 
the best exponential upper bound 
\begin{equation}\label{eq:EN}
\EN(x):=\EN_{\si^2}(x):=P_\infty(\Ga_{\si^2};x)
=\exp\Big\{-\frac{x^2}{2\si^2}\Big\}	
\end{equation}
on the tail of the normal distribution $\No(0,\si^2)$; 
of course, in general $\EN(x)$ is \emph{not} an upper bound on $\PP(S\ge x)$.   

The bound $\Ca(x)$ can be presented in a form similar to \eqref{eq:P_al=inf2} and \eqref{eq:P_infty def}: 

\begin{proposition}\label{prop:Ca=inf}
Take any $\si\in(0,\infty)$, any r.v.'s\ $\xi$ and $\eta$ such that $\E\xi\le0=\E\eta$ and $\E\xi^2\le\E\eta^2=\si^2$, and any $x\in[0,\infty)$. Then  
\begin{align}
\PP(\xi\ge x)&\le\Ca(x) 
=\inf_{t\in(-\infty,x)}\frac{\E(\eta-t)^2}{(x-t)^2}. \label{eq:Caexpr}
\end{align}
\end{proposition} 

This proposition is essentially well known; yet, we shall provide a proof for the readers' convenience. 

Now we are ready to turn to relations between the four related bounds: $\BH(x)$, $\PU(x)$, $\Be(x)$, and $\Pin(x)$, as well as $\Ca(x)$ and $\EN(x)$. 

\begin{proposition}\label{prop:ineqs}
For all $x>0$ and all values of the parameters: $\si>0$, $y>0$, and $\vp\in(0,1)$,  
\begin{enumerate}[(I)]
	\item $\Pin(x)\le\PU(x)\le\BH(x)\quad\text{and}\quad\Be(x)\le\Ca(x)\wedge\BH(x);$
	\item
$
	\Be(x)=\Ca(x)\quad\text{for all }x\in[0,y] 
$;
	\item $\BH(x)$ increases from $\EN(x)$ to $1$ as $y$ increases from $0$ to $\infty$;   
	\item there exists some $u_{y/\si}\in(0,\infty)$ that depends only on the ratio $y/\si$ such that $\Ca(x)<\BH(x)$ if $x\in(0,\si u_{y/\si})$ and $\Ca(x)>\BH(x)$ if $x\in(\si u_{y/\si},\infty)$; moreover, $u_{y/\si}$ increases from $u_{0+}=1.585\dots$ to $\infty$ as $y/\si$ increases from $0$ to $\infty$; 
	in particular, $\Ca(x)<\EN(x)$ if $x/\si\in(0,1.585)$ and $\Ca(x)>\EN(x)$ for $x/\si\in(1.586,\infty)$. 
	\item $\PU(x)$ increases from $\EN(x)$ to $\BH(x)$ as $\vp$ increases from $0$ to $1$.  
\end{enumerate}
\end{proposition}

\begin{proposition}\label{prop:idents}
For all $\si>0$, $y>0$, $\vp\in(0,1)$, and $x>0$
\begin{align}
	\PU(x)
&=\max\{\Ne_{(1-\vp)\si^2}((1-\al)x)\BH_{\vp \si^2,y}(\al x)\colon\al\in(0,1)\} \label{eq:max al}\\
&=\Ne_{(1-\vp)\si^2}((1-\al_x)x)\BH_{\vp \si^2,y}(\al_x x), \label{eq:max at al_x}
\end{align} 
where $\al_x$ is the only root in $(0,1)$ of the equation 
\begin{equation}\label{eq:al_x}
	\frac{(1-\al)x^2}{(1-\vp)\si^2}-\frac xy\ln\Big(1+\frac{\al xy}{\vp\si^2}\Big)=0.
\end{equation} 
Moreover, $\al_x$ increases from $\vp$ to $1$ as $x$ increases from $0$ to $\infty$; in particular,
\begin{equation}\label{eq:al>vp}
	\al_x\in(\vp,1)
\end{equation}
for all $x>0$.  
\end{proposition}

Expressions \eqref{eq:max al} and \eqref{eq:max at al_x} provide a rather curious interpretation of the bound $\PU(x)$ as the product of the best exponential upper bounds on the tails $\PP\big(\Ga_{(1-\vp)\si^2}\ge(1-\al)x\big)$ and $\PP\big(\tPi_{\vp\si^2}\ge\al x\big)$ --- for some $\al$ in $(0,1)$ (in fact, the $\al$ is in the interval $(\vp,1)$). In view of \eqref{eq:PUexp-expr}, this interpretation should not come as a big surprise. 
Proposition~\ref{prop:idents} will useful in the proof of Proposition~\ref{prop:PU<<BH}.  

\begin{proposition}\label{prop:mono in y}
For any $f\in\H2$, $\si>0$, $y>0$, and $\vp\in(0,1)$, 
\begin{equation}\label{eq:mono in y}
	\E f\big(\Ga_{(1-\vp)\si^2}+y\tPi_{\vp\si^2/y^2}\big)\le\E f\big(y\tPi_{\si^2/y^2}\big). 
\end{equation}
\end{proposition}

So, by Proposition~\ref{prop:mono in y}, of the two r.v.'s --- $\Ga_{(1-\vp)\si^2}+y\tPi_{\vp\si^2/y^2}$ and \ $y\tPi_{\si^2/y^2}$ --- with the same variance $\si^2$, the former one (with a light-tail component $\Ga_{(1-\vp)\si^2}$) is in a certain sense smaller than the latter, purely heavy-tail one. This suggests that the upper bounds $\Pin(x)$ and $\PU(x)$, which are based on $\Ga_{(1-\vp)\si^2}+y\tPi_{\vp\si^2/y^2}$, will tend to be smaller than the bound $\Be(x)$, which is based on $y\tPi_{\si^2/y^2}$. Such heuristics is to an extent justified by results of  Subsubsections~\ref{large dev} and \ref{moder dev}, especially by Corollary~\ref{cor:compar} in Subsubsection~\ref{large dev} and the graphics for $\vp=0.1$ in Subsubsection~\ref{moder dev}. 

\begin{proposition}\label{prop:PLC(Po>x)}
(Recall Definition~\ref{def:lc}.) 
For the least concave majorant of the tail function of the Poisson distribution $\Po(\th)$ one has
\begin{equation*}
	\PP^\lc(\Pi_\th\ge u)=\PP(\Pi_\th\ge j)^{j+1-u}\PP(\Pi_\th\ge j+1)^{u-j}\le\PP(\Pi_\th\ge j)
\end{equation*}
for all $\th>0$ and $u\in\R$, where $j:=j_u:=\lceil u-1\rceil$. 
\end{proposition}

\begin{proposition}\label{prop:PLC(Ga+Po>x)}
For all $\si>0$, $y>0$, $\vp\in(0,1)$, and $x\in\R$
\begin{equation}\label{eq:PLC(Ga+Po>x)}
	\PP^\lc(\Ga_{(1-\vp)\si^2}+y\tPi_{\vp\si^2/y^2}\ge x)\le
	\int_\R\PP^\lc(y\tPi_{\vp\si^2/y^2}\ge z)\PP(x-\Ga_{(1-\vp)\si^2}\in\dd z).  
\end{equation}
\end{proposition}
The term $\PP^\lc(y\tPi_{\vp\si^2/y^2}\ge z)$ in \eqref{eq:PLC(Ga+Po>x)} is to be evaluated or bounded according to Proposition~\ref{prop:PLC(Po>x)}, using at that Remark~\ref{rem:PLC}.

\subsubsection{Asymptotics for large deviations}\label{large dev} 

Here and in what follows, for any two expressions $\EE_j(x)=\EE_{j;\si,y,\vp}(x)$ (with $j=1,2$) the notation $\EE_1(x)\OO\EE_2(x)$ will mean ``$\EE_1(x)\le C\EE_2(x)$ for some positive constant factor $C$ not depending on $x$, for all large enough $x>0$''; $\EE_2(x)\OOG\EE_1(x)$ will mean the same as $\EE_1(x)\OO\EE_2(x)$. 
Notation like $\EE_1(x)\sim\EE_2(x)$ will mean, as usual, that $\EE_1(x)/\EE_2(x)\to1$.  

\begin{proposition}\label{prop:PU<<BH}
For any fixed $\si>0$, $y>0$, and $\vp\in(0,1)$, 
\begin{align}
	\PU(x)&\sim C\vp^{x/y}\BH(x)
	\exp\Big\{\frac{(1-\vp)\si^2}{2y^2}\Big[\ln^2\Big(1+\frac{xy}{\vp\si^2}\Big)-2\ln\Big(1+\frac{xy}{\si^2}\Big)\Big]\Big\} 
	\notag
	\\ 
	&=(\vp+o(1))^{x/y}\BH(x) 
	\notag
\end{align}
as $x\to\infty$, where $C:=\exp\{\frac{\si^2}{y^2}\psi(\vp-1)\}$. 
\end{proposition}

\begin{proposition}\label{prop:be asymp}
For any fixed $\si>0$ and $y>0$ 
and (say) all $x\in[y,\infty)$
\begin{align}
\frac{\BH(x)}{x^{3/2}}\OO\PP(y\tPi_{\si^2/y^2}\ge x)\le\Be(x)&\le c_{2,0}\PP^\lc(y\tPi_{\si^2/y^2}\ge x) 
\label{eq:be asymp1}\\
&\OO x\PP(y\tPi_{\si^2/y^2}\ge x)\le x\BH(x). \label{eq:be asymp2}
\end{align}  
\end{proposition}

Proposition~\ref{prop:be asymp} implies that for $x\in[y,\infty)$ the upper bounds $\BH(x)$, $\Be(x)$, and  $c_{2,0}\PP^\lc(y\tPi_{\si^2/y^2}\ge x)$ on $\PP(S\ge x)$ --- as well as the particular, limit instance $\PP(y\tPi_{\si^2/y^2}\ge x)$ of $\PP(S\ge x)$ --- are the same up to a power-function factor, of the form $Cx^{5/2}$, where $C=C_{\si,y}>0$ does not depend on $x$. 

\begin{proposition}\label{prop:pin asymp}
For any fixed $\si>0$, $y>0$, and $\vp\in(0,1)$, 
and (say) all $x\in[y,\infty)$
\begin{align}
\frac{\PU(x)}{x^{3/2}}\OO\PP(\eta_{\si,y,\vp}\ge x)\le\Pin(x) 
&\le c_{3,0}\PP^\lc(\eta_{\si,y,\vp}\ge x) \label{eq:pin asymp1}\\
&\OO x\PP(\eta_{\si,y,\vp}\ge x) 
\le x\PU(x), \label{eq:pin asymp2}
\end{align} 
where
\begin{equation}\label{eq:eta}
	\eta_{\si,y,\vp}:=\Ga_{(1-\vp)\si^2}+y\tPi_{\si^2/y^2}.
\end{equation} 
\end{proposition}

Proposition~\ref{prop:pin asymp} implies that for $x\in[y,\infty)$ the upper bounds $\PU(x)$, $\Pin(x)$, and  $c_{3,0}\PP^\lc(\Ga_{(1-\vp)\si^2}+y\tPi_{\si^2/y^2}\ge x)$ on $\PP(S\ge x)$ --- as well as the particular, limit instance $\PP(\Ga_{(1-\vp)\si^2}+y\tPi_{\si^2/y^2}\ge x)$ of $\PP(S\ge x)$ --- are the same up to a power-function factor, of the form $Cx^{5/2}$, where $C=C_{\si,y,\vp}>0$ does not depend on $x$. 

Thus, Propositions~\ref{prop:PU<<BH}, \ref{prop:be asymp}, and \ref{prop:pin asymp} imply that either of the bounds $\PU(x)$ or $\Pin(x)$ is better than both $\BH(x)$ and $\Be(x)$ by a factor which is decreasing exponentially fast in $x$, for large enough $x>0$. More precisely, taking also into account the inequality $\Be(x)\le\BH(x)$ in 
Proposition~\ref{prop:ineqs}(i), one immediately obtains 

\begin{corollary}\label{cor:compar}
For any fixed $\si>0$, $y>0$, and $\vp\in(0,1)$, 
and all $x\ge0$
\begin{equation*}
	\Pin(x)\le\PU(x)\le(\vp+o(1))^{x/y}\Be(x)\le(\vp+o(1))^{x/y}\BH(x)
\end{equation*} 
as $x\to\infty$. 	
\end{corollary}

Of course, the asymptotically better bounds $\PU$ and $\Pin$ require information on the sum of truncated third moments, in addition to that on the sum of second moments. However, it is difficult to imagine a situation when only the latter (but not the former) kind of information is available.



\begin{proposition}\label{prop:normal opt asymp}
For any fixed $\al>1$ and $\si>0$,  
\begin{equation*}
	P_\al(\Ga_{\si^2};x)\sim c_{\al,0}\PP(\Ga_{\si^2}\ge x)\quad\text{as $x\to\infty$.}
\end{equation*} 
\end{proposition}


Thus, for a centered Gaussian r.v.\ $\eta$, the optimal upper bound $P_\al(\eta;x)$ on the tail $\PP(\eta\ge x)$ 
differs from it approximately by a constant factor $c_{\al,0}\in(1,\infty)$ for large $x>0$. 

If $\eta$ is a centered Poisson r.v.\ $\tPi_\th$, then the asymptotic behavior of the ratio $P_\al(\eta;x)/\PP(\eta\ge x)$ is starkly different: it oscillates between nearly $1$ and a factor of the order of $x$ -- as seen from the following proposition, which also shows that the factor $x$ in \eqref{eq:be asymp2} cannot be substantially improved. 
More precisely, one has 

\begin{proposition}\label{prop:pois opt asymp}
For any fixed $\al>1$ and $\th>0$,  
\begin{equation}\label{eq:pois centered opt asymp}
	P_\al(\tPi_\th;k-\th)\sim\PP(\tPi_\th\ge k-\th)\sim\frac k\th\PP(\tPi_\th>k-\th)
\end{equation} 
as $\Z\ni k\to\infty$. 
\end{proposition}

To illustrate Proposition~\ref{prop:pois opt asymp}, here is the graph of 
$\frac{\Be(x)}{\PP(\tPi_\th\ge x)}-1=\frac{P_2(\tPi_\th;x)}{\PP(\tPi_\th\ge x)}-1$ with $\th=\si^2=0.6$ and $y=1$, over  $x\in[0,7.4]$: 
\begin{center}
\includegraphics[scale=0.6]{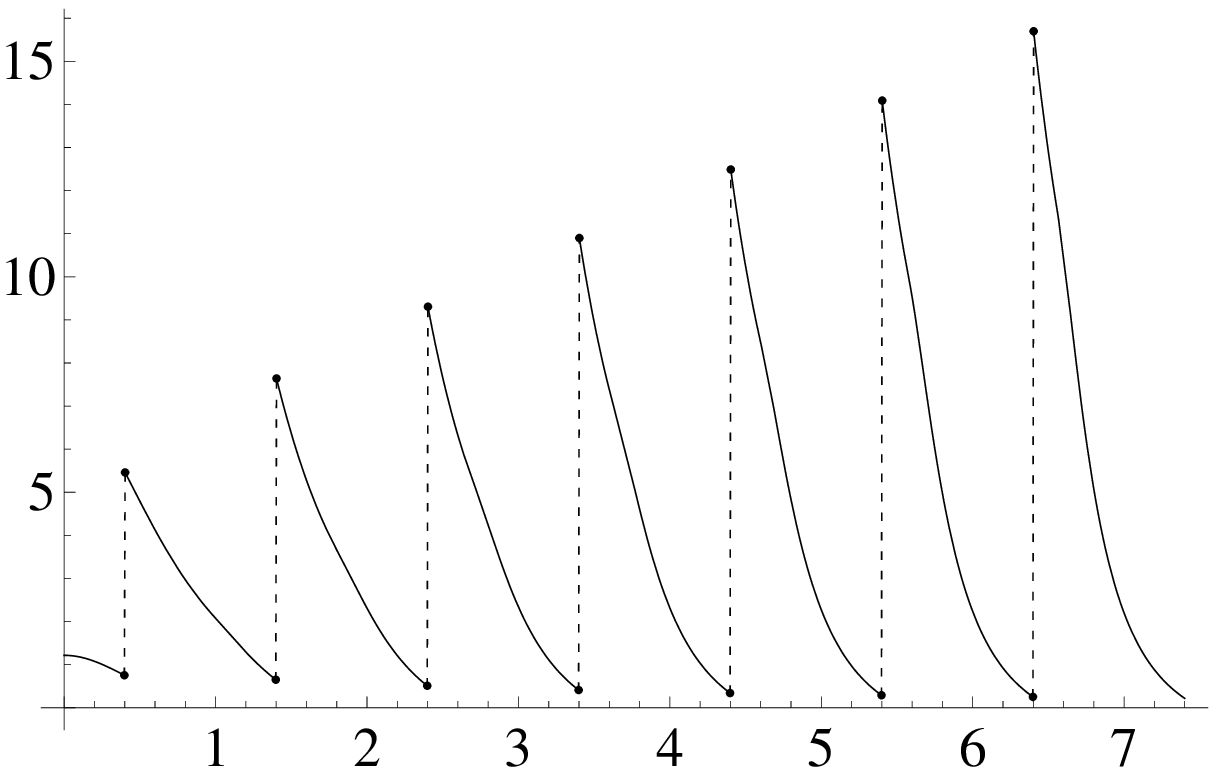}
\end{center}

One can expect the behavior of the ratio $P_\al(\eta;x)/\PP(\eta\ge x)$ for $\eta=\Ga_{(1-\vp)\si^2}+y\tPi_{\si^2/y^2}$ and large $x>0$ to be intermediate between the two kinds described in Propositions~\ref{prop:normal opt asymp} and \ref{prop:pois opt asymp}.

\subsubsection{Numerics and graphics for moderate deviations}\label{moder dev} 

In Subsubsection~\ref{large dev}, it was shown that the bounds $\Pin(x)$ and $\PU(x)$ are much better than $\Be(x)$ and $\BH(x)$ for all large enough $x>0$. 
For moderate deviations, the comparison is more complicated. Recall that
the bound $\Be(x)=P_2(y\tPi_{\si^2/y^2};x)$ is based on the comparison inequality \eqref{eq:Bef} over the class $\H2$ of generalized moment functions $f$, while the bound $\Pin(x)=P_3(\Ga_{(1-\vp)\si^2}+y\tPi_{\vp\si^2/y^2};x)$ is based on the comparison inequality \eqref{eq:PUfF3} over the class $\F3$, and the latter comparison is essentially equivalent to that over the class $\H3$, which is smaller than $\H2$ (by \eqref{eq:F-al-beta}). 
This is the factor that may make $\Be(x)$ better than $\Pin(x)$ (and hence better than $\PU(x)$) if $x$ is not so large; this factor will be especially significant when $\vp$ is close to $1$ and thus the role of the light-tail component $\Ga_{(1-\vp)\si^2}$ is negligible. However, as was noted in the Introduction concerning non-uniform Berry-Esseen type bounds, in typical applications when the $X_i$'s do not differ too much in distribution from one another, $\vp$ will be  close to $0$, rather than to $1$. The interplay between these two factors --- the presence of a light-tail component vs.\ the larger class of generalized moment functions --- is illustrated below. 

\begin{center}
\includegraphics[scale=0.6]{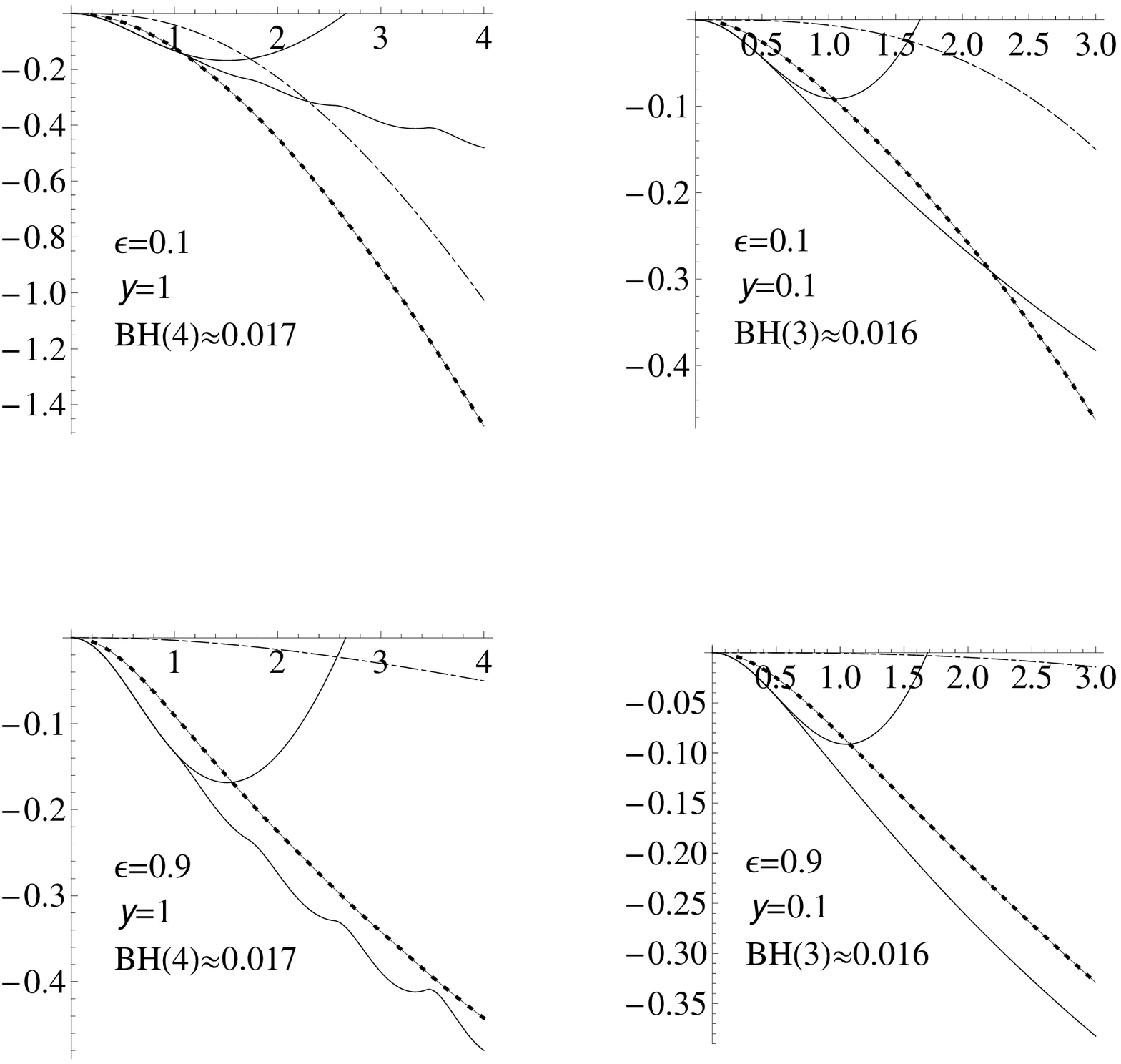}
\end{center}
Here, for $\si$ normalized to be $1$, and for $\vp\in\{0.1,0.9\}$ and $y\in\{0.1,1\}$, the graphs $G(P):=\{\big(x,\log_{10}\frac{P(x)}{\BH(x)}\big)\colon 0<x\le x_{\max}\}$ of the decimal logarithms of the ratios of the bounds $P=\Ca,\PU,\Be,\Pin$ to the benchmark Bennett-Hoeffding bound are shown, where $x_{\max}$ equals either $3$ or $4$, depending on whether $y=0.1$ (relatively little skewed-to-the-right summands $X_i$) or $y=1$ (relatively highly skewed-to-the-right summands $X_i$). 
The corresponding values of $\vp$, $y$, and $\BH(x_{\max})$ are shown for each of the four pictures. 
Note that, for such choices of $x_{\max}$, the values of $\BH(x_{\max})$ are approximately the same (about $2\%$), whether $y=0.1$ or $y=1$.

The graphs $G(P)$ for the bounds $P=\PU$ and $P=\Be$ are shown by the dot-dashed and solid lines, respectively; 
the graph $G(\Ca)$ too is shown by a solid line, but only on the interval $(0,u_y)$, on which $\Ca<\BH$, that is, $\log_{10}\frac{\Ca}{\BH}<0$ --- see Proposition~\ref{prop:ineqs}(IV). 
One can see, for $y=1$, $\Ca(x)$ is better than $\BH(x)$ for all $x\in(0,2.66)$. 
In accordance with Proposition~\ref{prop:ineqs}(I,II), the graph $G(\Ca)$ lies above $G(\Be)$ except that the two graphs coincide on the interval $[0,y]$, even though the graph $G(\Be)$ is seen to be very close $G(\Ca)$ well to the right of the interval $[0,y]=[0,0.1]$ for $y=0.1$.  
For the bound $\Pin$, actually two approximate graphs are shown: the one given by the thick dashed line was produced using formula \eqref{eq:laplace} (with $s=\ln(1+y)/y$ and $j=-1$) and the one given by the thin solid line was produced using formula \eqref{eq:char}; one can see that the two lines look practically the same -- as they should. (However, no other accuracy control of the performance of the Mathematica numerical integration command \verb9NIntegrate9 used to evaluate the integrals in \eqref{eq:laplace} and \eqref{eq:char} was done.) 
In fact, the graph for $\Pin$ was obtained via a ``parametric'' setting, as the set of the form 
$\{\big(x,\log_{10}\frac{\Pin(x)}{\BH(x)}\big)\colon x=m(t), t=u-1/u, 0.1\le u\le u_{\max}\}$, where the function $m$ is as in \eqref{eq:m} and $u_{\max}$ is the positive root $u$ of the equation $m(u-1/u)=x_{\max}$; this way, one have to solve the equation $m(t)=x$ in $t$ only for $x=x_{\max}$.

These pictures confirm the thesis that, if the weight $\vp$ of the heavy-tail Poisson component is relatively small, then the bound $\Pin(x)$ is significantly better (i.e., smaller) than $\Be(x)$ for (say) $x\ge3$. If $\vp$ is relatively large, then $\Be(x)$ may be slightly better than $\Pin(x)$ for moderate $x>0$ (say for $x<4$). Both $\Pin(x)$ and $\Be(x)$ are significantly better than the Bennett-Hoeffding bound $\BH(x)$, even for moderate $x>0$. The bound $\PU(x)$ is close to $\BH(x)$ for moderate $x>0$ if $\vp$ is close to $1$, which is in accordance with Proposition~\ref{prop:ineqs}(V). On the other hand, if the weight $\vp$ of the heavy-tail Poisson component is small while $y$ is large enough so that the Poisson component is quite distinct from the Gaussian component, then $\PU(x)$ is better than $\Be(x)$ even for such rather small $x$ as $x=2.5$. Here it is with more detail:
\begin{enumerate}[(i)]
	\item If the weight of the Poisson component is small ($\vp=0.1$) and the Poisson component is quite distinct from the Gaussian component ($y=1$), then $\Be(x)$ is about $9.93$ times worse (i.e., greater) than $\Pin(x)$ at $x=4$. Moreover, for these values of $\vp$ and $y$, even the bound $\PU(x)$ is better than $\Be(x)$ already at about $x=2.5$. 
	\item If the weight of the Poisson component is small ($\vp=0.1$) and the Poisson component is close to the Gaussian component ($y=0.1$), then $\Be(x)$ is still about $20\%$ greater than $\Pin(x)$ at $x=3$. 
	\item If the weight of the Poisson component is large ($\vp=0.9$) and the Poisson component is quite distinct from the Gaussian component ($y=1$), then $\Be(x)$ is about $8\%$ better than $\Pin(x)$ at $x=4$. For $x\in[0,4]$, $\Pin(x)$ and $\Be(x)$ are close to each other and both are significantly better than either $\BH(x)$ or $\PU(x)$ (which latter are also close to each other).  
	\item If the weight of the Poisson component is large ($\vp=0.9$) and the Poisson component is close to the Gaussian component ($y=0.1$), then $\Be(x)$ is about $12\%$ better than $\Pin(x)$ at $x=3$. For $x\in[0,3]$, $\Pin(x)$ and $\Be(x)$ are close to each other and both are significantly better than either $\BH(x)$ or $\PU(x)$ (which latter are \emph{very} close to each other).  
\end{enumerate}
In particular, we see that the latter two of the four enumerated cases are quite similar to each other. That is, if the weight of the Poisson component is large, then it does not matter much whether the Poisson component is close to the Gaussian component.


\textbf{A summary of the comparisons} made in this subsubsection and in the previous one is as follows. For all $x>0$, bounds $\Pin(x)$ and $\Be(x)$ are respectively better than the corresponding exponential bounds $\PU(x)$ and $\BH(x)$. For large $x>0$, each of the bounds $\Pin(x)$ and $\PU(x)$ is better than $\Be(x)$; the same may hold even for moderate $x>0$, especially when the weight $\vp$ of the Poisson component vs.\ the weight $1-\vp$ of the Gaussian one is relatively small; this is the case in typical applications. Otherwise, that is for relatively large $\vp\in(0,1)$ and moderate $x>0$, bound $\Be(x)$ may be a little better than $\Pin(x)$ and significantly better than $\PU(x)$. (On comparisons of bound $\BH(x)$ with previously known to Bennett bounds that show that $\BH(x)$ is superior to them, see \cite{bennett}.) Overall, the upper bound $\Pin(x)$ introduced in this paper usually outperforms the other three bounds: $\BH(x)$, $\PU(x)$, and $\Be(x)$. The minimum $\Pin(x)\wedge\Be(x)$ will in all cases be better (and usually significantly better) than $\PU(x)\wedge\BH(x)$.  

These relations are illustrated by the following diagram: 
$$
\begin{CD}
\BH @>r>> \PU\\
@VViV @VViV\\
\Be @>pr>> \Pin
\end{CD}
$$
In particular, it shows that $\PU$ is a refinement (denoted by $r$) of $\BH$. This refinement is also an improvement, as is obviously the case with any refinement that is exact in its own terms; indeed, the more specific the terms, the better the best possible result is; the usual downside of a refinement, though, is that it is more difficult to deal with: in terms of getting more specific information on the distributions of the $X_i$'s, as well as proving and computing the bound. Also, $\PU$ may be considered as a generalization of $\BH$, as $\BH$ may be considered as a special, limit case of $\PU$, with $\vp\to1$. 

The relation of $\Pin$ with $\Be$ is almost parallel to that of $\PU$ with $\BH$. However, the refinement (and hence the improvement and generalization) here are only partial ($pr$), because, as discussed, the class $\H3$ (corresponding to $\Pin$) is a bit smaller than $\H2$ (corresponding to $\Be$), even though, according to Propositions~\ref{prop:p=3} and \ref{prop:p=2}, $\H3$ is essentially the largest possible class for  $\Pin$, just as $\H2$ is for $\Be$. 

The relations of $\Be$ to $\BH$ and $\Pin$ to $\PU$ are pure improvements ($i$), due to using the larger classes $\H\al$ in place of the smaller class of exponential moment functions.

\section{Proofs}\label{proofs}

In Subsection~\ref{proofs1} of this section, we shall first state several lemmas; based on these lemmas, we shall provide the necessary proofs of results stated in Sections~\ref{results} and \ref{comp}. 
Proofs of the lemmas will be deferred to Subsection~\ref{proofs2}. 
We believe that such a structure will allow us to effectively present first the main ideas of the proofs and then the details. 

\subsection{Statements of lemmas, and proofs of theorems, corollaries, 
and propositions}\label{proofs1}

First here, let us state a few lemmas used in the proofs of Theorem~\ref{th:main} and Proposition~\ref{prop:exact}. We shall need more notation. 

Let $\si$ and $y$ be any (strictly) positive real numbers. 
For any pair of numbers $(a,b)$ such that $a\ge0$ and $b>0$, let $X_{a,b}$ denote any r.v.\ such that $X_{a,b}\sim\frac b{a+b}\de_{-a}+\frac a{a+b}\de_b$; that is, the distribution of $X_{a,b}$ is $\frac b{a+b}\de_{-a}+\frac a{a+b}\de_b$, the unique zero-mean distrubution on the two-point set $\{-a,b\}$; here and in what follows $\de_x$ stands, as usual, for the (Dirac) distribution concentrated at point $x$. 

\begin{lemma}\label{lem:1}
For all $x\in(-\infty,y]$,
\begin{equation*}
	x_+^3\le\frac{y^5}{(y^2+\si^2)^2}\,(x+\si^2/y)^2.
\end{equation*}
\end{lemma}

\begin{lemma}\label{lem:2}
Let $X$ be any r.v.\ such that $X\le y$ a.s., $\E X\le0$, and $\E X^2\le\si^2$. Then 
\begin{equation}\label{eq:2}
	\E X_+^3\le\frac{y^3\si^2}{y^2+\si^2}.
\end{equation}
\end{lemma}

\begin{lemma}\label{lem:3}
For any 
\begin{equation}\label{eq:be}
\be\in\Big(0,\frac{y^3\si^2}{y^2+\si^2}\Big]	
\end{equation}
there exists a unique pair $(a,b)\in(0,\infty)\times(0,\infty)$ such that $X_{a,b}\le y$ a.s., 
$\E X^2_{a,b}=\si^2$, and $\E(X_{a,b})_+^3=\be$; more specifically, $b$ is the only positive root of equation 
\begin{equation}\label{eq:b}
\si^2 b^3=\be(b^2+\si^2),	
\end{equation}
and 
\begin{equation}\label{eq:a}
	a=\frac{\si^2}b=\frac{\be b}{b^3-\be}.
\end{equation}  
\end{lemma}
In particular, Lemma~\ref{lem:3} implies that inequality \eqref{eq:2} is exact.

For any given $w\in\R$, $y>0$, $\si>0$, and $\be>0$, consider now the problem of finding the exact upper bound of $\E(X-w)_+^3$ over all r.v.'s $X$ satisfying the conditions $X\le y$ a.s., $\E X=0$, $\E X^2=\si^2$, and $\E X_+^3=\be$. At that, by Lemma~\ref{lem:2}, w.l.o.g.\ condition \eqref{eq:be} holds, since otherwise the corresponding set of r.v.'s $X$ is empty. 

\begin{lemma}\label{prop:ab}   
Fix any $w\in\R$, $y>0$, $\si>0$, and $\be$ satisfying condition \eqref{eq:be}, and let $(a,b)$ be the unique pair of numbers described in Lemma~\ref{lem:3}. Then 
\begin{align}
\sup&\{\E(X-w)_+^3\colon X\le y \text{ a.s.}, \E X=0, \E X^2=\si^2, \E X_+^3=\be\} \notag\\
=\max&\{\E(X-w)_+^3\colon X\le y \text{ a.s.}, \E X=0, \E X^2=\si^2, \E X_+^3=\be\} \label{eq:max_equal}\\
=\max&\{\E(X-w)_+^3\colon X\le y \text{ a.s.}, \E X\le0, \E X^2\le\si^2, \E X_+^3\le\be\} \label{eq:max_less}\\
=&\begin{cases}
\E(X_{a,b}-w)_+^3 \text{ if } w\le0,\\
\E(X_{\ta,\tb}-w)_+^3 \text{ if } w\ge0,
\end{cases} \label{eq:attain}
\end{align}
where
\begin{equation}\label{eq:tb,ta}
	\tb:=y\quad\text{and}\quad\ta:=\frac{\be y}{y^3-\be}
\end{equation}
(cf.\ \eqref{eq:a}). 
At that, $\ta>0$, $X_{\ta,\tb}\le y$ a.s., $\E X_{\ta,\tb}=0$, and $\E(X_{\ta,\tb})_+^3=\be$, but one can only say that $\E X_{\ta,\tb}^2\le\si^2$, and the latter  inequality is strict if $\be\ne\frac{y^3\si^2}{y^2+\si^2}$. 
\end{lemma}

Together with Lemma~\ref{th:basic} below, Lemma~\ref{prop:ab} represents one of the two most important steps in the proof of Theorem~\ref{th:main}. 

\begin{lemma}\label{lem:F3}
Let $\xi$ and $\eta$ be any real-valued r.v.'s such that $\E\xi\le\E\eta$, $\E\xi^2\le\E\eta^2<\infty$, and 
$\E f(\xi)\le\E f(\eta)$ for all $f\in\H3$. 
Then  
\begin{enumerate}[(i)]
	\item inequality $\E f(\xi)\le\E f(\eta)$ will hold for all $f\in\F3$;   
	\item if 
the condition 
$\E\xi\le\E\eta$ is replaced by 
$\E\xi=\E\eta$, then the inequality $\E f(\xi)\le\E f(\eta)$ will hold for all $f$ in the larger class ${\F3}_{,1}$, defined in Proposition~\ref{prop:F3larger};   
	\item 
	if the conditions 
$\E\xi\le\E\eta$ and $\E\xi^2\le\E\eta^2$ are both replaced by the equalities 
$\E\xi=\E\eta$ and $\E\xi^2=\E\eta^2$, then the inequality $\E f(\xi)\le\E f(\eta)$ will hold for all $f$ in the larger class ${\F3}_{,12}$; 	
\item however, it is \emph{not} enough to replace  
the condition 
$\E\xi^2\le\E\eta^2$ by the equality
$\E\xi^2=\E\eta^2$ for the inequality $\E f(\xi)\le\E f(\eta)$ to hold for all $f$ in the larger class ${\F3}_{,2}$ defined by removing $f'$ from the list ``$f,f',f'',f'''$'' in \eqref{eq:F3}.      
\end{enumerate}
\end{lemma}

\begin{lemma}\label{lem:F2}
Let $\xi$ and $\eta$ be any real-valued r.v.'s such that $\E\xi\le\E\eta$, $\E\xi^2\le\E\eta^2<\infty$, and 
$\E f(\xi)\le\E f(\eta)$ for all $f\in\H2$. 
Then  
\begin{enumerate}[(i)]
	\item inequality $\E f(\xi)\le\E f(\eta)$ will hold for all $f\in\F2$;   
	\item if 
the condition 
$\E\xi\le\E\eta$ is replaced by 
$\E\xi=\E\eta$, then the inequality $\E f(\xi)\le\E f(\eta)$ will hold for all $f$ in the larger class ${\F2}_{,1}$, defined in Proposition~\ref{prop:F2larger};   
	\item 
	if the conditions 
$\E\xi\le\E\eta$ and $\E\xi^2\le\E\eta^2$ are both replaced by the equalities 
$\E\xi=\E\eta$ and $\E\xi^2=\E\eta^2$, then the inequality $\E f(\xi)\le\E f(\eta)$ will hold for all $f$ in the larger class ${\F2}_{,12}$; 	
\item however, it is \emph{not} enough to replace 
the condition 
$\E\xi^2\le\E\eta^2$ by the equality
$\E\xi^2=\E\eta^2$ for the inequality $\E f(\xi)\le\E f(\eta)$ to hold for all $f$ in the larger class ${\F2}_{,2}$ defined by removing $f'$ from the list ``$f,f',f''$'' in \eqref{eq:F2}.      
\end{enumerate}
\end{lemma}

\begin{lemma}\label{lem:le}
Let $\si_0,\be_0,\si,\be$ be any real numbers such that $0\le\si_0\le\si$, $0\le\be_0\le\be$, $\be_0\le\si_0^2y$, and $\be\le\si^2y$. Then  
\begin{equation}\label{eq:le}
	\E f(\Ga_{\si_0^2-\be_0/y}+y\,\tPi_{\be_0/y^3})\le
	\E f(\Ga_{\si^2-\be/y}+y\,\tPi_{\be/y^3})
\end{equation}
for all $f\in\H2$, and hence for all $f\in\F2$ and for all $f\in\F3$. 
\end{lemma}

\begin{lemma}\label{th:basic}
Let $X$ be any r.v\ such that $X\le y$ a.s., $\E X\le0$, $\E X^2\le\si^2$, and $\E X_+^3\le\be$, where $\be$ satisfies condition \eqref{eq:be}. 
Then for all $f\in\F3$
\begin{equation}\label{eq:basic}
	\E f(X)\le
	\E f(\Ga_{\si^2-\be/y}+y\,\tPi_{\be/y^3}). 
\end{equation} 
\end{lemma}

\begin{lemma}\label{lem:exact} (Recall here the definition of $X_{a,b}$ in the beginning of Section~\ref{proofs}.) 
Fix any $\si>0$, $y>0$, and $\vp\in(0,1)$. 
Let then $\be:=\vp\si^2y$, in accordance with \eqref{eq:vp}. 
Then for each large enough $m\in\N$ there exist positive real numbers $a=a_m$ and $b=b_m$ such that the following statement is true: 
\begin{quotation}
{\normalsize \noindent 
if $n:=2m$ and $X_1,\dots,X_n$ are independent r.v.'s such that $X_1,\dots,X_m$ are independent copies of 
$X_{b/\sqrt m,b/\sqrt m}$ and $X_{m+1},\dots,X_{2m}$ are independent copies of 
$X_{a/m,y}$, then $X_1,\dots,X_n$ satisfy conditions \eqref{eq:ineqs}, with equalities in place of the first three inequalities there. 
} 
\end{quotation}
Moreover, then $S=X_1+\dots+X_n$ converges in distribution to $\Ga_{(1-\vp)\si^2}+y\tPi_{\vp\si^2/y^2}$ as $m\to\infty$. 
\end{lemma}

\begin{proof}[Proof of Theorem~\ref{th:main}]
Let $\si_i^2:=\E X_i^2$, $\be_i:=\E(X_i)_+^3$, 
$\si_0^2:=\sum_{i=1}^n\si_i^2$, $\be_0:=\sum_{i=1}^n\be_i$, $Y_i:=\Ga_{\si_i^2-\be_i/y}+y\tPi_{\be_i/y^3}$, and $T:=\sum_{i=1}^n Y_i$. 
Then, 
by a standard argument (cf.\ e.g.\ the proof of \cite[Theorem~2.1]{binom}) based on Lemma~\ref{th:basic}, one has
\begin{equation*}
	\E f(S)\le\E f(T)=\E f(\Ga_{\si_0^2-\be_0/y}+y\tPi_{\be_0/y^3})\quad\text{for all }f\in\F3. 
\end{equation*}
On the other hand, it is clear from \eqref{eq:ineqs} that $0\le\si_0^2\le\si^2$ and $0\le\be_0\le\be$; 
next, $\be_i\le\si_i^2 y$ for all $i=1,\dots,n$ and hence $\be_0\le\si_0^2 y$; also, 
by \eqref{eq:vp}, $\si^2-\be/y=(1-\vp)\si^2$ and $\be/y^3=\vp\si^2/y^2$. 
It remains to use Lemma~\ref{lem:le}.  
\end{proof}

\begin{proof}[Proof of Proposition~\ref{prop:F3larger}]
This follows by Lemma~\ref{lem:F3}(ii,iii). 
\end{proof}

\begin{proof}[Proof of Proposition~\ref{prop:exact}]
This follows by Lemma~\ref{lem:exact} and the Fatou lemma for convergence in distribution -- see e.g.\ \cite[Theorem~5.3]{bill}. 
\end{proof}

\begin{proof}[Proof of Proposition~\ref{prop:p=3}]
To obtain a contradiction, suppose that for some $p\in(0,3)$ one can replace $\H3$ in Corollary~\ref{cor:H3} by $\H p$. By \eqref{eq:F-al-beta}, w.l.o.g.\ $p\in(2,3)$. Take any $a\in(0,1)$ and introduce the new variable  
\begin{equation*}
	\tau:=\frac a{\sqrt{1+a}}.
\end{equation*}
Next, take any $n\in\N$ and let $X_1=X_{a,1}$ and $X_2=\dots=X_n=0$ a.s.\ (recall the definition of $X_{a,b}$ at the beginning of Section~\ref{proofs}). 
Then conditions \eqref{eq:vp} and \eqref{eq:ineqs} hold for $y=1$, $\si=\sqrt a$, and $\be=\frac a{1+a}$; at that,  $\vp=\frac1{1+a}$, $(1-\vp)\si^2=\tau^2$, $\vp\si^2=\frac a{1+a}$, and $-\vp\si^2+a=\tau^2$. 
Note that the function $x\mapsto f_{-a}(x):=(x+a)_+^p$ belongs to the class $\H p$. 
Consider 
\begin{equation*}
	\EE_1(a):=\E f_{-a}(S)\quad\text{and}\quad \EE_2(a):=\E f_{-a}(\Ga_{(1-\vp)\si^2}+\tPi_{\vp\si^2}),
\end{equation*}
respectively the left-hand side and the right-hand side of inequality \eqref{eq:PUfF3} with $f=f_{-a}$. 
Then 
\begin{equation}\label{eq:EE1}
	\EE_1(a)=\E(X_{a,1}+a)_+^p=\E(X_{a,1}+a)^p=a(1+a)^{p-1}=a+(p-1)a^2+o(a^2);
\end{equation}
here and in the rest of the proof of Proposition~\ref{prop:p=3}, the limit relations are understood as $a\downarrow0$. 

On the other hand,
\begin{equation}\label{eq:EE2}
	\EE_2(a)=\E(\Ga_{\tau^2}+\Pi_{a/(1+a)}+\tau^2)_+^p
	=\EE_{2,0}(a)+\EE_{2,1}(a)+\EE_{2,2}(a)+\EE_{2,\ge3}(a),
\end{equation}
where 
\begin{align*}
	\EE_{2,k}(a)&:=\PP(\Pi_{a/(1+a)}=k)\E(\Ga_{\tau^2}+\tau^2+k)_+^p \\
	&=\frac{e^{-a/(1+a)}}{k!}\frac{a^k}{(1+a)^k}\E(\tau Z+\tau^2+k)_+^p, \\
	\EE_{2,\ge3}(a)&:=\sum_{k=3}^\infty\EE_{2,k}(a),
\end{align*}
and $Z$ is a standard normal r.v. 
Note that $(\tau Z+\tau^2+k)_+^p=O(\tau^p|Z|^p+\tau^{2p}+k^p)$, whence 
\begin{equation}\label{eq:O}
0\le\E(\tau Z+\tau^2+k)_+^p=O(\tau^p+k^p)
\end{equation}
over all $k\ge0$. So, 
\begin{equation}\label{eq:EE20}
	\EE_{2,0}(a)=O(\tau^p)=O(a^p)=o(a^2), 
\end{equation}
since $p\in(2,3)$. Similarly using \eqref{eq:O}, it is easy to see that
\begin{equation}\label{eq:EE2ge3}
	\EE_{2,\ge3}(a)=O(a^3)=o(a^2). 
\end{equation}
By dominated convergence, $\E(\tau Z+\tau^2+2)_+^p=2^p+o(1)$. Hence, 
\begin{equation}\label{eq:EE22}
	\EE_{2,2}(a)=\frac{a^2}2 2^p(1+o(1))=2^{p-1}a^2+o(a^2). 
\end{equation}
To estimate $\EE_{2,1}(a)$, introduce 
$
	h(\tau,z):=\E(1+\tau z R+\tau^2)_+^p,
$ 
where $R:=X_{1,1}$ is a Rademacher r.v.\ which is independent of $Z$. Then $h(0,z)=1$, $h'_\tau(0,z)=0$, $|h''_\tau(\tau,z)|=O(|z|^p+1)$, and so, 
$\E(\tau Z+\tau^2+1)_+^p=\E h(\tau,Z)=1+O(\tau^2)=1+o(a)$, which implies   
\begin{equation}\label{eq:EE21}
	\EE_{2,1}(a)=e^{-a/(1+a)}\frac a{1+a}(1+o(a))=a-2a^2+o(a^2). 
\end{equation}
Thus, \eqref{eq:EE2}, \eqref{eq:EE20}, \eqref{eq:EE2ge3}, \eqref{eq:EE22}, and \eqref{eq:EE21} yield $\EE_1(a)=a+(2^{p-1}-2)a^2+o(a^2)$. So, recalling \eqref{eq:EE1}, one has 
$\EE_2(a)-\EE_1(a)=g(p)a^2+o(a^2)$, where $g(p):=2^{p-1}-1-p$. 
Observe that $g(2)=-1<0=g(3)$ and $g$ is a convex function, so that $g(p)<0$ for all $p\in(2,3)$. 
Therefore, the difference $\EE_2(a)-\EE_1(a)$ between the right-hand side of inequality \eqref{eq:PUfF3} (with $f=f_{-a}$) and its left-hand side is negative for small enough $a\in(0,1)$. 
This contradiction concludes the proof of Proposition~\ref{prop:p=3}. 
\end{proof}

\begin{proof}[Proof of Corollary~\ref{cor:improve}]
Let $\tilde S:=\sum_i X_i\ii{y_i>\si_i}$ and $\tsi:=\sqrt{\sum_i\si_i^2\ii{y_i>\si_i}}$, and let $\si$ be defined as in \eqref{eq:tbe}. 
Just as was noted concerning condition \eqref{eq:vp}, w.l.o.g.\ let us assume that $\check\vp:=\frac\tbe{\tsi^2y}\in(0,1)$. 
Then, by Theorem~\ref{th:main}, 
\begin{equation}\label{eq:ttPUfF3}
	\E f(x+\tS)\le\E f\big(x+\Ga_{(1-\check\vp)\tsi^2}+y\tPi_{\check\vp\tsi^2/y^2}\big) 
\end{equation}
for all $x\in\R$ and $f\in\F3$, 
since 
the class $\F3$ is obviously invariant with respect to the shifts $\F3\ni f(\cdot)\mapsto f(x+\cdot)$, for all $x\in\R$. 
On the other hand, by \cite[Corollary~1 with $p=\frac12$ and (10)]{asymm}, 
\begin{equation*}
	\label{eq:f(S-tS)}
	\E f(S-\tS+z)\le\E f\big(\Ga_{\si^2-\tsi^2}+z\big)
\end{equation*}
for all $z\in\R$ and $f\in\F3$, where one may assume that the r.v.\ $\Ga_{\si^2-\tsi^2}$ is independent of the r.v.'s $\tS$, $\Ga_{(1-\check\vp)\tsi^2}$, and $\tPi_{\check\vp\tsi^2/y^2}$ in \eqref{eq:ttPUfF3}. 
Using now \eqref{eq:ttPUfF3} and the independence of $S-\tS$ and $\tS$, for all $f\in\F3$ one has 
\begin{align*}
	\E f(S)&=\int_\R\E f(S-\tS+z)\PP(\tS\in\dd z) \\
	&\le\int_\R \E f\big(\Ga_{\si^2-\tsi^2}+z\big)\PP(\tS\in\dd z) \\
	&=\E f\big(\Ga_{\si^2-\tsi^2}+\tS\big) \\
	&=\int_\R\E f(x+\tS)\PP(\Ga_{\si^2-\tsi^2}\in\dd x) \\
	&\le\int_\R \E f\big(x+\Ga_{(1-\check\vp)\tsi^2}+y\tPi_{\check\vp\tsi^2/y^2}\big)\PP(\Ga_{\si^2-\tsi^2}\in\dd x) \\
		&=\E f\big(\Ga_{\si^2-\tsi^2}+\Ga_{(1-\check\vp)\tsi^2}+y\tPi_{\check\vp\tsi^2/y^2}\big) \\
		&=\E f\big(\Ga_{(1-\tvp)\si^2}+y\tPi_{\tvp\si^2/y^2}\big),
\end{align*}
since $\check\vp\tsi^2=\tvp\si^2$. 
Thus, inequality \eqref{eq:tPUfF3} is proved, which in turn implies inequalities \eqref{eq:tmain} and \eqref{eq:tmain LC} (cf.\ Corollary~\ref{cor:main}). 
\end{proof}

\begin{proof}[Proof of Proposition~\ref{prop:F2main}]
As noted in the Introduction, inequality \eqref{eq:PUfF2} for all $f$ of the form $f(x)\equiv(x-t)_+^2$ was obtained by Bentkus~\cite{bent-liet02,bent-ap}. By the Fubini theorem, one has \eqref{eq:PUfF2} for all $f\in\H2$. Then the extension to all $f\in\F2$ follows by Lemma~\ref{lem:F2}(i). 
\end{proof}

\begin{proof}[Proof of Proposition~\ref{prop:F2larger}]
This follows by Lemma~\ref{lem:F2}(ii,iii). 
\end{proof}

\begin{proof}[Proof of Proposition~\ref{prop:exact2}]
This proof is quite similar to (and even somewhat simpler than) that of Proposition~\ref{prop:exact}. 
\end{proof}

\begin{proof}[Proof of Proposition~\ref{prop:p=2}]
This proof is somewhat similar to but much simpler than that of Proposition~\ref{prop:p=3}. 
To obtain a contradiction, suppose that for some $p\in(0,2)$ one can replace $\H2$ in Corollary~\ref{cor:H2} by $\H p$. 
W.l.o.g.\ $p\in(1,2)$. Take any $n\in\N$ and let $X_1=X_{1,1}$ and $X_2=\dots=X_n=0$ a.s. 
Then for $y=\si=1$ and all $i$ one has $\E X_i=0$, $X_i\le y$ a.s., and $\sqrt{\sum_i\E X_i^2}=\si$. 
The function $x\mapsto f_{-1}(x):=(x+1)_+^p$ belongs to the class $\H p$. 
Then the left-hand side and the right-hand side of inequality \eqref{eq:PUfF2} with $f=f_{-1}$ are, respectively, 
$\EE_1(p):=\E(X_{1,1}+1)_+^p=2^{p-1}$ and $\EE_2(p):=\E\Pi_1^p$. 
Observe that the function $p\mapsto\EE_2(p)/\EE_1(p)$ is strictly convex on the interval $(1,2)$, and its values at the endpoints $1$ and $2$ of the interval are $1$. 
It follows that $\EE_2(p)/\EE_1(p)<1$ for all $p\in(1,2)$. 
\end{proof}

\begin{proof}[Proof of Proposition~\ref{prop:lambert}]
Take indeed any $\si>0$, $y>0$, $\vp\in(0,1)$, and $x\ge0$. 
Let for brevity $f:=\ln\PU_{\exp}$. Then, by the definition in \eqref{eq:PU}, 
$\PU(x)=\exp\inf_{\la\ge0}\big(-\la x+f(\la)\big)$. By the definition of $\ln\PU_{\exp}$ in \eqref{eq:PUexp}, 
\begin{equation}\label{eq:f'(la)}
	f'(\la)=\la(1-\vp)\si^2+\frac{e^{\la y}-1}y\,\vp\si^2, 
\end{equation}
which increases from $0$ to $\infty$ as $\la$ does so. 
Thus, there exists a unique root $\la=\la_x$ in $[0,\infty)$ of the equation $f'(\la)=x$, and $\la_x$ is the unique minimum point for $-\la x+f(\la)$ over all $\la\in[0,\infty)$, so that \eqref{eq:PU la_x} holds, with the so defined $\la_x$. It is also clear now that $\la_x=(f')^{-1}(x)$ increases from $0$ to $\infty$ as $x$ does so; that is, the last sentence of Proposition~\ref{prop:lambert} is verified. 

Next, rewrite the equation $f'(\la)=x$ as $e^{\la y}=\frac w\ka$ and then $we^w=\ka e^{(1+r)\ka}$, 
in terms of the new variable 
$w:=(1+r)\ka-\la y$, where $r:=\frac{xy}{\vp\si^2}$ and 
$\ka:=\frac\vp{1-\vp}$, so that $\la=\frac{(1+r)\ka-w}y$. 
Now one sees that $\la_x$ defined above in this proof as the unique root of equation $f'(\la)=x$ also satisfies definition \eqref{eq:la_x,w}. 

By \eqref{eq:PUexp}, 
\begin{equation}\label{eq:f(la)}
e^{-\la x}\PU_{\exp}(\la)=\exp\Big\{-\la x+\frac{\la^2}2\,(1-\vp)\si^2+\frac{e^{\la y}-1-\la y}{y^2}\,\vp\si^2\Big\}. 
\end{equation}
Now use again the mentioned equation $e^{\la y}=\frac w\ka$ to 
substitute $\frac w\ka$ for $e^{\la y}$ in the expression 
\eqref{eq:f(la)} and then substitute there $\frac{(1+r)\ka-w}y$ for $\la$. Then \eqref{eq:PUexpr} follows by simple algebra. 
\end{proof}

\begin{proof}[Proof of Proposition~\ref{prop:Th2.5,pin98}]\ 

(i) The continuity of $m$ on $(-\infty,x_*)$ follows by the condition $\E\eta_+^\al<\infty$ and dominated convergence. 

The left continuity of $m$ at $x_*$ follows by \eqref{eq:m} and the definition $m(x_*):=x_*$. Indeed, in view of the first expression for $m(t)$ in \eqref{eq:m}, 
\begin{equation}\label{eq:m(t)>t}
\text{$m(t)>t$ \quad for all $t\in(-\infty,x_*)$,} 	
\end{equation}
whence $m(t)\to\infty=x_*$ as $t\uparrow x_*$ in the case when $x_*=\infty$. Now if $x_*<\infty$ then, in view of the last expression for $m(t)$ in \eqref{eq:m}, 
$x_*-m(t)=\frac{\E(x_*-\eta)(\eta-t)_+^{\al-1}}{\E(\eta-t)_+^{\al-1}}\in[0,x_*-t]$ for all $t\in(-\infty,x_*)$, so that $m(t)\to x_*$ as $t\uparrow x_*$ in this case as well. 

That $m(t)=x_*$ for all $t\in[x_{**},x_*]$ also follows in view of the last expression for $m(t)$ in \eqref{eq:m}, taking also into account the definition of $x_{**}$, which implies that $(\eta-t)_+=(x_*-t)\ii{\eta=x_*}$ a.s.\ for all $t\in[x_{**},x_*]$. 

That the function $m$ is strictly increasing on $(-\infty,x_{**})$, with $m((-\infty)+)=\E\eta$, follows immediately from parts (i) and (ii) of \cite[Theorem~2.5]{pin98}. 
This completes the proof of part (i) of Proposition~\ref{prop:Th2.5,pin98}. 

(ii) Part (ii) of Proposition~\ref{prop:Th2.5,pin98} follows immediately from its part (i), taking also into account \eqref{eq:m(t)>t}. 

(iii) The first equality in part (iii) of Proposition~\ref{prop:Th2.5,pin98} follows immediately from part (iv) of \cite[Theorem~2.5]{pin98}; the second equality follows by \eqref{eq:m(t)=x} and \eqref{eq:m}.  
(The natural condition $x<x_*$ was missing in parts (iii) and (iv) of \cite[Theorem~2.5]{pin98}; thanks are due to Bentkus for having drawn my attention to that omission.) 

(iv) Part (iv)(a) of Proposition~\ref{prop:Th2.5,pin98} follows from the last sentence of \cite[Theorem~2.5]{pin98} and Proposition~\ref{prop:P_al=inf}, to be proved next. 

Let us now verify part (iv)(b). Take indeed any $x\in[x_*,\infty)$. 
Then for all $t\in(-\infty,x)$, by 
the already proved part (i) of Proposition~\ref{prop:Th2.5,pin98}, one has $m(t)\le m(x_*)=x_*\le x$, 
and so, by the second displayed formula on \cite[page~302]{pin98}, 
\begin{equation}\label{eq:F(t,x)}
	F(t,x):=\frac{\E(\eta-t)_+^\al}{(x-t)^\al}
\end{equation}
is nonincreasing in $t\in(-\infty,x)$. 
Recalling now \eqref{eq:comp-prob2} and taking also into account that $\eta\le x$ a.s.\ for all $x\in[x_*,\infty)$, one sees that 
\begin{align}
	P_\al(\eta;x)&=\inf_{t\in(-\infty,x)}\,F(t,x) \label{eq:P,F}\\
	&=\lim_{t\uparrow x}\,F(t,x)
=\lim_{t\uparrow x}\,\frac{\E(\eta-t)^\al\ii{\eta\in(t,x]}}{(x-t)^\al}
=\PP(\eta=x)=\PP(\eta\ge x),  
\end{align}
since $0\le\lim_{t\uparrow x}\,\frac{\E(\eta-t)^\al\ii{\eta\in(t,x)}}{(x-t)^\al}\le\lim_{t\uparrow x}\,\E\ii{\eta\in(t,x)}=0$. 
This completes the proof of part (iv) of Proposition~\ref{prop:Th2.5,pin98}. 

(v)(a)  
Take any $x$ and $y$ such that $\E\eta<x<y<x_*$. Then 
$F(t,x)>F(t,y)$ for each $t<(-\infty,x)$. Hence, by \eqref{eq:P(x)}, \eqref{eq:F(t,x)}, and \eqref{eq:P,F}, $P_\al(\eta;x)=F(t_x,x)>F(t_x,y)\ge P_\al(\eta;y)$. This proves part (v)(a) of Proposition~\ref{prop:Th2.5,pin98}. 

(v)(b) By parts (i) and (ii) of Proposition~\ref{prop:Th2.5,pin98}, the function $x\mapsto t_x$ is continuous on $(\E\eta,x_*)$. 
Also, $\E(\eta-t)_+^\al$ and $\E(\eta-t)_+^{\al-1}$ are continuous in $t\in\R$, by the condition $\E\eta_+^\al<\infty$ and dominated convergence. Hence, in view of the last expression in \eqref{eq:P(x)}, the function $x\mapsto P_\al(\eta;x)$ is continuous on $(\E\eta,x_*)$. 

Consider now the right continuity at $\E\eta$. Let $x\downarrow\E\eta$. Then, by parts (i) and (ii) of Proposition~\ref{prop:Th2.5,pin98}, $t_x\to-\infty$. 
If $\E\eta>-\infty$ then, by the condition $\E\eta_+^\al<\infty$, one has $\E\eta\in\R$, so that $x-t_x\sim\E\eta-t_x\sim-t_x$. 

Let us show that the conclusion that $x-t_x\sim-t_x$ holds when $\E\eta=-\infty$. Note that 
$\frac{(\eta-t)_+^\al}{(-t)^\al}\le(1+\eta_+)^\al$ for all $t\le-1$; hence, by dominated convergence, 
\begin{equation}\label{eq:simi}
\E(\eta-t)_+^\al\sim(-t)^\al	
\end{equation}
and, similarly,  $\E(\eta-t)_+^{\al-1}\sim(-t)^{\al-1}$ as $t\to-\infty$. 
So, by \eqref{eq:m}, $m(t)=t+(-t)(1+o(1))=o(|t|)$  as $t\to-\infty$. It follows by part (ii) of Proposition~\ref{prop:Th2.5,pin98} that $x=o(t_x)$ (as $x\downarrow\E\eta$), which indeed implies $x-t_x\sim-t_x$, even in the case when $\E\eta=-\infty$. 
Therefore, by the first equality in \eqref{eq:P(x)}, \eqref{eq:simi}, and part (iv) of Proposition~\ref{prop:Th2.5,pin98}, 
$$P_\al(\eta;x)
=\frac{\E(\eta-t_x)_+^\al}{(x-t_x)^\al}\sim\frac{\E(\eta-t_x)_+^\al}{(-t_x)^\al}\to1=P_\al(\eta;\E\eta)$$ 
as $x\downarrow\E\eta$, which concludes the proof of the right continuity at $\E\eta$.  

To complete the proof of part (v)(b) of Proposition~\ref{prop:Th2.5,pin98}, it remains to verify the left continuity at $x_*$. The easier case here is when $x_*=\infty$; then 
$$0\le P_\al(\eta;x)\le F(x-1,x)=\E(\eta-x+1)_+^\al \to 0=P_\al(\eta;\infty)=P_\al(\eta;x_*)$$
as $x\to\infty=x_*$, 
by the definitions \eqref{eq:comp-prob2} and \eqref{eq:F(t,x)} of $P_\al(\eta;x)$ and $F(t,x)$, the condition $\E\eta_+^\al<\infty$ and dominated convergence, and part (iv) of Proposition~\ref{prop:Th2.5,pin98}.  

Assume now that $x_*<\infty$. Let $x\uparrow x_*$. Introduce $\tilde t_x:=x-\sqrt{x_*-x}$, so that $\tilde t_x<x$, $\tilde t_x\uparrow x_*$, $x_*-\tilde t_x\sim x-\tilde t_x$, and hence
$$0\le\frac{\E(\eta-\tilde t_x)^\al\ii{\eta\in(\tilde t_x,x_*)}}{(x-\tilde t_x)^\al}
\le\Big(\frac{x_*-\tilde t_x}{x-\tilde t_x}\Big)^\al\E\ii{\eta\in(\tilde t_x,x_*)}\to0,$$
which in turn implies 
$$F(\tilde t_x,x)=\frac{\E(\eta-\tilde t_x)^\al\ii{\eta\in(\tilde t_x,x_*]}}{(x-\tilde t_x)^\al}
\to\PP(\eta=x_*)=P_\al(\eta;x_*),$$
by part (iv)(b) of Proposition~\ref{prop:Th2.5,pin98}. 
It is clear from the definition \eqref{eq:comp-prob2} of $P_\al(\eta;x)$ that $P_\al(\eta;x)\ge P_\al(\eta;y)$ whenever  $-\infty<x<y<\infty$. Hence, recalling again the definition \eqref{eq:F(t,x)} of $F(t,x)$, one has  
$$P_\al(\eta;x_*)\le\lim_{x\uparrow x_*}P_\al(\eta;x)\le\lim_{x\uparrow x_*}F(\tilde t_x,x)
=P_\al(\eta;x_*),$$
which implies the left continuity at $x_*$. 
Thus, part (v)(b) of Proposition~\ref{prop:Th2.5,pin98} is completely proved.  

(vi) Part (vi) of Proposition~\ref{prop:Th2.5,pin98} follows immediately from the definitions \eqref{eq:comp-prob2},  \eqref{eq:m(t)=x}, and \eqref{eq:m}.

\end{proof}

\begin{proof}[Proof of Proposition~\ref{prop:P_al=inf}]
For brevity, let us denote the infima in \eqref{eq:P_al=inf1} and \eqref{eq:P_al=inf2} by $\inf_1$ and $\inf_2$, respectively. Take any $x\in\R$. 

Then $\inf_2\le P_\al(\eta;x)$, because the function $u\mapsto(u-t)_+^\al$ is in $\H\al$ for every $t\in\R$. 
If $\inf_2<P_\al(\eta;x)$, then there is some $f\in\H\al$ such that $f(x)>0$ and 
$\frac{\E f(\eta)}{f(x)}<\frac{\E(\eta-t)_+^\al}{(x-t)^\al}$ for all $t\in(-\infty,x)$;  
but, by \eqref{eq:H}, 
$f(u)=\int_{(-\infty,u)}(u-t)^\al\,\mu(dt)$ for some nonnegative measure $\mu$ and all $u\in\R$; at that, $\mu\big((-\infty,x)\big)\ne0$, since $f(x)>0$;  
so, 
$$\E f(\eta)=\frac{\E f(\eta)}{f(x)}\int_{(-\infty,x)}(x-t)^\al\,\mu(dt)
<\int_{(-\infty,x)}\E(\eta-t)_+^\al\,\mu(dt)\le\E f(\eta),$$
by the Fubini theorem. This contradiction shows that $\inf_2=P_\al(\eta;x)$. 

It remains to show that $\inf_1=\inf_2$. Take any $f\in\H\al$ such that $f(x)>0$, and let $g:=g_f:=\frac f{f(x)}$. 
Then $\frac{\E f(\eta)}{f(x)}=\E g(\eta)$, 
$g\in\H\al$, $g$ is nonnegative and nondecreasing, and $g(x)=1$. It follows that $g(u)\ge\ii{u\ge x}$ for all $u\in\R$. Thus, $\inf_1\le\inf_2$. 

Vice versa, take any $f\in\H\al$ such that $f(u)\ge\ii{u\ge x}$ for all $u\in\R$. Then $f(x)\ge1$, and so, 
$f\ge\frac f{f(x)}=g\in\H\al$ and $g(x)=1$. Hence, 
$\E f(\eta)\ge\E g(\eta)=\frac{\E g(\eta)}{g(x)}$, which implies $\inf_1\ge\inf_2$. 
So, indeed $\inf_1=\inf_2$. 
\end{proof}

\begin{proof}[Proof of Proposition~\ref{prop:H_infty}]
The first equality in \eqref{eq:H_infty} follows easily from \eqref{eq:F-al-beta} and Proposition~\ref{prop:F-al}; indeed, for any convex $f\colon\R\to\R$ such that $f(-\infty+)=0$ one has $f\ge0$ and $f'\ge0$ on $\R$. 
As for the second equality in \eqref{eq:H_infty}, it follows by the Bernstein theorem on completely monotone functions (see, e.g., \cite{choquet} or \cite{phelps}) and the fact that the Laplace transform of a measure uniquely characterizes the measure --- cf.\ Remark~3.5 in \cite{asymm_arxiv}.  
Indeed, take any $f\in\H\infty$. Then 
for each $w\in[0,\infty)$ the function $(-\infty,0)\ni x\mapsto f_w(x)=f(x+w)$ is completely monotone, in the sense that $f_w^{(j)}\ge0$ on $\R$ for all $j=0,1,\dots$; hence, there exists a unique nonnegative Borel measure $\mu_w$ on $[0,\infty)$ such that 
$f(x+w)=\int_{[0,\infty)}e^{tx}\mu_w(\dd t)$ for all $x\in(-\infty,0)$ 
or, equivalently, 
\begin{equation}\label{eq:bernstein}
	\text{$f(u)=\int_{[0,\infty)}e^{tu}e^{-tw}\mu_w(\dd t)$}
\end{equation}
for all $u\in(-\infty,w)$, 
and hence for all $u\in(-\infty,0)$ 
(see e.g.\ \cite[Ch.\ 2,\S2]{phelps}); in fact, one must have $\mu(\{0\})=0$, since $f(-\infty+)=0$. 
In particular, identity \eqref{eq:bernstein} holds for all $u\in(-\infty,0)$ with $\mu_0(\dd t)$ in place of $e^{-tw}\mu_w(\dd t)$. By the uniqueness of the measure, one has $f(u)=\int_{(0,\infty)}e^{tu}\mu_0(\dd t)$ for all $w\in[0,\infty)$ and all $u\in(-\infty,w)$, and hence for all $u\in\R$. 
By dominated convergence, now one also obtains the condition $f^{(j)}(-\infty+)=0$ for all $j=0,1,\dots$.  
\end{proof}

\begin{proof}[Proof of Proposition~\ref{prop:P_infty}] 
Similarly to \eqref{eq:P_al=inf1}-\eqref{eq:P_al=inf2}, one has 
\begin{equation}\label{eq:P_infty=}
	P_\infty(\eta;x)=\inf_{f\in\H{\exp}}\,\frac{\E f(\eta)}{f(x)} 
\end{equation}
for all $x\in\R$. 
Indeed, by definition \eqref{eq:P_infty def} and because the class $\H{\exp}$ contains all increasing exponential functions, the right-hand side of \eqref{eq:P_infty=} is no greater than its left-hand side, $P_\infty(\eta;x)$. 
To complete the proof of inequality \eqref{eq:P_infty=}, take any $f\in\H{\exp}$ and any $x\in\R$. Then $f(u)=\textstyle{\int_{(0,\infty)}}e^{tu}\mu(\dd t)$ for all $u\in\R$, where $\mu\ge0$ is some Borel measure. So, by the Fubini theorem and \eqref{eq:P_infty def}, 
\begin{equation}
	\E f(\eta)=\int_{(0,\infty)}\E e^{t\eta}\mu(\dd t)
	\ge\int_{(0,\infty)}P_\infty(\eta;x)e^{tx}\mu(\dd t)
	=P_\infty(\eta;x)f(x),
\end{equation}
which shows that the right-hand side of \eqref{eq:P_infty=} is no less than its left-hand side, $P_\infty(\eta;x)$. 
Thus, \eqref{eq:P_infty=} is verified. 
On the other hand, by Proposition~\ref{prop:H_infty}, \eqref{eq:F-al-beta}, and \eqref{eq:P_al=inf1}-\eqref{eq:P_al=inf2}, 
\begin{equation}
	\inf_{f\in\H{\exp}}\,\frac{\E f(\eta)}{f(x)}
	=\lim_{\al\uparrow\infty}\inf_{f\in\H\al}\,\frac{\E f(\eta)}{f(x)}
	=\lim_{\al\uparrow\infty}P_\al(\eta;x).
\end{equation}
This, together with \eqref{eq:P_infty=}, yields \eqref{eq:P_infty}.  

That the function $\al\mapsto P_\al(\eta;x)$ is nondecreasing on $(0,\infty)$ follows immediately by \eqref{eq:P_al=inf1}-\eqref{eq:P_al=inf2} and \eqref{eq:F-al-beta}; that this function is nondecreasing on $(0,\infty]$ now follows by \eqref{eq:P_infty}. 
\end{proof}

\begin{proof}[Proof of Proposition~\ref{prop:Ca=inf}]
Take indeed any $\si\in(0,\infty)$, any $x\in[0,\infty)$, and any 
r.v.'s\ $\xi$ and $\eta$ such that $\E\xi\le0=\E\eta$ and $\E\xi^2\le\E\eta^2=\si^2$. 
Let $f(t):=\frac{\si^2+t^2}{(x-t)^2}$. 
Then for any $t<0$ one has 
$
	(x-t)^2\PP(\xi\ge x)\le\E(\xi-t)^2=\E\xi^2+2|t|\E\xi+t^2\le\si^2+t^2
$, 
whence $\PP(\xi\ge x)\le\inf_{t<0}f(t)=\inf_{t\in(-\infty,x)}f(t)=f(-\si^2/x)=\Ca(x)$. 
\end{proof}

\begin{proof}[Proof of Proposition~\ref{prop:ineqs}]\ 

\noindent(I)\quad Inequalities $\Pin(x)\le\PU(x)$ and $\Be(x)\le\BH(x)$ follow because for each $\al>0$ the class $\H\al$ contains the class of all increasing exponential functions, taking at that into account
the expression \eqref{eq:P_al=inf2} for $P_\al(\eta;x)$, 
the definitions of $\Pin(x)$, $\PU(x)$, and $\Be(x)$ in \eqref{eq:main}, \eqref{eq:PU}, and \eqref{eq:Be}, the expressions for $\BH(x)$, $\BH_{\exp}(\la)$, and $\PU_{\exp}(\la)$ in \eqref{eq:BH}, \eqref{eq:BHexp-expr}, and \eqref{eq:PUexp-expr}. 
As for the inequality $\PU(x)\le\BH(x)$, it follows, as discussed in the Introduction after \eqref{eq:PU}, because $\PU_{\exp}\le\BH_{\exp}$. 
Inequality $\Be(x)\le\Ca(x)$ follows by \eqref{eq:Be} and the expressions \eqref{eq:comp-prob2} and \eqref{eq:Caexpr} for $P_\al(\eta;x)$ and $\Ca(x)$, because obviously $\E(\eta-t)_+^2\le\E(\eta-t)^2$ for any $\eta$ and $t$. 

\noindent(II)\quad The identity in part (II) of Proposition~\ref{prop:ineqs} follows by \cite[(10.5)]{bent-64pp} and \eqref{eq:Ca}. 

\noindent(III)\quad  
Applying twice the special l'Hospital-type rule for monotonicity (as well as the l'Hospital rule for limits), one sees that the ratio $\frac{\psi(u)}{u^2}$ decreases from $\frac12$ to $0$ as $u$ increases from $0$ to $\infty$, where the function $\psi$ is defined by \eqref{eq:psi}. Now part (III) of Proposition~\ref{prop:ineqs} follows.  

\noindent(IV)\quad By re-scaling, w.l.o.g.\ $\si=1$. Consider the function 
\begin{equation}
	d(x):=\frac1{\Ca(x)}-\frac1{\BH(x)}=
	1 + x^2 - \exp\frac{\psi(x y)}{y^2}. 
\end{equation}
Let $d_3(x):= d'''(x) y^3/e^{\psi(x y)/y^2}$ and $d_4(x):= d_3'(x) (1 + x y)/(-3y)$. 
Then $d_4(x)$ is a monic quadratic polynomial in $\ln(1+xy)$ with (coefficients being rational functions of $x$ and $y$, and) a negative discriminant, so that $d_4(x)>0$ and hence $d'_3(x)<0$ for all $x\ge0$. 
So, $d_3$ decreases on $[0,\infty)$ from $d_3(0)=y^4>0$ to $d_3(\infty-)=-\infty<0$. 
Hence, $d'''$ is $+-$ on $[0,\infty)$; that is, 
$d'''(x)$ switches in sign from $+$ to $-$ as $x$ increases from $0$ to $\infty$. 
Thus, $d''$ is up-down on $[0,\infty)$; that is, switches from increase to decrease on $[0,\infty)$. 
Since $d''(0)=1>0$ and $d''(\infty-)=-\infty$, one sees that $d''$ is $+-$ on $[0,\infty)$. 
Since $d'(0)=0$ and $d'(\infty-)=-\infty$, one sees that $d'$ is $+-$ on $(0,\infty)$. 
Since $d(0)=0$ and $d(\infty-)=-\infty$, one sees that $d$ is $+-$ on $(0,\infty)$ as well. 
This proves the existence of a unique $u_y$ in $(0,\infty)$ such that $\Ca(x)<\BH(x)$ for $x\in(0,u_y)$ and $\Ca(x)>\BH(x)$ for $x\in(u_y,\infty)$. 
That $u_y$ increases from $u_{0+}=1.585\dots$ to $\infty$ as $y$ increases from $0$ to $\infty$ now follows by part (III) of Proposition~\ref{prop:ineqs}, since $\Ca(x)$ does not depend on $y$. 
   
\noindent(V)\quad Take any $x>0$. By \eqref{eq:PUexp}, 
$
\PU_{\exp}(\la)=\exp\big\{\frac{\la^2}2\,\si^2+\frac{e^{\la y}-1-\la y-\la^2y^2/2}{y^2}\,\vp\si^2\big\}
$ 
strictly increases from $\exp\big\{\frac{\la^2}2\,\si^2\big\}=\E\exp\{\la\Ga_{\si^2}\}$ to $\exp\frac{e^{\la y}-1-\la y}{y^2}\,\si^2=\BH_{\exp}(\la)$
as $\vp$ increases from $0$ to $1$. 
So, in view of \eqref{eq:PU}, $\PU(x)$ is nondecreasing in $\vp\in(0,1)$. 

Moreover, $\PU(x)$ is strictly increasing in $\vp\in(0,1)$, from $\EN(x)$ to $\BH(x)$; this follows by \eqref{eq:PU la_x}. 
Indeed, $\la_x$ is the only positive root of the equation $f'(\la)=x$, where (see \eqref{eq:f'(la)}) 
$f'(\la)=\la\si^2+\frac{e^{\la y}-1-\la y}y\,\vp\si^2$ is strictly increasing in $\la>0$ and in $\vp\in[0,1]$. 
So, the unique root $\la_x$ of the equation $f'(\la)=x$ is decreasing in $\vp\in[0,1]$, from $\frac x{\si^2}<\infty$ to $\frac1y\ln(1+xy/\si^2)>0$, so that $\la_x>0$ remains bounded away from $0$ and $\infty$ as $\vp$ increases from $0$ to $1$.  

Now part (V) of Proposition~\ref{prop:ineqs} follows, and the entire Proposition~\ref{prop:ineqs} is proved.   
\end{proof}

\begin{proof}[Proof of Proposition~\ref{prop:idents}]
In view of \eqref{eq:f'(la)}, for $f:=\ln\PU_{\exp}$ and each $x>0$ the equation $f'(\la_x)=x$ implies that 
there is some $\al_x\in(0,1)$ such that 
\begin{equation}\label{eq:la_x,al_x}
	\la_x(1-\vp)\si^2=(1-\al_x)x\quad\text{and}\quad\frac{e^{\la_x y}-1}y\,\vp\si^2=\al_x\,x, 
\end{equation}
whence 
\begin{equation}\label{eq:la_x12}
	\la_x=\la_{x,1}:=\frac{(1-\al_x)x}{(1-\vp)\si^2}
\quad\text{and}\quad
\la_x=\la_{x,2}:=\frac1y\ln\Big(1+\frac{\al_x xy}{\vp\si^2}\Big).  
\end{equation}
On the other hand, introducing $f_1(\la):=-\la(1-\al_x)x+\frac{\la^2}2\,(1-\vp)\si^2$ and 
$f_2(\la):=-\la\al_x x+\frac{e^{\la y}-1-\la y}{y^2}\,\vp\si^2$, 
by \eqref{eq:PU la_x} and \eqref{eq:f(la)} 
one has 
\begin{align*}
	\ln\PU(x)=-\la_x x+f(\la_x)&=f_1(\la_x)+f_2(\la_x)=f_1(\la_{x,1})+f_2(\la_{x,2})
	=g(\al_x),  
\end{align*}
where 
\begin{equation}\label{eq:g(al)}
	g(\al):=-\frac{(1-\al)^2x^2}{2(1-\vp)\si^2}
	-\frac{\vp\si^2}{y^2}\,\psi\Big(\frac{\al xy}{\vp\si^2}\Big)
	=\ln\Big(\Ne_{(1-\vp)\si^2}((1-\al)x)\BH_{\vp \si^2,y}(\al x)\Big)
\end{equation}
and $\psi$ is defined by \eqref{eq:psi}; thus, the expression in \eqref{eq:max at al_x} equals $\PU(x)$.  
The derivative
\begin{equation}\label{eq:g'}
	g'(\al)=\frac{(1-\al)x^2}{(1-\vp)\si^2}
	-\frac xy\,\ln\Big(1+\frac{\al xy}{\vp\si^2}\Big)
\end{equation}
decreases from $g'(0)>0$ to $g'(1)<0$ as $\al$ increases from $0$ to $1$. Hence, there is a unique maximum point of $g$ in $[0,1]$ (say $\tilde\al_x$). Moreover, $\al=\tilde\al_x$ must be the unique root in $(0,1)$ of the equation $g'(\al)=0$, which is the same as \eqref{eq:al_x}. But, by \eqref{eq:g'} and \eqref{eq:la_x12}, this equation is satisfied by $\al=\al_x\in(0,1)$. Thus, $\tilde\al_x=\al_x$. 
This completes the proof of \eqref{eq:max al}, \eqref{eq:max at al_x}, and \eqref{eq:al_x}. 

Finally, it follows from \eqref{eq:la_x,al_x} that $\frac{\al_x}{1-\al_x}=\frac\vp{1-\vp}\frac{e^{\la_x y}-1}{\la_x y}$, which increases in $\la_x$ from $\frac\vp{1-\vp}$ to $\infty$ as $\la_x$ increases from $0$ to $\infty$, which it does  (according to the last sentence in Proposition~\ref{prop:lambert}) as $x$ increases from $0$ to $\infty$. Now the entire proposition is proved. 
\end{proof}

\begin{proof}[Proof of Proposition~\ref{prop:mono in y}]
This follows immediately from Lemma~\ref{lem:le} with $\si_0=\si$, $\be_0=\vp\si^2y$, and $\be=\si^2y$.  
\end{proof}

\begin{proof}[Proof of Proposition~\ref{prop:PLC(Po>x)}]
This follows because of the known fact that the function $\Z\ni j\mapsto\PP(\Pi_\th\ge j)$ is log-concave (see e.g.\ \cite[Theorem~1 and Remark~13]{pin99}), in the sense of \cite[Definition~1]{pin99}.  
\end{proof}

\begin{proof}[Proof of Proposition~\ref{prop:PLC(Ga+Po>x)}]
The right-hand side of \eqref{eq:PLC(Ga+Po>x)} with $\PP^\lc(y\tPi_{\vp\si^2/y^2}\ge z)$ replaced by $\PP(y\tPi_{\vp\si^2/y^2}\ge z)$ would equal $\PP(\Ga_{(1-\vp)\si^2}+y\tPi_{\vp\si^2/y^2}\ge x)$. So, since $\PP^\lc(y\tPi_{\vp\si^2/y^2}\ge z)$ majorizes $\PP(y\tPi_{\vp\si^2/y^2}\ge z)$, the right-hand side of \eqref{eq:PLC(Ga+Po>x)} majorizes $\PP(\Ga_{(1-\vp)\si^2}+y\tPi_{\vp\si^2/y^2}\ge x)$. Also, the right-hand side of \eqref{eq:PLC(Ga+Po>x)} is log-concave in $x\in\R$ by the well-known theorem, which states that $\int_\R f(x,z)\dd z$ is log-concave in $x\in\R$ if a function $\R^2\ni(x,z)\mapsto f(x,z)$ is log-concave (see e.g.\ \cite{prek} as well as the corresponding review by Perlman in Mathematical Reviews); 
here we also used the obvious fact that any normal density function is log-concave. 
This concludes the proof of Proposition~\ref{prop:PLC(Ga+Po>x)}.  
\end{proof}

\begin{proof}[Proof of Proposition~\ref{prop:PU<<BH}]
All the limit relations in this proof are of course as $x\to\infty$, unless specified otherwise. 
By Proposition~\ref{prop:idents}, $\al_x\to1$. Equations \eqref{eq:la_x12} allow one to qualify the rate of convergence of $\al_x$ to $1$. Indeed, 
\begin{equation*}
	\ln\Big(1+\frac{\al_x xy}{\vp\si^2}\Big)\sim\ln\frac{\al_x xy}{\vp\si^2}=\ln\frac{xy}{\vp\si^2}+\ln\al_x 
	\sim\ln\Big(1+\frac{xy}{\vp\si^2}\Big),
\end{equation*}
whence, by \eqref{eq:la_x12}, 
\begin{equation}\label{eq:1-al_x}
	1-\al_x=\frac{(1-\vp)\si^2}{xy}\ln\Big(1+\frac{\al_x xy}{\vp\si^2}\Big)
	\sim\frac{(1-\vp)\si^2}{xy}\ln\Big(1+\frac{xy}{\vp\si^2}\Big).  
\end{equation}
Now one can see that 
\begin{align*}
	\ln\Big(1+\frac{\al_x xy}{\vp\si^2}\Big)-\ln\Big(1+\frac{xy}{\vp\si^2}\Big)
	&=\ln\Big(1+\frac{(\al_x-1)xy}{xy+\vp\si^2}\Big)
	\sim\al_x-1 \\
	&\sim-\frac{(1-\vp)\si^2}{xy}\ln\Big(1+\frac{xy}{\vp\si^2}\Big),
\end{align*}
and so, again by \eqref{eq:la_x12}, 
\begin{align}
	\la_x&=\frac1y\ln\Big(1+\frac{xy}{\vp\si^2}\Big)\Big(1-\frac{(1-\vp)\si^2}{xy}(1+o(1))\Big); \notag\\
	1-\al_x&=\frac{(1-\vp)\si^2}{xy}\ln\Big(1+\frac{xy}{\vp\si^2}\Big)\Big(1-\frac{(1-\vp)\si^2}{xy}(1+o(1))\Big); \notag\\
	-\frac{(1-\al_x)^2x^2}{2(1-\vp)\si^2}
	&=-\frac{(1-\vp)\si^2}{2y^2}\ln^2\Big(1+\frac{xy}{\vp\si^2}\Big)\Big(1-2\frac{(1-\vp)\si^2}{xy}(1+o(1))\Big) \notag\\
	&=-\frac{(1-\vp)\si^2}{2y^2}\ln^2\Big(1+\frac{xy}{\vp\si^2}\Big)+o(1). 	\label{eq:la_x,o(1)}
\end{align}
Next, with the same $\psi(u)=(1+u)\ln(1+u)-u$ as in \eqref{eq:psi}, one has 
$\psi'(u)=\ln(1+u)$ and $\psi''(u)=\frac1{1+u}$ for $u>0$, whence $\psi(u)-\psi(v)=(u-v)\ln(1+v)+O((u-v)^2/u)$ as $v>u\to\infty$. Hence and view of \eqref{eq:1-al_x}, 
\begin{equation}\label{eq:sq}
	\frac{\vp\si^2}{y^2}\Big[\psi\Big(\frac{xy}{\vp\si^2}\Big)-\psi\Big(\frac{\al_x xy}{\vp\si^2}\Big)\Big]
	=\frac{(1-\vp)\si^2}{y^2}\ln^2\Big(1+\frac{xy}{\vp\si^2}\Big)+o(1).
\end{equation}
Now \eqref{eq:max at al_x}, \eqref{eq:g(al)}, \eqref{eq:la_x,o(1)}, and \eqref{eq:sq} yield
\begin{equation}\label{eq:PU sim}
	\PU(x)\sim\exp\Big\{\frac{(1-\vp)\si^2}{2y^2}\ln^2\Big(1+\frac{xy}{\vp\si^2}\Big)
	-\frac{\vp\si^2}{y^2}\psi\Big(\frac{xy}{\vp\si^2}\Big)\Big\}. 
\end{equation}
Next,
\begin{equation*}
	\frac{\vp\si^2}{y^2}\psi\Big(\frac{xy}{\vp\si^2}\Big)
	=\frac{\vp\si^2+xy}{y^2}\,\Big[\ln\frac1\vp+\ln\Big(\vp+\frac{xy}{\si^2}\Big)\Big]-\frac{xy}{y^2}.
\end{equation*}
Using this together with $\ln(\vp+\frac{xy}{\si^2})=\ln(1+\frac{xy}{\si^2})-\frac{(1-\vp)\si^2}{xy}(1+o(1))$, \eqref{eq:PU sim}, and the definition of $\BH(x)$ in \eqref{eq:hoeff}, one concludes the proof of Proposition~\ref{prop:PU<<BH}. 
\end{proof}

\begin{proof}[Proof of Proposition~\ref{prop:be asymp}]
Fix indeed any $\si>0$ and $y>0$. 
By rescaling, w.l.o.g.\ $y=1$. For brevity, let
\begin{equation*}
\th:=\si^2\quad\text{and}\quad	v:=\frac{xy}{\si^2}=\frac x\th.
\end{equation*}
Letting also $k:=\lceil x+\th\rceil=\lceil\th(1+v)\rceil$ (so that $\th(1+v)\le k<\th(1+v)+1$) and using the Stirling formula, one has the following for $x>0$:  
\begin{align*}
\PP(y\tPi_{\si^2/y^2}\ge x)&=\PP(\Pi_\th\ge x+\th)\ge
	\PP(\Pi_\th=k)=e^{-\th}\th^k/k!\\
	&\OOG\frac1{\sqrt k}\,\Big(\frac{e\th}{k}\Big)^k
	\OOG\frac1{\sqrt v}\,\exp\Big\{\big(\th(1+v)+1\big)\,\ln\frac{e\th}{\th(1+v)+1}\Big\} \\
	&\OOG\frac1{v^{3/2}}\,\exp\big\{\th(1+v)\,[1-\ln(1+v+1/\th)]\big\} \\
	&>\frac{e^\th}{ev^{3/2}}\,\exp\big\{\th\big(v-(1+v)\ln(1+v)\big)\big\} 
	=\frac{e^\th \th^{3/2}}{ex^{3/2}}\,\BH(x), 
\end{align*}
which proves the first inequality in \eqref{eq:be asymp1}. 

The second inequality in \eqref{eq:be asymp1} follows by the first inequality in \eqref{eq:Be}, since $\tPi_\th$ is the limit in distribution as $n\to\infty$ of $S=X_1+\dots+X_n$, where the $X_i$'s are i.i.d.\ r.v.'s each with the centered Bernoulli distribution with parameter $\th/n$. 

The second inequality in \eqref{eq:be asymp2} follows similarly by the Bennett-Hoeffding inequality \eqref{eq:hoeff}. 

The third inequality in \eqref{eq:be asymp1} is the second inequality in \eqref{eq:Be}. 

It remains to verify the first inequality in \eqref{eq:be asymp2}. Let, for brevity, $G(u):=\PP(\Pi_\th\ge u)=\PP(\tPi_\th\ge u-\th)=\PP(y\tPi_{\si^2/y^2}\ge u-\th)$. Then, 
by Proposition~\ref{prop:PLC(Po>x)} (with $j=\lceil u-1\rceil$) and Remark~\ref{rem:PLC}, for $x:=u-\th$ one has 
\begin{multline*}
	\PP^\lc(y\tPi_{\si^2/y^2}\ge x)=\PP^\lc(\Pi_\th\ge u) \\
	\le G(j)=G(u)\frac{G(j)}{G(j+1)}\OO j\,G(u)\le u\PP(\Pi_\th\ge u) \\
	=(x+\th)\PP(y\tPi_{\si^2/y^2}\ge x)
\end{multline*}
if $u>1$, since 
$\frac{\th^j}{j!}e^{-\th}\le G(j)=\sum_{m=j}^\infty\frac{\th^m}{m!}e^{-\th}\le\frac{\th^j}{j!}e^{-\th}\frac1{1-\th/j}$ for all $j=1,2,\dots$ such that $j>\th$. 
This concludes the proof of Proposition~\ref{prop:be asymp}. 
\end{proof}

\begin{proof}[Proof of Proposition~\ref{prop:pin asymp}]
The second inequality in \eqref{eq:pin asymp1} follows by inequality~\eqref{eq:main} and Lemma~\ref{lem:exact}. 

The second inequality in \eqref{eq:pin asymp2} follows similarly by inequality \eqref{eq:PU} and Lemma~\ref{lem:exact}. 

The third inequality in \eqref{eq:pin asymp1} is inequality \eqref{eq:main LC}. 

It remains to prove the first inequality in \eqref{eq:pin asymp1} and the first inequality in \eqref{eq:pin asymp2}. W.l.o.g.\ $y=1$. 

To prove the first inequality in \eqref{eq:pin asymp1}, use identity \eqref{eq:max at al_x}, the first inequality in \eqref{eq:be asymp1}, and the Laplace method for the asymptotics of integrals, as follows. By \eqref{eq:eta} (with $y=1$), \eqref{eq:al>vp}, the first inequality in \eqref{eq:be asymp1}, \eqref{eq:max al}, and \eqref{eq:max at al_x}, 
\begin{align}
	\PP(\eta_{\si,1,\vp}\ge x)&\ge\int_{\al_x x}^x \PP(\tPi_{\vp\si^2}\ge z) \PP(x-\Ga_{(1-\vp)\si^2}\in\dd z) \label{eq:>int}\\
	&\OOG\int_{\al_x x}^x \frac{\BH_{\vp\si^2,1}(z)}{z^{3/2}} \exp\Big\{-\frac{(x-z)^2}{2(1-\vp)\si^2}\Big\}\,\dd z \notag\\
	&\ge\frac{\PU(x)}{x^{3/2}}\int_{\al_x x}^x e^{h(z)-h(\al_x x)}\dd z \label{eq:>PU/x^{3/2}}
\end{align}
for all large enough $x>0$, where 
\begin{equation*}
	h(z):=h_x(z):=\ln\big(\BH_{\vp\si^2,1}(z)\Ne_{(1-\vp)\si^2}(x-z)\big)
	=-\frac{(x-z)^2}{2(1-\vp)\si^2}-\vp\si^2\psi\Big(\frac z{\vp\si^2}\Big),  
\end{equation*}
recalling also \eqref{eq:hoeff}. By \eqref{eq:max at al_x} (or because $\al_x$ is a root of equation \eqref{eq:al_x}), one has $h'(\al_x x)=0$. Also, for all large enough $x>0$ and all $z\in[\al_x x,x]$ one has 
$h''(z)=-\frac1{(1-\vp)\si^2}-\frac1{z+\vp\si^2}\ge-\frac2{(1-\vp)\si^2}$ (here using \eqref{eq:al>vp} again) and hence $h(z)-h(\al_x x)\ge-\frac{(z-\al_x x)^2}{(1-\vp)\si^2}$. 
Note also that, by \eqref{eq:1-al_x}, $x-\al_x x\to\infty$ as $x\to\infty$. 
Now the first inequality in \eqref{eq:pin asymp1} follows by \eqref{eq:>int}-\eqref{eq:>PU/x^{3/2}}. 

Finally, to prove the first inequality in \eqref{eq:pin asymp2}, use \eqref{eq:eta} and \eqref{eq:PLC(Ga+Po>x)}. Let   
$I_1$, $I_2$, and $I_3$ be the integrals of the integrand in \eqref{eq:PLC(Ga+Po>x)} (with $y=1$) over the intervals $(-\infty,1]$, $(1,2x]$, and $(2x,\infty)$, respectively. 
Then, in view of the trivial bound $\PP^\lc(\tPi_{\vp\si^2}\ge z)\le1$ for all $z\in\R$, 
\begin{align}
	I_1&\le\int_{-\infty}^1\PP(x-\Ga_{(1-\vp)\si^2}\in\dd z)
	=\PP(\Ga_{(1-\vp)\si^2}\ge x-1) \notag\\
	&\OO\exp\Big\{-\frac{(x-1)^2}{2(1-\vp)\si^2}\Big\}
	\OO e^{-x}\frac1{x\sqrt{2\pi}}\exp\Big\{-\frac{x^2}{2\si^2}\Big\}
	\OO e^{-x}\PP(\Ga_{\si^2}\ge x) \notag \\
	&\le e^{-x}\PU(x)\OO e^{-x}x^{3/2}\PP(\eta_{\si,y,\vp}\ge x)\OO x\PP(\eta_{\si,y,\vp}\ge x)\label{eq:I1}
\end{align}
for all large enough $x>0$; the first inequality in the line \eqref{eq:I1} is a limit case of the inequality in \eqref{eq:PU} (cf.\ Lemma~\ref{lem:exact}), while the second inequality in the line \eqref{eq:I1} is the first inequality in \eqref{eq:pin asymp1}, proved in the previous paragraph. 
Quite similarly, 
\begin{align}
	I_3&\le\int_{2x}^\infty\PP(x-\Ga_{(1-\vp)\si^2}\in\dd z) \notag\\
	&=\PP(\Ga_{(1-\vp)\si^2}\le-x)=\PP(\Ga_{(1-\vp)\si^2}\ge x) 
	\OO x\PP(\eta_{\si,y,\vp}\ge x)\notag
\end{align}
for all large enough $x>0$. 
Finally, by the first inequality in \eqref{eq:be asymp2}, for $x>1$ 
\begin{align*}
	I_2&\OO\int_1^{2x}z\PP(\tPi_{\vp\si^2}\ge z)\PP(x-\Ga_{(1-\vp)\si^2}\in\dd z) \\
	&\le2x
	\int_\R\PP(\tPi_{\vp\si^2}\ge z)\PP(x-\Ga_{(1-\vp)\si^2}\in\dd z)  
	=2x\PP(\eta_{\si,y,\vp}\ge x). 
\end{align*}
Thus, the first inequality in \eqref{eq:pin asymp2} follows from \eqref{eq:PLC(Ga+Po>x)} and the above bounds on $I_1$, $I_3$, and $I_2$.  
This concludes the proof of Proposition~\ref{prop:pin asymp}. 
\end{proof}

\begin{proof}[Proof of Proposition~\ref{prop:normal opt asymp}]
This follows from the more general \cite[Theorem~4.2]{pin98} (or, rather, from its proof) and Proposition~\ref{prop:P_al=inf}. 
\end{proof}

\begin{proof}[Proof of Proposition~\ref{prop:pois opt asymp}]
Fix indeed any $\al>1$ and $\th>0$. By Proposition~\ref{prop:Th2.5,pin98}(vi), \eqref{eq:pois centered opt asymp} is equivalent to   
\begin{equation}\label{eq:pois opt asymp}
	P_\al(\Pi_\th;k)\sim\PP(\Pi_\th\ge k)\sim\frac k\th\PP(\Pi_\th>k)
\end{equation} 
as $\Z\ni k\to\infty$. 
To begin proving this, take any $k=0,1,\dots$. Then, by \eqref{eq:m},
\begin{equation}\label{eq:m(k)}
	m(k)=m_{\al,\Pi_\th}(k)=k+\frac{\E(\Pi_\th-k)_+^\al}{\E(\Pi_\th-k)_+^{\al-1}}>k+1=m(t_{k+1}); 
\end{equation}
the last equality here holds by \eqref{eq:m(t)=x}, while the inequality in \eqref {eq:m(k)} follows because $(\Pi_\th-k)_+^\al\ge(\Pi_\th-k)_+^{\al-1}$ a.s., and the latter inequality is a.s.\ strict on the event $\{\Pi_\th\ge k+2\}$, which is of nonzero probability. So, by Proposition~\ref{prop:Th2.5,pin98}(i), 
\begin{equation}\label{eq:t>k}
	t_{k+1}<k.
\end{equation}

The other key observation about $t_{k+1}$ is that it is close enough to $k$ for large $k$. To see that, let $\de\in(0,1)$ and $k\in\Z$ be varying so that $\de\downarrow0$ and $k\to\infty$. Then
\begin{equation*}
	\E(\Pi_\th-k+\de)_+^\al
	=a_k+s_k,
\end{equation*}
where $a_j:=a_{j,k}:=(j-k+\de)^\al\frac{\th^j}{j!}e^{-\th}$ and $s_k:=\sum_{j=k+1}^\infty a_j$. 
Note that 
$$\frac{a_{j+1}}{a_j}=\Big(\frac{j+1-k+\de}{j-k+\de}\Big)^\al\frac\th{j+1}
\le\Big(\frac{2+\de}{1+\de}\Big)^\al\frac\th{j+1}\to0$$
for $j\ge k+1$. Hence, 
$$s_k\sim a_{k+1}=(1+\de)^\al\frac{\th^{k+1}}{(k+1)!}e^{-\th}\sim\frac{\th^{k+1}}{(k+1)!}e^{-\th},$$
and so, 
\begin{equation}\label{eq:E al}
	\E(\Pi_\th-k+\de)_+^\al\sim a_k+a_{k+1}\sim\frac{\th^k}{k!}e^{-\th}\Big(\de^\al+\frac{\th+o(1)}k\Big). 
\end{equation}
Similarly,
\begin{equation}\label{eq:E al-1}
	\E(\Pi_\th-k+\de)_+^{\al-1}\sim\frac{\th^k}{k!}e^{-\th}\Big(\de^{\al-1}+\frac{\th+o(1)}k\Big). 
\end{equation}
Now let us choose an arbirary constant $c>0$ and then specify $\de$ by the formula
\begin{equation*}
	\de=\Big(\frac ck\Big)^{\frac1{\al-1}},
\end{equation*}
so that $\de^\al=o(\frac1k)$ and $\de^{\al-1}=\frac ck$. Then, using \eqref{eq:E al} and \eqref{eq:E al-1}, one has the following for large enough $k$ (cf.\ \eqref{eq:m(k)}):
\begin{equation*}
	m(k-\de)=k-\de+\frac\th{c+\th}(1+o(1))<k+1,
\end{equation*}
whence $t_{k+1}>k-\de=k-\big(\frac ck\big)^{1/(\al-1)}$. Recalling now \eqref{eq:t>k} and that $c>0$ was arbitrary, one concludes that indeed $t_{k+1}$ is close to $k$, in the sense that 
\begin{equation*}
	k-o\big(k^{-1/(\al-1)}\big)<t_{k+1}<k. 
\end{equation*}
Revisiting \eqref{eq:E al} and \eqref{eq:E al-1} with $\de:=k-t_{k+1}=o\big(k^{-1/(\al-1)}\big)$, one has  
$$\E(\Pi_\th-t_{k+1})_+^\al\sim\frac{\th^{k+1}}{(k+1)!}e^{-\th} \quad\text{and}\quad  \E(\Pi_\th-t_{k+1})_+^{\al-1}\sim\frac{\th^{k+1}}{(k+1)!}e^{-\th}.$$
So, recalling the last expression in \eqref{eq:P(x)}, one concludes that 
\begin{equation}\label{eq:pois opt asymp1}
	P_\al(\Pi_\th;k+1)\sim\frac{\th^{k+1}}{(k+1)!}e^{-\th}. 
\end{equation} 
On the other hand, it is easy to see (cf.\ the first relation in \eqref{eq:E al}) that 
\begin{equation*}
	\PP(\Pi_\th\ge k+1)\sim\frac{\th^{k+1}}{(k+1)!}e^{-\th}\quad\text{and}\quad
	\PP(\Pi_\th>k+1)
	\sim\frac{\th^{k+2}}{(k+2)!}e^{-\th}. 
\end{equation*} 
Now \eqref{eq:pois opt asymp} follows by \eqref{eq:pois opt asymp1}. 
\end{proof}

\subsection{Proofs of the lemmas}\label{proofs2}

\begin{proof}[Proof of Lemma~\ref{lem:1}]
This follows because $\frac{x^3}{(x+\si^2/y)^2}$ is nondecreasing in $x\in[0,y]$ from $0$ to $\frac{y^3}{(y+\si^2/y)^2}=\frac{y^5}{(y^2+\si^2)^2}$.
\end{proof}

\begin{proof}[Proof of Lemma~\ref{lem:2}]
This follows by Lemma~\ref{lem:1}:  
$$\E X_+^3\le\frac{y^5}{(y^2+\si^2)^2}\big(\E X^2+(\si^2/y)^2\big)\le\frac{y^3\si^2}{y^2+\si^2}.$$
\end{proof}

\begin{proof}[Proof of Lemma~\ref{lem:3}]
Take any $\be$ satisfying condition \eqref{eq:be}. Let $f(x):=\si^2x^{3/2}-\be(x+\si^2)$. Then $f(0)=-\be\si^2<0$, $f(y^2)=\si^2y^3-\be(y^2+\si^2)\ge0$ by \eqref{eq:be}, 
and the function $f$ is convex on $[0,\infty)$. Hence, $f$ has exactly one positive root, say $x_*$, and at that $x_*\in(0,y^2]$. Let $b:=x_*^{1/2}$, so that $b\in(0,y]$ and $b$ is the only positive root of equation \eqref{eq:b}. Letting now $a:=\si^2/b$, one has $X_{a,b}\le y$ a.s., $\E X_{a,b}=0$, $\E X^2_{a,b}=ab=\si^2$, and $\E(X_{a,b})_+^3=\frac{ab^3}{a+b}=\frac{\si^2b^3}{\si^2+b^2}=\be$, by \eqref{eq:b}. 
It also follows that $a=\frac{\be b}{b^3-\be}$. 
Finally, the uniqueness of the pair $(a,b)$ follows from the uniqueness of the positive root $b$ of equation \eqref{eq:b}.
\end{proof}

\begin{proof}[Proof of Lemma~\ref{prop:ab}]
Let $X$ be any r.v.\ such that $X\le y$ a.s., $\E X\le0$, $\E X^2\le\si^2$, and $\E X_+^3\le\be$. 
Let us consider separately the following possible cases: $w\le-a$, $-a\le w\le0$,  and $w\ge0$. 

\emph{Case 1: $w\le-a$.}\quad Then 
\begin{equation*}
	f_1(x):=A_0+A_1x+ A_2x^2+A_3x_+^3\ge(x-w)_+^3
\end{equation*}
for all $x\in\R$, where 
\begin{align*}
	A_0&:=\frac{2 a^3 b}{3 a+b}-w^3,\\
	A_1&:=3\frac{b(a^2+w^2)+a(3w^2-a^2)}{3a+b},\\
   A_2&:=-3\frac{(a+b)w+2a(a+w)}{3a+b},\\
   A_3&:=\frac{(a+b)^3}{b^2(3a+b)}
\end{align*}
are obviously nonnegative constants; moreover, $f_1(x)=(x-w)_+^3$ for $x\in\{-a,b\}$. This claim can be verified using the Mathematica command \\
\begin{equation*}
\begin{aligned}
& \verb9Reduce[b > 0 && a > 0 && w < -a &&9\\
& \verb9A0 + A1 x + A2 x^2 + A3 Max[0,x]^3 - Max[0,x - w]^3 <= 0, {w, x}]9 	
\end{aligned}
\end{equation*}
which produces the output 
$$\verb9b > 0 && a > 0 && w < -a && (x == -a || x == b)9$$ 
where \verb9A09, \verb9A19, \verb9A29, \verb9A39 represent the constants $A_0,A_1,A_2,A_3$ as defined above; this verification takes about $1$ second (this and other execution times given in this paper are in reference to an Intel Core 2 Duo PC with 4 GB of RAM).  

Therefore and because $\E X\le0=\E X_{a,b}$, $\E X^2\le\si^2=\E X_{a,b}^2$, and $\E X_+^3\le\be=\E(X_{a,b})_+^3$, one has
\begin{align*}
	\E (X-w)_+^3&\le A_0+A_1\E X+A_2\E X^2+A_3\E X_+^3 \\
	&\le A_0+A_1\E X_{a,b}+A_2\E X_{a,b}^2+A_3\E(X_{a,b})_+^3 \\
	&=\E (X_{a,b}-w)_+^3.
\end{align*}

\emph{Case 2: $-a\le w\le0$.}\quad Then 
\begin{equation*}
	f_2(x):=\la_2(x+a)^2+\la_3x_+^3\ge(x-w)_+^3
\end{equation*}
for all $x\in\R$,
where
\begin{equation*}
	\la_2:=\frac{-3w(b - w)^2}{(a+b)(3a+b)}\quad\text{and}\quad
	\la_3:=\frac{(b-w)^2\big(2(w+a)+a+b\big)}{b^2(3a+b)}
\end{equation*}
are obviously nonnegative constants; moreover, $f_2(x)=(x-w)_+^3$ for $x\in\{-a,b\}$. This claim can be verified using a similar Reduce command, which takes under $1$ second.   
Therefore,
\begin{align*}
	\E (X-w)_+^3&\le \la_2\E(X+a)^2+\la_3\E X_+^3 \\
&\le \la_2\E(X_{a,b}+a)^2+\la_3\E(X_{a,b})_+^3  \\
	&=\E (X_{a,b}-w)_+^3.
\end{align*}

\emph{Case 3: $w\ge0$.}\quad Then 
\begin{equation*}
	f_3(x):=\frac{(y-w)_+^3}{y^3}\,x_+^3\ge(x-w)_+^3
\end{equation*}
for all $x\in(-\infty,y]$, since $\frac{(x-w)^3}{x^3}$ is nondecreasing in $x\in[w,\infty)$ for each $w\ge0$; 
moreover, it is obvious that $f_3(x)=(x-w)_+^3$ for $x\in(-\infty,0]\cup\{y\}$. 

Further, 
$y^3\ge b^3>\frac{ab^3}{a+b}=\be$; hence, again by \eqref{eq:tb,ta}, $\ta>0$. It follows that
$f_3(x)=(x-w)_+^3$ for $x\in\{-\ta,\tb\}$. 
Moreover, $\E(X_{\ta,\tb})_+^3=\be$. 
Thus,
\begin{align*}
	\E (X-w)_+^3\le\frac{(y-w)_+^3}{y^3}\,\E X_+^3 
\le\frac{(y-w)_+^3}{y^3}\,\E(X_{\ta,\tb})_+^3 
	=\E (X_{\ta,\tb}-w)_+^3.
\end{align*}
Moreover,
\begin{equation}\label{eq:le si}
	\E X_{\ta,\tb}^2=\ta\tb=\frac{\be y^2}{y^3-\be}\le
	\frac{\be b^2}{b^3-\be}=ab=\si^2;
\end{equation}
the inequality here takes place because $\frac{\be u^2}{u^3-\be}$ decreases in $u>\be^{1/3}$, while, as shown, $y^3\ge b^3>\be$; 
the inequality in \eqref{eq:le si} is strict if $\be\ne\frac{y^3\si^2}{y^2+\si^2}$ \big(because then
$\frac{b^3\si^2}{b^2+\si^2}=\be<\frac{y^3\si^2}{y^2+\si^2}$, and hence $b<y$\big). 


Thus, in all the three cases, one has equality \eqref{eq:attain}. 
Moreover, in the case $w\le0$ the maximum in \eqref{eq:max_equal} is attained and equals $\E(X_{a,b}-w)_+^3$, since $X_{a,b}\le y$ a.s., $\E X_{a,b}=0$, $\E X_{a,b}^2=\si^2$, and $\E(X_{a,b})_+^3=\be$. 
The last sentence of Lemma~\ref{prop:ab} has also been proved. 

To complete the proof of the lemma, it remains to show that in the case $w\ge0$ the maxima in \eqref{eq:max_equal}  and \eqref{eq:max_less} are attained and equal $\E(X_{\ta,\tb}-w)_+^3$; the same last sentence of Lemma~\ref{prop:ab} shows that in this case the $\max$ in \eqref{eq:max_equal} is \emph{not} attained \emph{at} $X=X_{\ta,\tb}$ if $\be\ne\frac{y^3\si^2}{y^2+\si^2}$ -- because then $\E X_{\ta,\tb}^2<\si^2$. 

Thus, it suffices to construct a r.v., say $X_v$, such that $\E(X_v-w)_+^3=\E(X_{\ta,\tb}-w)_+^3$, while $X_v\le y$ a.s., $\E X_v=0$, $\E X_v^2=\si^2$, and $\E(X_v)_+^3=\be$. One way to satisfy all these conditions is to let 
$X_v\sim\tp\de_y+q_1\de_{-a_1}+r_1\de_v$, where $v$ is close enough to $-\infty$, $r_1:=-\De q$, $q_1:=\tq+\De q$, $a_1:=\ta+\De a$, $\tp:=\be/y^3$, $\tq:=1-\tp$, 
$\De q:=-\frac{\tq d^2}{d^2+\tq(v+\ta)^2}$, $\De a:=\frac{d^2}{\tq(v+\ta)}$, $d:=\sqrt{\si^2-\ta\tb}=\sqrt{\si^2-\ta y}$, and $\ta$ and $\tb$ are given by \eqref{eq:tb,ta}. 
\end{proof}

\begin{proof}[Proof of Lemma~\ref{lem:F3}]
Take any $f\in{\F3}_{,12}$, so that $f$ is also in ${\F3}_{,1}$ and $\F3$; then $f'$ is convex or, equivalently, $f''$ is nondecreasing; at that, $f''$ is also convex. 
For any $z\in\R$, introduce the functions $g_z:=g_{z,f}$  and $h_z:=h_{z,f}$ by the formulas
\begin{align}
g_z(u)&:=\big(f(z)+(u-z)f'(z)+(u-z)^2f''(z)/2\big)\ii{u<z}+f(u)\ii{u\ge z} \label{eq:g_z}\\	
\intertext{and}
h_z(u)&:=g_z(u)-f(z)-(u-z)f'(z)-(u-z)^2f''(z)/2 \label{eq:h_z}\\
&=\big(f(u)-f(z)-(u-z)f'(z)-(u-z)^2f''(z)/2\big)\ii{u>z}, \notag	
\end{align}
for all $u\in\R$. 
Then 
$$g'_z(u)=\big(f'(z)+(u-z)f''(z)\big)\ii{u<z}+f'(u)\ii{u\ge z}$$
for all $u\in\R$. 
Since $f'$ is convex, $f'(z)+(u-z)f''(z)\le f'(u)$ and hence 
$g'_z(u)\le f'(u)$ for all $u\in\R$. 

Moreover, $g'_z(u)$ is nonincreasing in $z$ for each $u\in\R$. Indeed, take any real $z_1$ and $z_2$ such that $z_1<z_2$. 
If $u\ge z_2$ then $g'_{z_1}(u)=f'(u)=g'_{z_2}(u)$, so that $g'_{z_1}(u)\ge g'_{z_2}(u)$. 
Next, if $z_1\le u<z_2$ then, by the convexity of $f'$, one has $g'_{z_1}(u)=f'(u)\ge f'(z_2)+(u-z_2)f''(z_2)=g'_{z_2}(u)$, so that again $g'_{z_1}(u)\ge g'_{z_2}(u)$. 
Finally, if $u<z_1$ then, bounding the terms $f'(z_1)$ and $(u-z_1)f''(z_1)$ separately from below in view of the conditions that $f'$ is convex and $f''$ is nondecreasing, one has 
\begin{align*}
g'_{z_1}(u)&=[f'(z_1)]+[(u-z_1)f''(z_1)] \\ 
&\ge
[f'(z_2)+(z_1-z_2)f''(z_2]+[(u-z_1)f''(z_2)] \\
&=f'(z_2)+(u-z_2)f''(z_2)=g'_{z_2}(u),
\end{align*}
so that in this case as well $g'_{z_1}(u)\ge g'_{z_2}(u)$. 

Also, $g_z=f$ on $[z,\infty)$. It follows that 
\begin{equation}\label{eq:mono g_z}
	\text{$g_{z_2}\ge g_{z_1}\ge f$ on $\R$ for any real $z_1$ and $z_2$ such that $z_1<z_2$. }
\end{equation}

Next, $h_z(u)=h'_z(u)=h''_z(u)=0$ for all $u\in(-\infty,z)$ and 
$h''_z(u)=\big(f''(u)-f''(z)\big)\ii{u>z}=\big(f''(u)-f''(z)\big)_+$ for all $u\in\R$, since $f''$ is nondecreasing. 
Moreover, $h''_z$ is convex, since $f''$ is so. 
Therefore, by Proposition~\ref{prop:F-al}, $h_z\in\H3$, which yields $\E h_z(\xi)\le\E h_z(\eta)$. 

Assume now that $z\in(-\infty,0)$. 
Then, in view of \eqref{eq:h_z},  
\begin{equation}\label{eq:g=h}
g_z(u)
=h_z(u)+u\big(f'(z)+|z|f''(z)\big)+u^2f''(z)/2+c(z)	
\end{equation}
for all $u\in\R$, where $c(z):=f(z)-zf'(z)+z^2f''(z)/2$. 
So, if $f\in\F3$, then the established earlier inequality $\E h_z(\xi)\le\E h_z(\eta)$ yields now $\E g_z(\xi)\le\E g_z(\eta)$, because $\E\xi\le\E\eta$, $\E\xi^2\le\E\eta^2$, and the coefficients $f'(z)+|z|f''(z)$ and $f''(z)/2$ of $u$ and $u^2$ on the right-hand-side of \eqref{eq:g=h} are nonnegative; if $\E\xi=\E\eta$, then the sign of $f'(z)+|z|f''(z)$ does not matter, so that the same inequality $\E g_z(\xi)\le\E g_z(\eta)$ will hold whenever $f\in{\F3}_{,1}$. 
Similarly, if one has both equalities $\E\xi=\E\eta$ and $\E\xi^2=\E\eta^2$, then neither the sign of $f'(z)+|z|f''(z)$ nor that of $f''(z)/2$ matters, so that the inequality $\E g_z(\xi)\le\E g_z(\eta)$ will hold whenever $f\in{\F3}_{,12}$. 

Now we are ready to complete the proof of parts (i), (ii) and (iii) of Lemma~\ref{lem:F3}. Indeed, w.l.o.g.\ $\E f(\eta)<\infty$, for otherwise the inequality $\E f(\xi)\le\E f(\eta)$ is trivial. Since $f'$ is convex, there are some real $a$ and $b$ such that $f'(x)\ge a+bx$ for all $x\in\R$. Hence, $f(x)\ge-c(1+x^2)$ for some real $c=c_f>0$ and all $x\in\R$. Now the condition $\E\eta^2<\infty$ implies that $\E f(\eta)>-\infty$, and so, $\E f(\eta)\in\R$.  
Therefore, in view of \eqref{eq:g_z} and the condition $\E\eta^2<\infty$, one has $\E g_z(\eta)\in\R$. 
Now, letting $z\to-\infty$, observing that $g_z(u)\to f(u)$ for all $u\in\R$, and 
recalling \eqref{eq:mono g_z}, one concludes by dominated convergence that $\E g_z(\eta)\to\E f(\eta)$. 
Also, \eqref{eq:mono g_z} implies that $\E g_z(\xi)\ge\E f(\xi)$ for all $z\in\R$. 
Recall now that $\E g_z(\xi)\le\E g_z(\eta)$ for all $z\in\R$. 
Thus, parts (i), (ii) and (iii) of Lemma~\ref{lem:F3} are proved. 

Let us prove part (iv) of Lemma~\ref{lem:F3}. An idea here is to give the distribution of the r.v.\ $\eta$ a heavy left tail, to which the cubic moment function $x\mapsto x^3$ 
would be sensitive enough --- in contrast with the moment functions $x\mapsto (x-t)_+^2$ in $\H2$. 
Recall that $\de_x$ stands for the probability distribution concentrated at point $x$. 
Let 
$$\xi\sim\tfrac12\de_{-1}+\tfrac12\de_1\quad\text{and}\quad\eta\sim q\de_{-v}+(\tfrac12-q)\de_{-1+\vp}+\tfrac12\de_1,$$ 
with $v\to\infty$ and $\vp:=v^{-1/2}$, and let $q:=\frac{(2-\vp)\vp}{2(v^2-(1-\vp)^2)}\sim v^{-5/2}=\vp^5$ be chosen so that $1=\E\eta^2=v^2q+(-1+\vp)^2(\frac12-q)+\frac12$. 
Then (eventually, as $v\to\infty$) $\E\eta=-q(v-1+\vp)+\frac\vp2>0=\E\xi$, and $\E\eta^2=1=\E\xi^2$. 

It is not hard to see that $\E f(\xi)\le\E f(\eta)$ for all $f\in\H2$ and hence for all $f\in\H3$. Indeed, to see this it is enough to check that 
$\E(\xi-t)_+^2\le\E(\eta-t)_+^2$ for all $t\in\R$. 
If $t\ge-1+\vp$, then trivially 
$\E(\xi-t)_+^2=\frac 12(1-t)^2=\E(\eta-t)_+^2$. 
If $-1\le t\le-1+\vp$, then $\E(\xi-t)_+^2=\E(\eta-t)_+^2-(\frac12-q)(-1+\vp-t)^2\le\E(\eta-t)_+^2$. 
If $-v\le t\le-1$, then $\E(\xi-t)_+^2=\E(\eta-t)_+^2+(t+1)\vp-\vp^2/2+O(\vp^3)\le\E(\eta-t)_+^2$. 
Finally, if $t\le-v$, then $\E(\xi-t)_+^2=\E(\eta-t)_+^2+2t\E\eta\le\E(\eta-t)_+^2$, since $\E\eta>0$. 

Thus, all the conditions stated in the beginning of Lemma ~\ref{lem:F3} are satisfied: $\E\xi\le\E\eta$, $\E\xi^2\le\E\eta^2<\infty$, and 
$\E f(\xi)\le\E f(\eta)$ for all $f\in\H2$ and hence 
for all $f\in\H3$, and one even has the equality $\E\xi^2=\E\eta^2$. 
Yet, $\E f_*(\xi)>\E f_*(\eta)$ for the cubic function $f_*$ defined by the formula $f_*(x)=x^3$ for all $x\in\R$, even though $f_*\in{\F3}_{,2}$. 
Indeed, $\E f_*(\xi)=\E\xi^3=0$, while $\E f_*(\eta)=\E\eta^3=-v^3q+(-1+\vp)^3(\frac12-q)+\frac12\sim-v^3q\sim-v^{1/2}\to-\infty$. 
The proof of part (iv) and hence that of the entire Lemma~\ref{lem:F3} is now complete. 
\end{proof}

\begin{proof}[Proof of Lemma~\ref{lem:F2}] 
This proof is almost identical to that of Lemma~\ref{lem:F3}. Only two modifications are needed. 
First, in he beginning of the proof now we take any $f$ in ${\F2}_{,12}$ rather than in ${\F3}_{,12}$. 

Second, note that the ``$\F3$'' condition that $f'''$ is nondecreasing or, equivalently, that $f''$ is convex was used in the proof of Lemma~\ref{lem:F3} only once --- in the sentence ``Moreover, $h''_z$ is convex, since $f''$ is so.'' in the  paragraph right after \eqref{eq:mono g_z}, to come, via Proposition~1.1, to the conclusion that $h_z\in\H3$. 
Here, to come to the conclusion that $h_z\in\H2$, we note instead that for all $u\in\R$ 
$$h'_z(u)=\big(f'(u)-f'(z)-(u-z)f''(z)\big)\ii{u>z},
$$
which implies that $h'_z$ is convex --- 
since $f'$ is convex and the right derivative $(h'_z)'(u)$ equals $0$ at $u=z$.  
\end{proof}

\begin{proof}[Proof of Lemma~\ref{lem:le}]
In view of Lemma~\ref{lem:F2}(i), the relation $\F3\subseteq\F2$, definition \eqref{eq:H}, and the Fubini theorem, 
it is enough to prove inequality \eqref{eq:le} for all functions of the form $u\mapsto(u-w)_+^2$ for $w\in\R$. 
By rescaling, w.l.o.g.\ $y=1$.  
Further, r.v.\ $\Ga_{\si^2-\be_0}+\tPi_{\be_0}$ equals in distribution $\Ga+\Ga_{\si_0^2-\be_0}+\tPi_{\be_0}$, where $\Ga$ is any r.v.\ such that $\Ga\sim\No(0,\si^2-\si_0^2)$ and $\Ga$ is independent of $\Ga_{\si_0^2-\be_0}$ and $\tPi_{\be_0}$. Now, conditioning on $\Ga_{\si_0^2-\be_0}$ and $\tPi_{\be_0}$ and using Jensen's inequality, one has $\E(\Ga_{\si_0^2-\be_0}+\tPi_{\be_0}-w)_+^2\le\E(\Ga_{\si^2-\be_0}+\tPi_{\be_0}-w)_+^2$ for all $w\in\R$, so that w.l.o.g.\ $\si_0=\si$ and $\be_0<\be\le\si^2$. 
Moreover, r.v.'s\ $\Ga_{\si^2-\be_0}+\tPi_{\be_0}$ and $\Ga_{\si^2-\be}+\tPi_{\be}$ equal in distribution 
$\Ga_{d^2}+W$ and $\tPi_{d^2}+W$, respectively, where $d:=(\be-\be_0)^{1/2}$ and $W$ is any r.v.\ which is independent of $\Ga_{d^2}$ and $\tPi_{d^2}$ and equals $\Ga_{\si^2-\be}+\tPi_{\be_0}$ in distribution.  
Thus, by conditioning on $W$, it suffices to prove that 
\begin{equation}\label{eq:le Po}
\E(\Ga_{d^2}-w)_+^2\le
	\E(\tPi_{d^2}-w)_+^2	
\end{equation}
for all $d>0$ and $w\in\R$.
Note that $\Ga_{d^2}$ and $\tPi_{d^2}$ are limits in distribution of $U_n:=\sum_{i=1}^n U_{i;n}$ and $V_n:=\sum_{i=1}^n V_{i;n}$, respectively, as $n\to\infty$, where the $U_{i;n}$'s are i.i.d.\ copies of $X_{d/\sqrt n,d/\sqrt n}$ and the $V_{i;n}$'s are i.i.d.\ copies of $X_{d^2/n,1}$. By \cite{bent-liet02,bent-ap}, one has \eqref{eq:le Po} with $U_n$ and $V_n$ in place of $\Ga_{d^2}$ and $\tPi_{d^2}$, respectively, provided that $n\ge d^2$ (so that $X_{d/\sqrt n,d/\sqrt n}\le1$ a.s.). 

Finally, it is clear that, for each $w\in\R$, $(x-w)_+^2=o(e^x)$ as $x\to\infty$. Hence and in view of \eqref{eq:BH}, for each $w\in\R$ the sequences of r.v.'s $\big((U_n-w)_+^2\big)$ and $\big((V_n-w)_+^2\big)$ are uniformly integrable. Now \eqref{eq:le Po} follows by a limit transition; see e.g.\ \cite[Theorem~5.4]{bill}.
\end{proof}

\begin{proof}[Proof of Lemma~\ref{th:basic}]
In view of Lemma~\ref{lem:F3}(i), definition \eqref{eq:H}, and the Fubini theorem, 
it is enough to prove inequality \eqref{eq:le} for all functions of the form $u\mapsto(u-w)_+^3$ for $w\in\R$. 
By Lemma~\ref{prop:ab}, $\E(X-w)_+^3\le\E(X_{a,b}-w)_+^3$ for some $a$ and $b$ such that $a>0$, $b>0$, $X_{a,b}\le y$ a.s., $\E X_{a,b}^2=ab\le\si^2$, and $\E(X_{a,b})_+^3=\be$; at that, one of course also has $\E X_{a,b}=0$. 
So, if one could prove inequality \eqref{eq:basic} with $X_{a,b}$ in place of $X$ and $ab$ in place of $\si^2$, then it would remain to refer to Lemma~\ref{lem:le}.
Thus, w.l.o.g.\ one has $X=X_{a_0,b_0}$ for some positive $a_0$ and $b_0$, and at that 
\begin{equation*}
	\E X_{a_0,b_0}^2=a_0b_0=\si^2\quad\text{and}\quad \E(X_{a_0,b_0})_+^3=\frac{a_0b_0^3}{b_0+a_0}=\be. 
\end{equation*}
By rescaling, w.l.o.g.\ 
\begin{equation*}
	y=1, \quad\text{whence $b_0\le1$.}
\end{equation*}

The main idea of the proof is to introduce a family of r.v.'s of the form
\begin{equation*}
	\eta_b:=X_{a(b),b}+\xi_{\tau(b)}\quad\text{for}\quad b\in[\vp,b_0],
\end{equation*}
where 
\begin{gather*}
\vp:=\frac{b_0^2}{b_0+a_0}=\frac\be{\si^2},\\
a(b):=\frac b{\vp}(b-\vp),\\
\tau(b):=a_0b_0-a(b)b,\\
\xi_t:=W_{(1-\vp)t}+\tPi_{\vp t},\\
\tPi_s:=\Pi_s-\E\Pi_s,
\end{gather*}
$W_\cdot$ is a standard Wiener process, $\Pi_\cdot$ is a Poisson process with intensity $1$, and $X_{a(b),b}$, $W_\cdot$, $\Pi_\cdot$ are independent for each $b\in[\vp,b_0]$.  
Note that $\vp\in(0,b_0)\subseteq(0,1)$; also, $\tau$ is decreasing and hence nonnegative on the interval $[\vp,b_0]$, since $a(b)b$ is increasing in $b\in[\vp,b_0]$ and $a(b_0)=a_0$.

Let further 
\begin{equation*}
	\EE(b):=\EE(b,w):=\E(\eta_b-w)_+^3	=\frac{b\,\E(\xi_{\tau(b)}-a(b)-w)_+^3+a(b)\E(\xi_{\tau(b)}+b-w)_+^3}{b+a(b)}. 
\end{equation*}
Since $a(b_0)=a_0$ and $a(\vp)=0$, one has $X_{a(\vp),\vp}=0$ a.s. Thus, Lemma~\ref{th:basic} is reduced to the inequality $\EE(\vp)\ge\EE(b_0)$. 
Note that $\EE(b)$ is continuous in $b\in[\vp,b_0]$; this follows because of the uniform integrability (cf.\ the last paragraph in the proof of Lemma~\ref{lem:le}). 
So, it is enough to show that the left derivative $\EE'(b)$ of $\EE(b)$ is no greater than $0$ for all $b\in(\vp,b_0)$. To compute this derivative, one can use the following


\begin{lemma}\label{lem:der}
Consider any function $f\!:\,(\vp,b_0)\times\R\ni (b,x)\mapsto f(b,x)\in\R$
such that $|f''_{bb}(b,x)|+|f''_{bx}(b,x)|+|f''_{xx}(b,x)|+|f'_x(b,x)|
\le C_f e^{|x|}$ and  
$|f''_{xx}(b,x_1)-f''_{xx}(b,x_2)|\le C_f|x_1-x_2|(e^{|x_1|}+e^{|x_2|})$ for some constant $C_f$, 
all $b\in(\vp,b_0)$, and all $x,x_1,x_2$ in $\R$. Then for all $b\in(\vp,b_0)$ 
\begin{equation*}
	\lim_{h\downarrow0}\frac{\E f(b-h,\xi_{\tau(b-h)})-\E f(b,\xi_{\tau(b)})}{-h}
	=\E F_f(b,\xi_{\tau(b)}),
\end{equation*}
where 
\begin{equation*}
F_f(b,x):=f'_b(b,x)
	+\Big(\frac{(1-\vp)f''_{xx}(b,x)}2
	+\vp\big(f(b,x+1)-f(b,x)-f'_x(b,x)\big)\Big)\,
	\tau'(b).
\end{equation*}
\end{lemma}
The proof of this lemma, which involves little more than routine Taylor expansions, will be given later in this paper.

By Lemma~\ref{lem:der}, for all $b\in(\vp,b_0)$ one has $\EE'(b)
=\E G(b,\xi_{\tau(b)}-w)$, where 
\begin{gather*}
	G(b,x):=\Big(\frac b{b+a(b)}\Big)'_b \big(f_1(b,x)-f_2(b,x)\big) 
	+\frac{b\,F_{f_1}(b,x)+a(b)F_{f_2}(b,x)}{b+a(b)}
\label{eq:F(b,u)}\\
f_1(b,x):=(x-a(b))_+^3,\quad f_2(b,x):=(x+b)_+^3.	\notag
\end{gather*}

Thus, it suffices to show that $G(b,u)\le0$ for all $b\in(\vp,b_0)$ and $u\in\R$.
Observe now that 
$\big(\frac b{b+a(b)}\big)'_b=-\frac1{b+a(b)}$, 
$a'(b)=1+\frac{2a(b)}b$, 
$\tau'(b)=-\big(3a(b)+b\big)$, and $\vp=\frac{b^2}{b+a(b)}$. 
Substituting into \eqref{eq:F(b,u)} these expressions of $\big(\frac b{b+a(b)}\big)'_b$, 
$a'(b)$, 
$\tau'(b)$, and $\vp$ in terms of only $b$ and $a(b)$, one has 
$\big(b+a(b)\big)^2G(b,u)=-\tG(a(b),b,-u)$, where 
\begin{align*}
\tG(a,b,t)&:=
\left(a+b-3 ab^3-b^4\right) (-a-t)_+^3 \\
&-\left(a+b+3a^2b^2+ab^3\right) (b-t)_+^3 \\
&+b^2 (3 a+b)\left(b(1-a-t)_+^3 + a(1+b-t)_+^3\right) \\
&+3 \left(2a^2+3ab+b^2-3ab^3-b^4\right) (-a-t)_+^2 \\
&-3a \left(a+b+3ab^2+b^3\right)(b-t)_+^2 \\
&+3(3 a+b)\left(a+b-b^2\right) 
\left(b (-a-t)_+ + a (b-t)_+\right). 	
\end{align*}

To complete the proof of the theorem, it is enough to show that $\tG(a,b,t)\ge0$ for all $a>0$, $b\in(0,1]$, and $t\in\R$. At that, w.l.o.g.\ $t<1+b$, since $\tG(a,b,t)=0$ for all $a>0$, $b\in(0,1]$, and $t\ge1+b$. 
Next, one has either $a+b\le1$ or $a+b>1$. In the first case, $-a\le b\le1-a\le1+b$, while in the second case 
$-a\le1-a\le b\le1+b$. Therefore, it remains to verify that $\tG(a,b,t)\ge0$ in each of the following 8 (sub)cases:
\begin{enumerate}[]
\item \emph{Case 10:} $a > 0 \ \&\  b > 0 \ \&\  a + b \le 1 \ \&\  t \le -a$; 
\item \emph{Case 11:} $a > 0 \ \&\  b > 0 \ \&\  a + b \le 1 \ \&\  -a\le t \le b$; 
\item \emph{Case 12:} $a > 0 \ \&\  b > 0 \ \&\  a + b \le 1 \ \&\  b\le t \le 1-a$; 
\item \emph{Case 13:} $a > 0 \ \&\  b > 0 \ \&\  a + b \le 1 \ \&\  1-a\le t \le 1+b$; 
\item \emph{Case 20:} $a > 0 \ \&\  0<b\le1 \ \&\  a + b > 1 \ \&\  t \le -a$; 
\item \emph{Case 21:} $a > 0 \ \&\  0<b\le1 \ \&\  a + b > 1 \ \&\  -a\le t \le 1-a$; 
\item \emph{Case 22:} $a > 0 \ \&\  0<b\le1 \ \&\  a + b > 1 \ \&\  1-a\le t \le b$; 
\item \emph{Case 23:} $a > 0 \ \&\  0<b\le1 \ \&\  a + b > 1 \ \&\  b\le t \le 1+b$. 
\end{enumerate}

In Case 10, $\tG(a,b,t)=a^2 (5 a^2 + 8 a b + 3 b^2)$, which is obviously positive for all positive $a$ and $b$. 

In Case 11, 
\begin{multline*}
\tG(a,b,t)=G_{11}(a,b,t):= \\
9 a^3 b + 12 a^2 b^2 + 3 a b^3 - 9 a^2 b^3 + 9 a^3 b^3 - 3 a^4 b^3 - 
 3 a b^4 + 3 a^2 b^4 - 
 a^3 b^4 \\
 + (-9 a^3 - 6 a^2 b + 6 a b^2 + 3 b^3 - 9 a b^3 + 
    18 a^2 b^3 - 9 a^3 b^3 - 3 b^4 + 6 a b^4 - 
    3 a^2 b^4) t \\
    + (-3 a^2 - 6 a b - 3 b^2 + 9 a b^3 - 9 a^2 b^3 + 
    3 b^4 - 3 a b^4) t^2 \\
    + (a + b - 3 a b^3 - b^4) t^3,	
\end{multline*}
which is positive in this case, Case~11; this can be verified using a Mathematica command of the form 
\verb9Reduce[G11 <= 0 && case11, Reals]9, which outputs \verb9False9 (in about 13 seconds). 

The other 6 cases are verified similarly; the longest of them in terms of the time it takes Mathematica to check is Case~21 (in about 13 seconds). 
This concludes the proof of Lemma~\ref{th:basic}. 
\end{proof}

\begin{proof}[Proof of Lemma~\ref{lem:exact}]
To begin, let indeed $\be=\vp\si^2y$ and 
take any natural number $m$ and any positive real numbers $a$ and $b$. Let then $n:=2m$, and let indeed $X_1,\dots,X_n$ be any independent r.v.'s such that $X_1,\dots,X_m$ are independent copies of 
$X_{b/\sqrt m,b/\sqrt m}$ and $X_{m+1},\dots,X_{2m}$ are independent copies of 
$X_{a/m,y}$. Then, of course, $\E X_i=0$ for all $i$. 
Next, the system of equations $\sum_{i=1}^n\E X_i^2=\si^2$ and $\sum_{i=1}^n\E(X_i)_+^3=\be$ (cf.\ \eqref{eq:ineqs}) can be rewritten (in view of \eqref{eq:vp}) as $b^2+ay=\si^2$ and $\frac{b^3}{2\sqrt m}+\frac{may^3}{a+my}=\vp\si^2 y$, and then in turn as 
$a=(\si^2-b^2)/y$ and $b^2=(1-\vp)\si^2+r_m(b)$, 
where $r_m(b)$ is a certain expression in terms of $m$, $b$, $\si$, $y$, and $\vp$ (but not containing entries of $a$) such that $r_m(b)\to0$ uniformly in $b\in[0,\si]$ as $m\to\infty$ (recall that $\si$, $y$, and $\vp$ were fixed) and $r_m(b)$ is continuous in $b\in[0,\si]$.  
It follows that for all large enough $m\in\N$ the equation $b^2=(1-\vp)\si^2+r_m(b)$ has a solution $b=b_m\in[0,\si]$, and at that $b\to\si\sqrt{1-\vp}$ and hence $a=(\si^2-b^2)/y\to\vp\si^2/y$ as $m\to\infty$. 
In particular, this implies the statement indented in the formulation of Lemma~\ref{lem:exact}. 
The convergence of $S$ in distribution to $\Ga_{(1-\vp)\si^2}+y\tPi_{\vp\si^2/y^2}$ as $m\to\infty$ can now easily obtained, via either characteristic functions or well-known ready-to-use limit theorems.
\end{proof}

\begin{proof}[Proof of Lemma~\ref{lem:der}]
First, take any $b\in(\vp,b_0)$ and $h\in(0,b-\vp)$, and 
write
\begin{gather}
	\E f(b-h,\xi_{\tau(b-h)})-\E f(b,\xi_{\tau(b)})=\EE_1+\EE_2, \label{eq:EE1+EE2}\\
	\intertext{where}
\begin{aligned}
\EE_1&:=\E f(b-h,Y_{b-h})-\E f(b,Y_{b-h}),\\
\EE_2&:=\E f(b,Y_{b-h})-\E f(b,Y_b)=\E g(Y_{b-h})-\E g(Y_b),
\end{aligned}	\notag\\
Y_b:=\xi_{\tau(b)},\quad g(x):=g_b(x):=f(b,x); \notag
\end{gather}
note that 
\begin{equation}\label{eq:2nd,Y}
\E(Y_{b-h}-Y_b)^2=\tau(b-h)-\tau(b)=O(h)	
\end{equation}
and, using the Jensen inequality (as in the proof of Lemma~\ref{lem:le})
, 
\begin{equation}\label{eq:exp,Y}
	\E e^{\la|Y_{b}|}\le\E e^{\la|Y_{b-h}|}\le\E e^{\la|Y_\vp|}<\infty\quad\text{for all $\la\ge0$.}
\end{equation}

Next, for all real $x$ and $y$, all $b\in(\vp,b_0)$, and all $h\in(0,b-\vp)$ there exists some $\th\in(0,1)$ such that 
\begin{align}
f(b-h,y)-f(b,y)&=-f'_b(b,y)h+\frac{h^2}2 f''_{bb}(b-\th h,y) \notag\\
&=-f'_b(b,x)h+O\big((h|x-y|+h^2)(e^{|x|}+e^{|y|})\big), \label{eq:f increm}
\end{align}
since $|f''_{bb}(b,x)|+|f''_{bx}(b,x)|\le C_f e^{|x|}$ for all $x\in\R$. 
Using \eqref{eq:f increm} with $Y_{b-h}$ and $Y_b$ in place of $y$ and $x$, respectively, and also the C\'auchy-Schwarz inequality together with \eqref{eq:2nd,Y} and \eqref{eq:exp,Y}, one has
\begin{align}
	\EE_1&=-h\E f'_b(b,Y_b)\notag\\
	&\qquad\qquad+O\Big(h\sqrt{\E(Y_{b-h}-Y_b)^2\E(e^{2|Y_b|}+e^{2|Y_{b-h}|})}
	+h^2\E(e^{|Y_b|}+e^{|Y_{b-h}|})\Big) \notag\\
	&=-h\E f'_b(b,Y_b)+O(h^{3/2}). \label{eq:EE1bound}
\end{align}

To estimate $\EE_2$, write 
\begin{equation}\label{eq:EE2ito}
	\EE_2=\int_{-\infty}^\infty\PP(Y_b\in\dd x)\big(\E g(x+\eta+\nu-\E\nu)-g(x)\big),
\end{equation}
where $\eta$ and $\nu$ are independent r.v.'s, $\eta\sim\mathrm{N}\big(0,(1-\vp)\,\De\tau\big)$, $\nu\sim\mathrm{Pois}(\vp\,\De\tau)$, $\De\tau:=\tau(b-h)-\tau(b)=-h\tau'(b)+O(h^2)$. 
Note that
\begin{equation}\label{eq:ests}
\left.
\begin{aligned}
\PP(\nu=0)&=e^{-\vp\,\De\tau}=1-\vp\,\De\tau+O(\De\tau^2)\\
&=1+\vp\tau'(b)h+O(h^2)=1+O(h),\\
\PP(\nu=1)&=e^{-\vp\,\De\tau}\vp\,\De\tau=\vp\,\De\tau+O(\De\tau^2)\\
&=-\vp\tau'(b)h+O(h^2)=O(h),\\
\PP(\nu\ge2)&=O(h^2),\\
\E\nu&=\vp\,\De\tau=-\vp\tau'(b)h+O(h^2),\\
\E\eta^2&=(1-\vp)\,\De\tau=-(1-\vp)\tau'(b)h+O(h^2),\\
\E|\eta|^m&=O(\De\tau^{m/2})=O(h^{m/2})\quad\text{for all } m\in(0,6].
\end{aligned}	
\right\}
\end{equation}
Using some of these estimates together with the conditions $|f''_{xx}(b,x_1)-f''_{xx}(b,x_2)|\le C_f|x_1-x_2|(e^{|x_1|}+e^{|x_2|})$ and $|f''_{xx}(b,x)|+|f'_x(b,x)|\le C_f e^{|x|}$, as well as the C\'auchy-Schwarz inequality (cf.\ \eqref{eq:EE1bound}), one has
\begin{align*}
\E g(x+\eta-\E\nu)-g(x)&=-g'(x)\E\nu+\frac{g''(x)}2(\E\eta^2+\E^2\nu)\\
&\qquad\qquad+O\Big(\E\Big((|\eta|^3+\E^3\nu)(e^{|x+\eta-\E\nu|}+e^{|x|})\Big)\Big)\\
&=\Big(g'(x)\vp-\frac{g''(x)}2(1-\vp)\Big)\tau'(b)h+O(h^{3/2}e^{|x|});\\
\E g(x+\eta+1-\E\nu)-g(x)&=g(x+1)-g(x)\\
&\qquad\qquad+O\Big(\E\Big((|\eta|+\E\nu)(e^{|x+\eta+1-\E\nu|}+e^{|x+1|})\Big)\Big)\\
&=g(x+1)-g(x)+O(h^{1/2}e^{|x|});\\
\E\big(g(x+\eta+\nu-\E\nu)-g(x)\big)^2
&=O\Big(\E(|\eta|+\nu+\E\nu)^2\,\E\big(e^{|x+\eta+\nu-\E\nu|}+e^{|x|}\big)^2\Big)\\
&=O(e^{2|x|}).
\end{align*}
Hence and by \eqref{eq:ests},
\begin{align*}
\E g(x+\eta+\nu-\E\nu)-g(x)&=\PP(\nu=0)\big(\E g(x+\eta-\E\nu)-g(x)\big) \\
&+\PP(\nu=1)\big(\E g(x+\eta+1-\E\nu)-g(x)\big) \\
&+O\big(\PP(\nu\ge2)\,e^{|x|}\big) \\
&=\Big(g'(x)\vp-\frac{g''(x)}2(1-\vp)-\vp\,\big(g(x+1)-g(x)\big)\Big)\tau'(b)h \\
&\qquad\qquad+O(h^{3/2}e^{|x|}). 
\end{align*}
Now, in view of \eqref{eq:EE2ito} and \eqref{eq:exp,Y}, 
\begin{equation*}
	\EE_2=\E\Big(g'(Y_b)\vp-\frac{g''(Y_b)}2(1-\vp)-\vp\,\big(g(Y_b+1)-g(Y_b)\big)\Big)\tau'(b)h+O(h^{3/2}).
\end{equation*}
This, together with \eqref{eq:EE1bound} and \eqref{eq:EE1+EE2}, completes the proof of Lemma~\ref{lem:der}. 
\end{proof} 



\end{document}